\documentclass[12pt,a4paper,reqno]{article}

\usepackage[latin,english]{babel}
\usepackage[utf8]{inputenc}
\usepackage[T1]{fontenc}
\usepackage{eufrak, color, mathrsfs,dsfont}
\usepackage{times}
\usepackage{geometry}
\usepackage{amsmath}
\usepackage{amssymb}
\usepackage{amsthm}
\usepackage{mathtools}
\usepackage{mathabx}
\usepackage{cases}
\usepackage{braket}
\usepackage[toc,page]{appendix}
\usepackage{pdfsync}
\usepackage{hyperref}
\usepackage{psfrag}
\mathtoolsset{showonlyrefs}
\usepackage[mathscr]{euscript}
\usepackage{math}
\numberwithin{equation}{section}

\theoremstyle{plain}
\newtheorem{thm}{Theorem}[section]
\newtheorem{lem}[thm]{Lemma}
\newtheorem{prop}[thm]{Proposition}
\newtheorem{cor}[thm]{Corollary}
\theoremstyle{definition}
\newtheorem{defn}[thm]{Definition}

\theoremstyle{remark}
\newtheorem{rem}[thm]{Remark}

\newcommand{\LieTr}[2]{e^{-#1} #2 e^{#1}}

\DeclareMathOperator{\diag}{diag}

\DeclareMathOperator{\meas}{meas}
\DeclareMathOperator{\spec}{spec}

\DeclareMathOperator{\linspan}{span}

\newcommand{\Op}{{\rm Op}}
\newcommand{\Ops}{{\rm OP}S}

\newcommand{\lip}{{{\rm Lip}}}
\newcommand{\wlip}[1]{{{\rm Lip}(#1)}}

\newcommand{\bsigma}{\boldsymbol{\sigma}}
\renewcommand{\bar}{\overline}
\newcommand{\size}{{\rm size}}

\newcommand{\re}{{\rm Re}}
\newcommand{\me}{{\rm m}}

\newcommand{\pa}{\partial}

\newcommand{\normL}[2]{\| #1 \|^{\wlip{#2}} }

\providecommand{\vect}[2]{{\bigl(\begin{smallmatrix}#1\\#2\end{smallmatrix}\bigr)}}

\title{\bf Reducibility for a linear wave equation with Sobolev smooth fast driven potential}

\author{Luca Franzoi \footnote{NYUAD Research Institute, New York University Abu Dhabi, NYUAD Saadiyat Campus, 129188, Abu Dhabi, UAE .  {\it E-mail:} \texttt{lf2304@nyu.edu}}}
\date{}

\begin{document}

\maketitle

\noindent
{\bf Abstract.}
We prove a reducibility result for a linear wave  equation with a time quasi-periodic driving on the one dimensional torus. The driving is assumed to be fast oscillating, but not necessarily of small size. Provided
that the external frequency vector is sufficiently large and chosen from a Cantor set of large measure,
the original equation is conjugated to a time-independent, block-diagonal one. With the present paper we extend the previous work \cite{FM19} to more general assumptions: we replace the analytic regularity in time with Sobolev one; 
the potential in the Schr\"odinger operator is a non-trivial smooth function instead of the constant one.
 The key tool to achieve the result is a localization property of each eigenfunction of the Schr\"odinger operator close to a subspace of exponentials, with a polynomial decay away from the latter.

\smallskip 


\smallskip


\tableofcontents

\section{Introduction}
We consider on the one-dimensional periodic torus  $x\in\T:=\R / 2\pi\Z $ the linear wave equation 
\begin{equation}
\label{eq:KG}
u_{tt}-u_{xx}+ q(x) u  + v(\omega t,x)u = 0 \,.
\end{equation}
We assume the following conditions:  for a fixed $\nu\geq 1$, the time quasi-periodic potential $v(\vf,x)|_{\vf=\omega t}$ satisfies
\begin{equation}\tag{\bf V}
	\label{assumption V}
	v(\vf,x) \in H^S(\T^\nu\times \T,\R)\,, \quad S > s_0 := [\tfrac{\nu+1}{2}]+2\,, \quad \int_{\T^\nu}v(\vf,x)\wrt\vf=0\,;
\end{equation}
the real-valued function $q(x)\in H^\infty(\T,\R)$ satisfies $\inf\spec(-\pa_{xx}+q(x))>0$ and the Schrödinger operator $L_q := - \partial_{xx} + q(x)$ has a $L^2$-complete orthornormal basis of eigenfunctions  $(\psi_{j})_{j\in\Z}$  with corresponding eigenvalues $(\lambda_j^2)_{j\in\Z}$ for which
\begin{equation}\label{assumption Q}
	\tag{\bf Q}
	(-\partial_{xx} + q(x))\psi_{j}(x) = \mu_j^2 \psi_{j}(x) \ , \quad \mu_j^2 = j^2 + \tq + d(j) >0 \,, \quad j\in\Z\,,
\end{equation}
where $\tq:=\braket{q}_x := \frac{1}{2\pi}\int_\T q(x)\wrt x$ and $(d(j))_{j\in\Z}\in\ell^2(\Z)$. The main feature  is that we are not imposing any assumption on the size of this potential $v(\vf,x)$, but we require  it to be  \emph{fast oscillating}, namely $\abs\omega\gg 1$.

The goal of this paper is to show, for any frequency $\omega$ belonging to a Cantor set of large measure,  the reducibility of the linear system \eqref{eq:KG}. That is, we  construct a change of coordinates which conjugates  equation  \eqref{eq:KG} into a block-diagonal, time independent   one.  In the previous paper \cite{FM19}, we proved the reducibility of the equation \eqref{eq:KG} under stronger assumptions: Dirichlet boundary conditions; constant potential $q(x)=\tm^2>0$; analytic regularity in time for the driving potential $v(\vf,x)$. The extensions of these assumptions to our milder ones, therefore more general and suitable for applications, turned out to be not trivial and new ideas are needed. 

The scheme of the reducibility follows the one developed in \cite{FM19}.  It combines a preliminary transformation,  suitable to fast oscillating systems,  with a KAM normal form reduction to a time independent, block diagonal operator.
We first perform a change of coordinates, following  Abanin et al. \cite{abanin1}, that conjugates \eqref{eq:KG} to an equation with driving of size  $|\omega|^{-1}$, and thus perturbative in size. 
The price to pay is that the new equation 
might not fit in the standard KAM schemes studied in \cite{kuk87}.
The problem is overcome in our model by exploiting the  pseudodifferential properties of the operators involved, showing that the new perturbation features regularizing properties. In particular, in this paper we have to realize the operator $L_q$ and all its real powers as pseudodifferential operators, which is a known, but not trivial fact, 
as we explain in Section \ref{sez.schema}.

The second key ingredient of the proof concerns appropriate {\em balanced} Melnikov conditions (see \eqref{eq:meln_intro}), which allow us to perform a convergent KAM reducibility iteration.  The new contribution that is needed at this step is to prove a localization property in Sobolev regularity for the eigenfunctions $(\psi_{j}(x))_{j\in\Z}$ in \eqref{assumption Q} with respect to the exponential basis $(e^{\im jx})_{j\in\Z}$. In analytic regularity, this fact was proved by Craig \& Wayne in \cite{CrWa}. With this property, operators with smoothing pseudodifferential symbols exhibit off-diagonal decay of the matrix elements also when the latters are computed with respect to the eigenfunction basis in \eqref{assumption Q}.

\medskip

Fast periodically driven systems attract a great interest in physics, both theoretically and experimentally, especially in the study of many-body systems \cite{eisert,jotzu, kitagawa}. 
Modifying a system by periodic driving is referred as ``Floquet engineering''. The interest is to understand the rich behaviour that the dynamics of such models exhibit and to possibly observe novel quantum states of matter.
We refer to the recent review paper \cite{weitenberg} for an extended presentation on the state of the art in this research field, with experimental implementations for  ultracold atoms,   graphenes and  crystals.

The mathematical interest of the present paper is to extend the classical finite dimensional Floquet theory to PDEs. 
The usual setup for dealing with this problem is to treat with small perturbations of a diagonal operator, i.e. of the form  $D + \epsilon V(\omega t)$, where $D$ is diagonal, $\epsilon$ small and $\omega$ avoiding resonances with the spectrum of $D$. In the following we only mention only the most recent developments in this research and we refer to the dissertation \cite{tesi} for a larger overview of the literature.

In the pertubative regime, most of the new results cover the case where $V(\omega t)$ is a time quasi-periodic unbounded operator. In a series of papers \cite{Bam16I,Bam17II,Bam18III}, the reducibility was proved for the 1D quantum harmonic and anharmonic oscillators under quasi-periodic unbounded pseudodifferential perturbation. The reducibility of trasport equations on $\T^d$ was proved in \cite{BLM19} and  \cite{FGMP}, as well for some classes of wave equations on $\T$ \cite{SunLiXie} and $\T^d$ \cite{mont17}. In \cite{FGN} the authors proved the reducibility of the Schr\"odinger equation with pseudodifferential perturbations of order less or equal than $1/2$ on Zoll manifolds.  The reducibility was proved also for the relativistic Schr\"odinger equation on $\T$ with time quasi-periodic unbounded perturbations of order $1/2$ \cite{SunLi}, for a wave equation with time quasi-periodic perturbations of order 1 \cite{Sun22} and for a linear Schr\"odinger equation with an almost periodic unbounded pertubation \cite{MoPro}.
In the context of nonlinear PDEs, recently the existence of KAM reducible tori was proved for quasilinear perturbations the Degasperis-Procesi equation \cite{FGP}, for water waves equations \cite{BBHM,BFM,BFM21,BM}, for semilinear perturbation of the defocusing NLS equation \cite{BKM16} and quasilinear perturbations of the KdV equation \cite{BKM}. We remark that the latter two results actually prove the existence of \emph{large} KAM reducible tori. We also mention the work in \cite{BM20}, where the reducibility was proved for the linearization at time quasi-periodically solutions of forced 3D Euler equations close to constant fields.

When the smallness of the time quasi-periodic perturbations is replaced by the assumption of being fast oscillating, we developed an adapted normal form  in \cite{FM19}, which we called {\em Magnus normal form}, 
following 
 \cite{ abanin2, abanin4, abanin1}  where the classical Magnus expansion \cite{magnus} was generalized.
Their normal form allows to extract a time independent Hamiltonian  that approximates well the dynamics 
quantum many-body systems (spin chains) with a fast periodic driving up to some finite but  very long times \cite{abanin1}.
An important difference between \cite{abanin1} and \cite{FM19} lies in the fact that, while in \cite{abanin1} all the involved operators are bounded, on the contrary our principal operator is an unbounded one. The strategy developed for the Klein-Gordon equation in \cite{FM19} has been applied in \cite{Sun20} to the Schr\"odinger equation, where the fast oscillating, quasi-periodic driving is a smoothing pseudodifferential operator of order strictly less than -1. The question whether the latter result holds with perturbations of order greater than -1, ideally with a general time quasi-periodic multiplicative potential, is still an open problem.

In case of systems of the form $H_0 + V(t)$, where the perturbation $V(t)$ is neither small in size nor fast oscillating, a general reducibility is not known. 
However, in same cases  it is possible to find some results of "almost reducibility"; that is, the original Hamiltonian is conjugated to one of the form $H_0 + Z(t) + R(t)$, where $Z(t)$ commutes with $H_0$, while $R(t)$ is an arbitrary smoothing operator, see e.g. \cite{BGMR2}.
This normal form ensures  upper bounds on the speed of transfer of energy from low to high frequencies; e.g. it implies that 
the  Sobolev norms of each  solution grows at most as $t^\epsilon$ when $t \to \infty$, for any arbitrary small $\epsilon >0$.
This procedure (or close variants of it), has been applied also in   \cite{MaPro,MaRo, Mon17} and recently by \cite{BLM22, mont.sublin}. 
There are also  examples in \cite{bou99, ma18} where the authors  
engineer periodic drivings  aimed to  transfer  energy from low to high frequencies and leading to  unbounded growth of Sobolev norms (see also Remark \ref{bou} below).

Finally, we want to mention also the papers 
\cite{BB08, CG17, GengXue}, where KAM techniques are applied to construct quasi-periodic solutions with
$|\omega| \gg 1$. 
In   \cite{BB08} this is shown for 
a nonlinear wave equation with Dirichlet boundary conditions, while in \cite{GengXue} the quasi-periodic solutions are constructed for the two dimensional NLS with large forced nonlinearity. However, in both cases  reducibility is not obtained. 
In 
\cite{CG17}, KAM techniques are applied to a many-body  system with fast driving:
the authors construct a periodic orbit with large frequency and prove its asymptotic stability.

\subsection{Main result}

To state precisely our main result, equation \eqref{eq:KG} has to be rewritten as a linear Hamiltonian system. 
We introduce the new variables  $\phi := B^{1/2}u  + \im B^{-1/2} \partial_t u $ and $\bar \phi := B^{1/2}u  - \im B^{-1/2} \partial_t u$,
where
\begin{equation}
\label{def:B}
B:=\sqrt{-\pa_{xx} +q(x)} \,.
\end{equation}
Note that   the operator $B$ and all its positive powers are invertible by standard functional calculus and the spectral assumption $\inf\spec (-\pa_{xx}+q(x))>0$. In the new variables,
equation \eqref{eq:KG} is equivalent to
\begin{equation}\label{eq:KG_c}
\im\phi_t  = B\phi+\frac{1}{2}\,B^{-1/2}V(\omega t)B^{-1/2}(\phi+\bar\phi) \  . 
\end{equation}
Taking \eqref{eq:KG_c} coupled with its complex conjugate, we obtain the following system
\begin{equation}\label{eq:KG_matrix}
\im \phi_t = \bH(t)\phi \ , \quad \bH(t):=\left(\begin{matrix}
B & 0 \\ 0 & -B
\end{matrix}\right) + \frac{1}{2}\,B^{-1/2}V(\omega t )B^{-1/2}\left(\begin{matrix}
1 & 1 \\ -1 & -1
\end{matrix}\right) \ , 
\end{equation}
where, abusing notation, we denoted $\phi(t,x)=\vect{\phi(t,x)}{\bar{\phi(t,x)}}$ the vector with components  $\phi, \bar \phi$.
The linear system \eqref{eq:KG_matrix} is defined on the scale of Sobolev spaces $(\cH^s)_{s\in\R}$, where $\cH^s:= H^s(\T^{\nu+1})\times H^s(\T^{\nu+1}) $ and the scalar Sobolev norm is given by
\begin{equation}\label{eq:sob}
\begin{footnotesize}
	\begin{aligned}
	 u(\vf,x)=\sum_{(\ell,j)\in\Z^{\nu+1} } \whu(\ell,j) e^{\im (\ell\cdot\vf +jx)} \ \mapsto \  \| u \|_{s}^2 :=\sum_{(\ell,j)\in \Z^{\nu+1}}\braket{\ell,j}^{2s}|\whu(\ell,j)|^2<\infty  \,.
	\end{aligned}
\end{footnotesize}
\end{equation}
We use the notation $\braket{\ell,j}:=\max\{1,|\ell|,|j|\}$, which will be kept throughout all the paper.	
We define the $\nu$-dimensional annulus of size $\tM>0$ by
$$R_{\tM}:=\bar{B_{2\tM}(0)}\backslash B_{\tM}(0)\subset \R^{\nu} \,,$$
where $B_M(0)$ denotes the ball of center zero and radius $M>0$ in the Euclidean topology of $\R^\nu$.
\begin{thm}
	\label{thm:main}
	Let $q(x)\in H^{\infty}(\T,\R)$ and $v(\vf,x)\in H^{S}(\T^{\nu+1},\R)$, assuming  \eqref{assumption V}, \eqref{assumption Q}.
	Fix an arbitrary $\gamma_*>0$ sufficiently small and $\alpha\in (0,1)$.
	Then there exist $\tM_* >1 $, $\sigma_*>0$, $C >0$ such that,
	for any $\tM\geq \tM_*$, there exists a subset $\Omega_{\infty}^{\alpha}=\Omega_{\infty}^{\alpha}(\tM,\gamma_*)$ of $R_\tM$ of large measure relatively to $R_\tM$, namely
	\begin{equation}\label{measure.state}
	\me_r (\Omega_{\infty}^{\alpha}):=\frac{\meas(R_{\tM}\backslash\Omega_\infty^\alpha)}{\meas(R_{\tM})} \leq  C \gamma_* , 
	\end{equation}
	such that the following holds true.	 
	For any frequency vector $\omega\in\Omega_\infty^\alpha$ and any $S\geq s_0 +\sigma_*$, there exists a time quasi-periodic  operator $\cT(\omega;\omega t)$, bounded in $\cL(\cH^r)$, with $r\in[0,s_0]$, such that the change of coordinates $\phi= \cT(\omega;\omega t)\psi$  conjugates \eqref{eq:KG_matrix} to the block-diagonal time-independent system
	\begin{equation}\label{eq:eff_sys}
			\begin{aligned}
				& \im \dot{\psi}(t)=\bH^{\infty,\alpha}\psi(t) \,, \\
				 & \bH^{\infty,\alpha}=\bH^{\infty,\alpha}(\omega,\alpha)=\diag\Big\{\!\left.H_0^{\infty} \right._{[n]}^{[n]}(\omega,\tM,\alpha) \,, \ n\in\N_0  \Big\} \begin{pmatrix}
					1 & 0 \\ 0 & -1
				\end{pmatrix} \,,
			\end{aligned}
	\end{equation}
	with $\spec \big( \!\left.H_0^{\infty} \right._{[n]}^{[n]}(\omega,\tM,\alpha) \big) = \big\{ \lambda_{n,+}^{\infty}(\omega, \tM,\alpha),\,\lambda_{n,-}^{\infty}(\omega, \tM,\alpha) \big\}$ for any $n\in\N_0 $ and the block matrix representation in \eqref{blocks.not}.
	\\[1mm]
	\noindent The transformation $\cT(\omega;\omega t)$ is close to the identity, in the sense that, for any $r\in[0,s_0]$, there exists $C_r>0$ independent of $\tM$ such that
		\begin{equation}\label{close.id.state}
		\norm{\cT(\omega)-{\rm Id}}_{\cL(\cH^r)} \leq \frac{C_r}{ \tM^{\frac{1-\alpha}{2}}} \,.
	\end{equation}
	The eigenvalues of the final blocks $(\lambda_{n,\pm}^{\infty}(\omega))_{n\in\N_0}$ are real, Lipschitz in $\omega$,  and admit the following asymptotics for $n\in\N_0$:
	\begin{equation}\label{eq:eff_eig}
	\lambda_{n,\pm}^{\infty}(\omega) = \lambda_{n,\pm}^{\infty}(\omega,\alpha) = \lambda_n + \varepsilon_{n,\pm}^{\infty}(\omega,\alpha) \ , \quad \ \varepsilon_{n,\pm}^{\infty}(\omega,\alpha)\sim O\left( \frac{1}{\tM j^{\alpha}} \right)  \ ,
	\end{equation}
	where $\lambda_n = \sqrt{n^2 + \tq+ d(n)}$ are the eigenvalues of the operator $B:=\sqrt{-\partial_{xx}+q(x)}$.
\end{thm}
\begin{rem}
	The parameter $\alpha$, which one chooses and fixes in the real interval $(0,1)$, influences the asymptotic expansion of the final  eigenvalues,  as one can read from \eqref{eq:eff_eig}. Also the construction of the set of the admissible frequency vectors heavily depends on this parameter.
\end{rem}
\begin{rem}
	The assumption on the zero average in $\vf$ for the potential $v(\vf,x)$ in \eqref{assumption V} has no loss of generality at all. Indeed, in case of $\braket{v}_\vf:= \frac{1}{(2\pi)^\nu}\int_{\T^\nu}v(\vf,x)\wrt\vf \neq0$, one simply replaces the function $q(x)$ with $q_1(x):=q(x)+\braket{v}_\vf(x)$, asking only that the new potential $q_1(x)$ satisfies the spectral assumption $\inf\spec(-\pa_{xx}+q_1(x))>0$.
\end{rem}
\begin{rem}
	In Theorem \ref{thm:main} the spectral condition $\inf\spec(-\pa_{xx}+q(x))>0$ may be replaced by asking that the spectrum is just bounded from below. For technical reasons in the construction of the pseudodifferential calculus in Appendix \ref{app:funct.calc}, we have assumed that this bound is strictly positive without any loss of generality.
\end{rem}

Theorem \ref{thm:main} will be proved at the end of the paper, in Section \ref{subsec:proof.main}.
Let  us denote by $\cU_{ \omega}( t,\tau)$ the propagator generated by \eqref{eq:KG_matrix}  such that $\cU_{ \omega}(\tau, \tau)={\rm Id}$ for any $\tau \in \R$.
An immediate consequence of Theorem \ref{thm:main} is that we have a Floquet   decomposition:
\begin{equation}\label{decom}
\cU_{\omega}( t,\tau) =  \cT(\omega; \omega t)^* \circ  {\rm e}^{-\im(t-\tau)\bH^{\infty, \alpha}} \circ  \cT(\omega; \omega \tau) \ .
\end{equation}

Another consequence of \eqref{decom} is  that, for any $r \in [0,s_0]$, the norm $\norm{\cU_{\omega}(t,0) \vf_0}_{\cH_x^r}$, with $\cH_x^s:= H^s(\T)\times H^s(\T) $, is uniformly bounded.
\begin{cor}
	\label{cor.1}
	Let $\tM \geq \tM_*$ and $\omega \in \Omega_\infty^\alpha$. 
	For any $r \in [0,s_0]$ one has
	\begin{equation}\label{unest}
	c_{r}\norm{\vf_0}_{\cH_x^r} \leq \sup_{t\in\R}\norm{ \cU_{\omega}( t,0)\vf_0}_{\cH_x^r}
	\leq C_{r}\norm{\vf_0}_{\cH_x^r} \,,\quad \forall \,
	\vf_0\in \cH_x^r\,, 
	\end{equation}
	for some  $ c_{r}>0, C_{r}>0$.
	In particular,  there exists a constant $c'_{r}>0$ such that, if the initial data $\vf_0 \in \cH_x^{r}$, then
	\begin{equation*}
	\left(1- \frac{c'_r}{ \tM^{\frac{1-\alpha}{2}}}\right)\norm{\vf_0}_{\cH_x^r}   \leq \sup_{t\in\R}\norm{\cU_{\omega}( t,0)\vf_0}_{\cH_x^r} \leq 
	\left(1+ \frac{c'_r}{ \tM^{\frac{1-\alpha}{2}}} \right)
	\norm{\vf_0}_{\cH_x^r}\,.
	\end{equation*}
\end{cor}

\begin{rem}
	\label{bou}
	Corollary \ref{cor.1} shows that, if the frequency $\omega$ is chosen in the Cantor set $\Omega^\alpha_\infty$,  no phenomenon of  growth of Sobolev norms can happen. 
	On the contrary, if $\omega$ is chosen  resonant, one can construct drivings which provoke  norm explosion with exponential rate, see  \cite{bou99} (see also \cite{ma18} for other examples).
\end{rem}

\subsection{Scheme of the proof}\label{sez.schema}

The proof of Theorem \ref{thm:main} is built on the same scheme used in \cite{FM19}. We now recall the main steps of the proof and comment on the new contributions.
\\[1mm]
\noindent{ \bf The Magnus normal form.} In Section \ref{sec:magnus} we perform what we refer to Magnus normal form, namely a preliminary transformation, adapted to fast oscillating systems, that moves the non-perturbative equation \eqref{eq:KG_matrix} into a pertubative one where the size of the transformed quasi-periodic potential is as small as large is the module of the frequency vector. Sketchily, we perform a change of coordinates  which conjugates 
\begin{equation}
\label{sss}
\left\{ \begin{matrix}
\bH(t)=\bH_0+\bW(\omega t) \\
"\size(\bW) \sim 1"
\end{matrix} \right. \quad  \rightsquigarrow \quad  \left\{ \begin{matrix}
\wt\bH(t)=\bH_0+\bV(\omega;\omega t) \\
"\size(\bV) \sim |\omega|^{-1}"
\end{matrix} \right. \ .
\end{equation}
Note that  $\bH_0$ is the same on both sides of \eqref{sss} provided $\int_{\T^\nu}\bW(\vf) \di \vf = 0$, which is fulfilled in our case thanks to \eqref{assumption V}.
In principle, the new perturbation  may not be  sufficiently regularizing to 
fit in a standard KAM scheme. 
As in \cite{FM19}, we employ  pseudodifferential calculus, thanks to which we control the order (as a pseudodifferential operator) of the new perturbation, and prove that it is actually enough regular for the KAM iteration. 
This is true because the principal term of the new perturbation is a commutator with $\bH_0$   (see equation \eqref{eq:magnus_4}), and 
one can exploit the smoothing properties of the commutator of  pseudodifferential operators.

What is new in the present case is that it is not immediately clear whether the operator $B:=\sqrt{-\pa_{xx}+q(x)}$ and all its (real) powers belong to classes of pseudodifferential operators or not. For instance, one may guess that a pseudodifferentlal symbol for the operator $B$ is simply $b(x,\xi):=\sqrt{\xi^2 +q(x)}$, but a direct computation shows that the composition  $\Op(b)\circ \Op(b) = \Op(b \# b)$ is not equal to $\Op(b^2)=L_q$, see \eqref{compo_symb}. Luckily for us, the answer to the previous question is positive. We state this property in Theorem \ref{cor:pseudo_sqrt} and the proof is briefly presented in Appendix \ref{app:funct.calc}, recalling the construction due of Seeley \cite{seeley} and Shubin \cite{shubin}. The idea is first to realize any real power of a given self-adjoint operators in terms of the functional calculus in such a way that standard group properties are preserved (see Theorem \ref{thm:funct_seeley}). Then, using the parametrix symbol of the resolvent operator, it follows that such powers actually fit in the pseudodifferential calculus (see Theorem \ref{thm:structure_seeley}).
\\[1mm]
\noindent {\bf Craig-Wayne Lemma in Sobolev regularity.} 
After the Magnus normal form, the next is the KAM reducibility scheme in order to remove the time dependence on the coefficients of the equation. Here a new problem arises: indeed, the most natural way to proceed is to work on the (infinite dimensional) matrix representations of the operator with respect to the eigenfunction basis \eqref{assumption Q} of the self-adjoint operator $L_q:=-\pa_{xx}+q(x)$, whereas the quantization of the pseudodifferential symbols is constructed on the exponential basis. In \cite{FM19} this problem is not present because we had $q(x):=\tm^2$ and the eigenfunctions are the trigonometric sine functions. To solve this issue, in Section \ref{sec:embe} we prove the fact that each eigenfunctions $\psi_{j}(x)$ in \eqref{assumption Q} is essentially localized in the subspace spanned by the exponentials $\{e^{\im jx},e^{-\im jx}\}$. Assuming $q(x)$ real analytic, Craig \& Wayne \cite{CrWa} proved the following exponential decay
\begin{equation*}
	| (\psi_{j},e^{\im j'x})_{L^2(\T)} | \lesssim e^{-\sigma | j - |j'| |} \,, \quad j,j'\in \Z\,,
\end{equation*}
where $\sigma>0$ here denotes the radius of analyticity. In our case we have $q(x)$ only infinitely smooth and therefore their result cannot apply. What we can show, instead, is the polynomial decay
\begin{equation*}
	| (\psi_{j},e^{\im j'x})_{L^2(\T)} | \lesssim \braket{j-|j'|}^{-s} \,, \quad j,j'\in \Z\,,
\end{equation*}
for some Sobolev regularity $s>0$. This will be proved in Theorem \ref{thm:decay_eigen_smooth}.  The proof is based on a Lyapunov-Schmidt reduction with respect to the subspace ${\rm span} \{e^{-\im jx}, e^{\im jx}\}$, following \cite{Posch11}.  We use these bounds to convert the bounded and smoothing pseudodifferential perturbations provided by the Magnus normal form into classes of matrix representations with off-diagonal decay, suitable for the KAM reducibility scheme, see Theorem \ref{thm:embed_PSDO_decay}. 
The price to pay is a loss of regularity coming from the change of the basis, which, however, will affect the KAM reducibility scheme only for the estimates for its initial step.
\\[1mm]
\noindent {\bf The KAM reducibility and the balanced Melnikov conditions.} 
We are finally ready to perform the KAM reducibility scheme. This step is nowadays quite standard and it is presented in Section \ref{sec:kam}. The difference with \cite{FM19} is that we previously considered Dirichlet boundary condition on the compact interval $[0,\pi]$, whereas here we work with periodic boundary conditions. The consequence is that we cannot achieve a full diagonalization of the Hamiltonian due to the multiplicity of the eigenvalues $\lambda_j=\lambda_{-j}$ in \eqref{assumption Q} for $j\neq 0$, but only the reducibility to the block diagonal operator in \eqref{eq:eff_sys}. This reduction has been achieved for several different equations: in this paper we decided to follow in some parts \cite{Mont.Kirch17}.

Also in this case, second order Melnikov conditions  on the unperturbed eigenvalues $\lambda_j = \sqrt{j^2 + \tq+d(j)}$ are needed, namely lower bounds on the small denominators $| \omega\cdot \ell + \lambda_j \pm \lambda_{j'}|$ when they do not identically vanish. 
The same issue that has been encountered in \cite{FM19} of the interplay between the size of new perturbation  $\sim |\omega|^{-1}$, the size of the small denominators $\sim |\omega|$ and the smoothing properties of the perturbation.
To overcome the problem, we  impose  {\em balanced} Melnikov conditions, in which   we   balance the loss in size (in the denominator) and gain in regularity (in the numerator).
More precisely, we show that  for any $\alpha \in [0,1]$ one can impose 
\begin{equation}\label{eq:meln_intro}
	\begin{aligned}
		| \omega\cdot \ell + \lambda_j \pm \lambda_{j'}| \geq \frac{\gamma}{\braket{\ell}^\tau}\frac{\braket{|j|\pm |j'|}^\alpha}{|\omega|^\alpha} \,,
	\end{aligned}
\end{equation}
for any  $(\ell, j, j') \in  \Z^\nu\times \Z \times \Z $,  $(k,|j|,|j'|) \neq (0,|j|, |j|)$, 
for a set of parameters $\omega$ in $R_\tM$ of large relative measure. 
Note that the choice of $	\alpha$ will influence the regularizing effect given by $\la j\pm l \ra^\alpha$ in the right-hand side of \eqref{eq:meln_intro}; ultimately, this  modifies the asymptotic expansion of the final eigenvalues, as one can see in 
\eqref{eq:eff_eig}. The non-resonance condition \eqref{eq:meln_intro} will be extended in Section \ref{sec:melnikov} to the eigenvalues of the final blocks in the normal form \eqref{eq:eff_sys}, which will be proved to hold on a set of relative large measure with respect to $R_{\tM}$ in Theorem \ref{lem:meas_infty}.

\medskip

\noindent {\bf Acknowledgments.} The author would like to thank Massimiliano Berti and Alberto Maspero for the fruitful discussion when this work started. The work of the author is supported by Tamkeen under the NYU Abu Dhabi Research Institute grant CG002.

\section{Functional settings}\label{sec:fun}
Given  a set  $\Omega \subset \R^\nu$ and a Fréchet space $\cF$, the latter endowed with a system of seminorms $\{\|\,\cdot\,\|_n \,:\, \ n\in\N_0\}$, we define for a function $f: \Omega \ni \omega \mapsto f(\omega) \in \cF$ 
the quantities 
\begin{equation}
\label{lip}
\abs{f}_{n, \Omega}^\infty:=
\sup_{\omega \in \Omega} \norm{f(\omega)}_n
\ , \qquad
\abs{f}_{n,\Omega}^{\lip} :=
\sup_{\omega_1, \omega_2 \in \Omega \atop \omega_1 \neq \omega_2 }
\frac{\norm{f(\omega_1)- f(\omega_2)}_n}{\abs{\omega_1 - \omega_2}} .
\end{equation}
Given 
$\tw \in \R_+$,  we 
denote by $\lip_\tw(\Omega, \cF)$ the space of functions from $\Omega$ into $\cF$  such that 
\begin{equation}\label{eq:wlip_def}
\norm f_{n}^{\wlip{\tw}} = \norm f_{n,\Omega}^{\wlip{\tw}} := \abs f_{n,\Omega}^\infty + \tw \abs f _{n-1,\Omega}^{\lip} < \infty \ .
\end{equation}

\subsection{Pseudodifferential calculus}\label{sub:pseudo}
The Magnus transform in Section \ref{sec:magnus} is based on the calculus with pseudodifferential operators acting on the scale of the Sobolev spaces $H^s(\T^{\nu+1})$, $s\geq s_0$, as defined in \eqref{eq:sob}.
In this section we report fundamental notions of pseudodifferential calculus,
following \cite{BFM,BM}. 
\begin{defn}{\bf ($\Psi$DO)} \label{def:symbol}
	A  \emph{pseudodifferential} symbol $ a (x,j) $
	of order $m$ is 
	the restriction to $ \R \times \Z $ of a function $ a (x, \xi ) $ which is $ \cC^\infty $-smooth on $ \R \times \R $,
	$ 2 \pi $-periodic in $ x $, and satisfies 
	$
	\sup_{x\in\R}| \pa_x^\alpha \pa_\xi^\beta a (x,\xi ) | \leq C_{\alpha,\beta} \langle \xi \rangle^{m - \beta} $ for any $\xi\in\R$  and any $ \alpha, \beta \in \N_0
	$.
	We denote by $ S^m $ 
	the class of  symbols  of order $ m $ and 
	$ S^{-\infty} := \cap_{m \geq 0} S^m $. 
	To a  symbol $ a(x, \xi ) $ in $S^m$ we associate its quantization acting on a $ 2 \pi $-periodic function 
	$ u(x) = \sum_{j \in \Z} \whu(j)\, e^{\im j x} $
	as 
	$$
	[\Op(a)u](x) : = \sum_{j \in \Z}  a(x, j ) \whu(j) \, e^{\im j x} \, . 
	$$
	We denote by  $ \Ops^m $ 
	the  set of pseudodifferential  operators of order $ m $ and
	$ \Ops^{-\infty} := \bigcap_{m \in \R} \Ops^{m} $.
\end{defn}

When the symbol $ a (x) $ is independent of $ \xi $, the operator $ \Op (a) $ is 
the multiplication operator by the function $ a(x)$, i.e.\  $ \Op (a) : u (x) \mapsto a ( x) u(x )$. 
In such a case we also denote $ \Op (a)  = a (x) $.

Along the paper we consider families of pseudodifferential operators 
with a symbol $  a(\omega;\vf,x,\xi)  $
which is Lipschitz continuous with respect to a parameter $\omega\in\R^\nu$ in a open subset $\t\Omega\subset R_\tM$.

\begin{defn}
	{\bf (Weighted $\Psi DO$ norm)} \label{defn.pseudo.norm}
	Let $ A(\omega) := a(\omega; \vf, x, D) \in \Ops^m $ 
	be a family of pseudodifferential operators with symbol $ a(\omega; \vf, x, \xi) \in S^m $, $ m \in \R $, which are 
	Lipschitz continuous with respect to $ \omega\in\t\Omega\subset R_\tM $. 
	For $ \gamma \in (0,1) $, $ \alpha \in \N_0 $, $ s \geq 0 $, we define  
	$$
	\normL{A}{\tw}_{m, s, \alpha} := 
	\sup_{\omega\in\t\Omega}\norm{A(\omega)}_{m, s, \alpha} + \tw \sup_{\omega_1, \omega_2 \in \t\Omega \atop \omega_1 \neq \omega_2 } \frac{\| A(\omega_1)-A(\omega_2)\|_{m,s-1,\alpha} }{|\omega_1-\omega_2|} 
	$$
	where 
	$
	\norm{A(\omega)}_{m, s, \alpha} := 
	\max_{0 \leq \beta  \leq \alpha} \, \sup_{\xi \in \R} \|  \partial_\xi^\beta 
	a(\omega; \cdot, \cdot, \xi )  \|_{s} \  \langle \xi \rangle^{-m + \beta} $. 
\end{defn}

Given a function $a(\omega; \vf, x) \in \cC^\infty$ which is Lipschitz continuous with respect to $\omega$, the weighted norm of the corresponding multiplication operator is $	\normL{\Op(a)}{\tw}_{0,s,\alpha} = \normL{a}{\tw}_{s}  $ for any $\alpha \in \N_0$.

\paragraph{Composition of pseudodifferential operators.}
If $ {\rm Op}(a) $, ${\rm Op}(b) 
$ are pseudodifferential operators with symbols $a\in S^m$, $b\in S^{m'}$, 
$m,m'\in\R$,  
then the composition operator 
$ {\rm Op}(a) {\rm Op}(b) $ 
is a pseudodifferential operator  $ {\rm Op}(a\# b) $ with symbol $a\# b\in S^{m+m'}$. 
It 
admits the  asymptotic expansion: for any $N\geq 1$
\begin{align}\label{compo_symb}
	\begin{footnotesize}
		\begin{aligned}
			(a\# b)(\omega;\vf,x,\xi) & = \sum_{\beta= 0}^{N-1} \frac{1}{\im^\beta \beta!} \pa_\xi^\beta a(\omega;\vf,x,\xi) \pa_x^\beta b(\omega;\vf,x,\xi) + (r_N(a,b))(\omega;\vf,x,\xi) \,,
		\end{aligned}
	\end{footnotesize}
\end{align}
where  $ r_N(a,b) \in S^{m+m'-N} $. 
The following result is proved in Lemma 2.13 in \cite{BM}.

\begin{lem}{\bf (Composition)} \label{pseudo_compo}
	Let $ A = a(\omega; \vf, x, D) $, $ B = b(\omega; \vf, x, D) $ be pseudodifferential operators
	with symbols $ a (\omega;\vf, x, \xi) \in S^m $, $ b (\omega; \vf, x, \xi ) \in S^{m'} $, $ m , m' \in \R $. Then $ A \circ B \in \Ops^{m + m'} $
	satisfies,   for any $ \alpha \in \N_0 $, $ s \geq s_0 $, 
	\begin{equation}\label{eq:est_tame_comp}
		\begin{split}
			\normL{A B}{\tw}_{m + m', s, \alpha}&
			\lesssim_{m,  \alpha} C(s) \normL{ A }{\tw}_{m, s, \alpha}
			\normL{ B}{\tw}_{m', s_0 + |m|+\alpha, \alpha}  \\
			& \ \quad \quad + C(s_0) \normL{A }{\tw}_{m, s_0, \alpha}
			\normL{ B }{\tw}_{m', s + |m|+\alpha, \alpha} \, .
		\end{split}
	\end{equation}
	Moreover, for any integer $ N \geq 1  $,  
	the remainder $ R_N := {\rm Op}(r_N) $ in \eqref{compo_symb} satisfies
	\begin{equation}
		\begin{aligned}
			\normL{\Op(r_N(a,b))}{\tw}_{m+ m'- N, s, \alpha}
			&\lesssim_{m, N,  \alpha} 
			C(s) \normL{ A}{\tw}_{m, s, N + \alpha}
			\normL{ B }{\tw}_{m', s_0 + \abs{m} +  2N  + \alpha,N+\alpha }
			\\
			& \ \quad \quad  +  C(s_0)\normL{A}{\tw}_{m, s_0   , N + \alpha}
			\normL{ B}{\tw}_{m', s +|m| + 2 N  + \alpha, N+ \alpha }.
			\label{eq:rem_comp_tame} 
		\end{aligned}
	\end{equation}
	Both  \eqref{eq:est_tame_comp}-\eqref{eq:rem_comp_tame} hold  
	with the constant $ C(s_0) $ 
	interchanged with $ C(s) $. \\
\end{lem}

The commutator between two pseudodifferential operators $ \Op(a)\in \Ops^m$ and $\Op(b)\in\Ops^{m'}$ is a pseudodifferential operator
in $ \Ops^{m+m'-1}$ with symbol $a\star b\in S^{m+m'-1}$, 
namely $ \left[ \Op(a), \Op(b)\right] = \Op\left( a\star b \right)$, 
that admits, by \eqref{compo_symb},  the expansion
\begin{equation}\label{eq:moyal_exp}
	\begin{aligned}
		& a\star b= -\im\left\{ a,b \right\} + \wt{r_2}(a,b) \,, \quad \wt{r_2}(a,b):=r_2(a,b)-r_2(b,a)\in S^{m+m'-2} \,, \\
		& 
		{\rm where} \quad  \{ a,b \}:= \pa_\xi a \pa_x b - \pa_x a \pa_\xi b 
	\end{aligned}
\end{equation}
is the Poisson bracket between $a(x,\xi)$ and $b(x,\xi)$. 
As a corollary of  Lemma \ref{pseudo_compo} we have the following result,  which is proved in Lemma 2.15 in \cite{BM}. 
\begin{lem}{\bf (Commutator)} \label{pseudo_commu}
	Let $A = {\rm Op}(a) $ and $B = {\rm Op} (b) $ be pseudodifferential operators with symbols $a(\omega;\vf,x,\xi)\in S^{m}$, $b(\omega;\vf,x,\xi)\in S^{m'}$, $m,m'\in \R$. Then the commutator $[A,B]:=AB-BA\in \Ops^{m+m'-1}$ satisfies
	\begin{equation}\label{eq:comm_tame_AB}
		\begin{aligned}
			\normL{[A,B]}{\tw}_{m+m'-1,s,\alpha} & \lesssim_{m, m', \alpha} C(s)\normL{A }{\tw}_{m,s+|m'|+\alpha+2,\alpha+1}\normL{ B }{\tw}_{m',s_0+|m|+\alpha+2,\alpha+1} \\
			& \qquad \quad \ + C(s_0)\normL{A }{\tw}_{m,s_0+|m'|+\alpha+2,\alpha+1} \normL{ B }{\tw}_{m',s+|m|+\alpha+2,\alpha+1} \,.
		\end{aligned}
	\end{equation}
\end{lem}

%
%
%

The following result says that the operator $B=\sqrt{-\pa_{xx}+q(x)}$ is also a pseudodifferential operator.
\begin{thm}{\bf (Powers of $L_q$)}\label{cor:pseudo_sqrt}
	It holds that $B:=L_q^{1/2}\in\Ops^1$ and $B^\mu\in\Ops^\mu$ for any $\mu\in\R$.
\end{thm}
The proof of this theorem is provided in Appendix \ref{app:funct.calc}.

\subsection{Matrix representation and operator matrices}
\label{sec:mo}
For the KAM reducibility, a second and wider class of operators without a pseudodifferential structure is needed on the scale of Hilbert spaces $\left( H^{r} :=  H^r(\T^{\nu+1}) \right)_{r\in\R}$, as defined as in \eqref{eq:sob}. Moreover, let $H^\infty:=\bigcap_{r\in\R}H^r$ and $H^{-\infty}:=\bigcup_{r\in\R}H^r$.\\
If $A=A(\vf)$ is a linear operator, we denote by $A^*$ the adjoint of $A$ with respect to the scalar product of $H^0(\T^{\nu+1})=L^2(\T^{\nu+1})$,
whereas we denote by $\bar A$ the conjugate operator:
$\bar A \psi := \bar{A\bar \psi}$, for any $ \psi \in D(A)$.
\\[2mm]
\noindent{\bf Block representation of operators.}
In the following we partially follow \cite{Mont.Kirch17}.
Consider a family of $\vf$-dependent linear operator $A=A(\vf) : H^\infty(\T^{\nu+1}) \to H^{-\infty}(\T^{\nu+1})$ acting on scalar functions
\begin{equation}\label{u.expan}
	u(\vf,x) = \sum_{j'\in\Z} u^{j'}(\vf)\psi_{j'}(x)= \sum_{\ell'\in\Z^\nu \atop j'\in\Z  }u^{\ell',j'} e^{\im \ell'\cdot\vf}\psi_{j'}(x)
\end{equation}
as
\begin{equation}\label{Au.expan}
\begin{aligned}
	A(\vf)u(\vf,x)& = \sum_{j,j'\in\Z}A_j^{j'}(\vf) u^{j'} (\vf)\psi_{j}(x)   =  \sum_{\ell,\ell'\in\Z^\nu\atop  j,j'\in\Z  }A_{j}^{j'}(\ell-\ell') u^{\ell',j'}e^{\im\ell\cdot\vf} \psi_{j}(x)\,,
\end{aligned}
\end{equation}
where $A_j^{j'}:=(A\psi_{j'},\psi_{j})_{L^2}$.
We shall identify the linear operator $A(\vf)$ with the scalar valued matrix $(A_j^{j'}(\vf))_{j,j'\in\Z}=(A_{j}^{j'}(\ell-\ell'))_{ \ell,\ell'\in\Z^\nu,\atop j,j'\in\Z}$. We partition $\Z$ as $\Z = \bigsqcup_{n\in\N_0} [n] $, where  $[n]:=\{-n,n\}$  for $n\neq 0$  and $ [0]:=\{0\}$.
Therefore, we further  identify the operator $A$ with the tensor valued matrix $A=(A_{[n]}^{[n']}(\ell-\ell'))_{\ell,\ell'\in\Z^\nu, \atop n,n'\in \N_0}$, where, for any $n,n'\in\N_0$ and $\ell\in\Z^\nu$, we define $A_{[n]}^{[n']}(\ell)$ as
 \begin{equation}\label{blocks.not}
 	\begin{aligned}
 		& A_{[n]}^{[n']}(\ell):=  \begin{pmatrix}
 			A_{-n}^{-n'} (\ell)& A_{n}^{-n'}(\ell) \\
 			A_{-n}^{n'}(\ell) & A_n^{n'} (\ell)
 		\end{pmatrix}   \in \C^{2\times 2}\,, \ n,n'\neq 0 \,, \\
 		& A_{[0]}^{[n]}(\ell):= \begin{pmatrix}
 			A_{0}^{-n}(\ell) \\ A_{0}^{n}(\ell)
 		\end{pmatrix} =  (A_{[n]}^{[0]}(\ell))^T \in \C^{2\times 1} \,, \ n\neq 0 \,, \quad A_{[0]}^{[0]}(\ell):= A_0^0(\ell)\in\C\,.
 	\end{aligned}
 \end{equation}
Each $\#[n]\times\#[n']$ matrix $A_{[n]}^{[n']}(\ell)$ may be identified with a linear operator in $\cL(\fE_{n'},\fE_{n})$, where, for $n\in\N_0$,
\begin{equation}
	\fE_0:={\rm span}\{\psi_{0}(x)\}\,, \quad \fE_{n};={\rm span}\{\psi_{-n}(x),\psi_{n}(x)\}\,, \ n\geq1\,.
\end{equation}
Note that each finite dimensional space $\fE_{n}$ is an eigenspace for the operator $L_q$ with eigenvalue $\lambda_{n}$, see \eqref{assumption Q}, and that any linear operator $T\in \cL(\fE_{n'},\fE_{n})$ may be identified with a tensor $(T_{j}^{j'})_{j\in[n],\,j'\in[n']}$, with action given by
\begin{equation}
u(x)=\sum_{j'\in[n']}u^{j'}\,\psi_{j'}(x) \in\fE_{n'} \ \mapsto	\ Tu(x) = \sum_{j\in[n],\, j'\in[n']} T_{j}^{j'} u^{j'} \,\psi_{j}(x)\in \fE_{n}\,, \quad  \,.
\end{equation}
If $n=n'$, then we simply write $\cL(\fE_{n}):=\cL(\fE_{n},\fE_{n})$ and we denote by $\uno_{[n]}:=\uno_{\#[n]}$ the identity operator on $\fE_{n}$. By \eqref{u.expan}, \eqref{Au.expan}, \eqref{blocks.not},
the action of a linear operator $A$  on functions
\begin{equation}
	u(\vf,x)= \sum_{n'\in\N_0}\psi_{[n']}(x) u^{[n']}(\vf)  =\sum_{\ell'\in\Z^\nu \atop n'\in\N_0}\psi_{[n']}(x)u^{\ell',[n']} e^{\im\ell'\cdot\vf} \,,
\end{equation}
reads with respect to the block representation as
\begin{equation}
	Au(\vf,x)= \sum_{n,n'\in\N_0}  \psi_{[n]}(x)A_{[n]}^{[n']}(\vf)u^{[n']}(\vf) =\sum_{\ell,\ell'\in\Z^\nu \atop n,n'\in\N_0} \psi_{[n]}(x)A_{[n]}^{[n']}(\ell-\ell')u^{\ell',[n']}e^{\im\ell\cdot\vf} \,,
\end{equation}
where we denote $u^{[0]}(\vf):=u^{0}(\vf)\in L^2(\T^\nu)$, $ \psi_{[0]}:=\psi_{0} \in L^2(\T)$ and, for $n\geq 1$, 
\begin{equation}
	\begin{aligned}
		& u^{[n]}(\vf):=
		\begin{pmatrix}
			u^{-n}(\vf) \\ u^{n}(\vf) 
		\end{pmatrix}
		\in \C^{2\times 1}\otimes L^2(\T^\nu)\,, \quad \psi_{[n]}=\begin{pmatrix}
			\psi_{-n}, & \!\!\psi_{n}
		\end{pmatrix} \in \C^{1\times 2} \otimes L^2(\T)\,.
	\end{aligned}
\end{equation}

\begin{rem}\label{rem:coeff.mat}
	If $A(\vf)$ is a bounded operator, the following implications hold:
	\begin{align*}
	& A  = A^* \Longleftrightarrow  A_{j}^{j'}(\ell)  = \bar{ A_{j'}^{j}(-\ell)}
	\ \ \  \forall\,\ell\in\Z^\nu, \  j,j' \in \Z \,; \\
	& \bar A  = A^* \Longleftrightarrow  A_{j}^{j'}(\ell) =  A_{j'}^{j}(\ell)
	\  \ \ \forall\, \ell\in\Z^\nu, \ j,j'\in\Z \,.
	\end{align*}
	Moreover, in terms of the block representation, we have
	\begin{equation*}
	A = A^* \Longleftrightarrow A_{[n]}^{[n']}(\ell): = \bar{\big(A_{[n']}^{[n]} (-\ell)\big)^T} \quad \forall\,\ell\in\Z^\nu,\ n,n'\in\N_0
	\end{equation*}
	and, if the eigenfunctions of $L_q$ satisfy $\psi_{-j} =\bar{\psi_j} $ for any $j\in\Z$,
	\begin{equation*}
	\bar A = A^* \Longleftrightarrow A_{[n]}^{[n']}(\ell) = \big(A_{-[n']}^{-[n]}(\ell)\big)^T
\quad \forall\,\ell\in\Z^\nu, \ n,n'\in\N_0\,.
	\end{equation*}
\end{rem}
A useful norm we choose to put on the space of such operators is in the following:
\begin{defn}\label{def:sdecay}
	Let $\tw>0$, $s\in\R$ and let $A(\omega)= A(\omega;\vf) :  H^\infty(\T^{\nu+1}) \to H^{-\infty}(\T^{\nu+1})$ be  a  linear operator that is Lipschitz continuous with respect to $\omega\in\Omega\subset R_{\tM}$. We say that $A(\omega)$ in the class $\cM_{s}$ if 
	\begin{equation}
		| A |_{s}^{\lip(\tw)}:= | A |_{s,\Omega}^{\lip(\tw)} :=  \sup_{\omega\in\Omega}  |A(\omega)|_{s} + \tw \sup_{\omega_1, \omega_2 \in \Omega \atop \omega_1 \neq \omega_2 } \frac{|A(\omega_1)-A(\omega_2)|_{s}}{| \omega_1-\omega_2|} < \infty\,,
	\end{equation}
	where $| A|_s$ is the finite $s$-decay norm defined by
	\begin{equation}\label{eq:sdecay_norm}
	\abs{A}_s^2:=\sum_{h\in \N_0 \atop \ell\in\Z^\nu }\la \ell, h \ra^{2s}\sup_{|n'-n|=h}\|A_{[n]}^{[n']}(\ell)\|_{\rm HS}^2 \,, \quad \|A_{[n]}^{[n']}(\ell)\|_{\rm HS}^2 :=\sum_{j\in[n]\atop j'\in[n']} |A_{j}^{j'}(\ell)|^2 \,,
	\end{equation}
	with $\braket{\ell,h}:= \max\{ 1, | \ell|, h\}$ and  $\|\,\cdot\,\|_{\rm HS}$ being the Hilbert-Schmidt norm on the finite dimensional operator $A_{[n]}^{[n']}\in \cL(\fE_{n'},\fE_{n})$ (in particular, $\|A_{[0]}^{[0]}\|_{\rm HS}=\abs{A_0^0}$).
\end{defn}
\begin{lem}\label{prop:tame}
	{\bf (Tame estimates of the $s$-decay).}
	Let $A,B\in\cM_s$, with $s\geq s_0$. Then $AB\in\cM_s$ with tame estimate
	\begin{equation}
	\abs{AB}_s^{\lip(\tw)} \leq C(s_0)  |A|_{s_0}^{\lip(\tw)} |B|_s^{\lip(\tw)} + C(s) | A|_s^{\lip(\tw)} |B|_{s_0}^{\lip(\tw)}  \,.
	\end{equation}
\end{lem}
\begin{rem}\label{rem:opnorm_vs_sdecay}
	If $A: H^\infty(\T^{\nu+1}) \to H^{-\infty}(\T^{\nu+1})$ has finite $s$-decay norm with $s \geq s_0$, then, for any $r\in[0,s]$, 
	$A$ extends to a bounded operator $H^{r}(\T^{\nu+1}) \to H^{r}(\T^{\nu+1})$.  Moreover, by tame estimates, one has the quantitative bound
	$ \norm{A}_{\cL(H^r)} \leq C_{r,s} |A|_s$. 
\end{rem}

Given an operator $M : \cL(\fE_{n'},\fE_{n}) \to \cL(\fE_{n'},\fE_{n})$, we denote the operator norm by
\begin{equation}\label{Opnn}
	\| M \|_{\Op(n,n')} := \|  M \|_{\cL(\cL(\fE_{n'},\fE_{n}) )} := \sup \big\{ \| M X \|_{\rm HS} \,: \, \| X \|_{\rm HS}\leq 1 \big\}\,.
\end{equation}
The identity operator on $\cL(\fE_{n},\fE_{n'})$ will be denoted by $\uno_{n,n'}$.

\begin{lem}\label{lemma:MRMRL}
	Let $A\in \cL(\fE_{n})$ and $B\in \cL(\fE_{n'})$.
	We define  $M_L(A), M_R(B)$ as the operators on  $\cL(\fE_{n'},\fE_{n})$ acting as the left multiplication by $A\in \cL(\fE_{n})$ and as right multiplication by $B\in \cL(\fE_{n'})$, respectively:
	\begin{equation}\label{MRML}
		M_L(A)X:= AX \,, \quad M_R(B) X:=XB \,, \quad \forall X\in \cL(\fE_{n'},\fE_{n})\,.
	\end{equation}
	Then $M_L(A), M_R(B) \in \cL\big( \cL(\fE_{n'},\fE_{n}) \big)$, with estimates
	\begin{equation}
		\| M_L(A) \|_{\Op(n,n')} \leq \| A \|_{\rm HS}\,, \quad \| M_R(B) \|_{\Op(n,n')} \leq \| B \|_{\rm HS}\,.
	\end{equation}
If $A$, $B$ are self-adjoint, then the operators $M_L(A)$ and $M_R(B)$ are self-adjoint and $\spec(M_L(A)\pm M_R(B))= \{ \lambda\pm \mu \,:\, \lambda\in\spec(A),\, \mu\in\spec(B) \}$.
	For $A=A(\omega)$, $B=B(\omega)$ Lipschitz continuous with respect to the parameter $\omega\in\Omega\subseteq R_{\tM}$, the bounds extend to the norm $\|\,\cdot\, \|_{\rm HS}^{\lip(\tw)}$.
\end{lem}

The next result holds for a general finite dimensional Hilbert space $(\cH,(\,,\,)_{\cH}) $ of dimension $d\in\N$. For a given self-adjoint operator $A\in\cL(\cH)$, we order its spectrum as $\spec(A)=\{   \lambda_1(A)\leq ..., \leq \lambda_{d}(A)\}$.

\begin{lem}\label{lemma.astratto.fin}
	{\bf (Lemma 2.5, \cite{Mont.Kirch17}).} The following hold:
	\\[1mm]
	\noindent $(i)$ Let $A_1,A_2\in\cL(\cH)$ be self-adjoint. Then their eigenvalues satisfy the Lipschitz property $|\lambda_{p}(A_1)-\lambda_{p}(A_2) |\leq \| A_1 - A_2 \|_{\cL(\cH)}$ for any $p=1,...,d$;
	\\[1mm]
	\noindent $(ii)$ Let $A\in\cL(\cH)$ be self-adjoint, with $\spec(A)\subset \R \setminus\{0\}$. Then $A$ is invertible and its inverse $A^{-1}$ satisfies $\| A^{-1} \|_{\cL(\cH)} = \big(\min_{p=1,..,d}| \lambda_{p}(A)|\big)^{-1}$.
\end{lem}

\noindent{\bf Operator matrices.}
 We are going to meet matrices of operators of the form 
\begin{equation}
\label{eq:Amatrix}
\bA=\left( \begin{matrix}
A^{d} & A^o \\ -\bar{A^o} & -\bar{A^d} 
\end{matrix} \right) \ , 
\end{equation}
where $A^d$ and $A^o$ are linear  operators belonging to the class $\cM_{s}$.
Actually, the operator $A^d$ on the diagonal will have different decay properties than the element on the anti-diagonal $A^o$. Therefore, we introduce 
classes of operator matrices in which we keep track of these differences. To this purpose, for any $m\in\R$, we define the following linear operator operator, with $\braket{j}:=\max\{1,|j|\}$,
\begin{equation}\label{diago.Dx}
		\begin{aligned}
			u(x) = \sum_{j\in\Z} u^j \psi_{j}(x) \mapsto \braket{D}^m u(x) := \sum_{j\in\Z} \braket{j}^m u^j \psi_{j}(x)\,.
		\end{aligned}
\end{equation}

\begin{defn}\label{M.sdecay.ab}
	Let   $\alpha, \beta \in \R$, $s \geq 0$ and let $\bA(\omega)$ be a  operator matrix  of the form
	\eqref{eq:Amatrix} that is Lipschitz continuous with respect ot the parameter $\omega\in\Omega\subseteq R_{\tm}$. We say that 
	$\bA $ belongs to $\cM_{s}(\alpha, \beta)$ if
	\begin{equation}\label{struttura}
	[A^d]^* = A^d \ , \qquad [A^o]^* = \bar{A^o}
	\end{equation}	
	and one also has 
	\begin{align}\label{M1}
	&	\la D \ra^\alpha \,  A^d  \ , \ A^d \, \la D \ra^{\alpha}  \in \cM_{s}  \,, \\
	\label{M2}
	&	\la D \ra^\beta \,  A^o  \ , \ A^o \, \la D \ra^{\beta}  \in \cM_{s}  \,, \\
	\label{M3}
	&	\la D \ra^\varsigma \,  A^\delta \, \la D \ra^{- \varsigma}  \in \cM_{s} \ , \quad \forall\, \varsigma \in \{ \pm \alpha, \pm \beta, 0 \} \,, \ \ \forall\, \delta \in \{d, o\} \ .
	\end{align}
	We endow $\cM_{s}(\alpha, \beta)$  with the  norm 
	\begin{align}
		\abs{\bA}_{s,\alpha, \beta}^{\lip(\tw)} := & 
		|\braket{D}^{\alpha}A^d|_{s}^{\lip(\tw)}
		+|A^d\braket{D}^{\alpha}|_{s} ^{\lip(\tw)}
		+|\braket{D}^{\beta}A^o|_{s}^{\lip(\tw)}
		+|A^o\braket{D}^{\beta}|_{s}^{\lip(\tw)} \nonumber \\	
		& 
		+ \sum_{\varsigma \in\{ \pm \alpha, \pm \beta,  0 \} \atop \delta \in \{d, o\}} |\braket{D}^{\varsigma}A^\delta \braket{D}^{-\varsigma}|_{s}^{\lip(\tw)} \,, \label{eq:sdecay_matrix_norm}	 
	\end{align}
	with the convention that, in case of repetition (when $\alpha=\beta$, $\alpha=0$ or $\beta=0$), the same terms are not summed twice.
\end{defn}

\begin{rem}
	Let us motivate the properties describing the class $\cM_{s}(\alpha, \beta)$.
	Condition \eqref{struttura} is equivalent to ask that $\bA$ is the Hamiltonian vector field of a real valued quadratic Hamiltonian, see e.g. \cite{mont17} for a discussion.  Conditions \eqref{M1} and \eqref{M2} control the decay properties for the coefficient of the coefficients of the matrices associated to $A^d$ and $A^o$: indeed, recalling \eqref{diago.Dx}, the matrix coefficients of $\la D \ra^{\alpha} A \,  \la D\ra^{\beta} $ are given by
	\begin{equation}\label{maledetta}
		[ \la D \ra^{\alpha} A \,  \la D\ra^{\beta}]_{[n]}^{[m]}(\ell)   = \la n \ra^\alpha \, A_{[n]}^{[m]}(\ell) \, \la m \ra^\beta \,,
	\end{equation}
	therefore decay (or growth) properties for the matrix coefficients of the operator $A$ are implied by the boundedness of the norms $| \cdot |_{s}^{\lip(\tw)}$. Condition \eqref{M3} is just for  simplifying some   computations below.
\end{rem}

\begin{lem}\label{propieta.class.sdecay}
	Let $0 \leq s' \leq s $ $\alpha\geq \alpha'$, $\beta\geq \beta'$.  The following holds:
	\\[1mm]
	\noindent $(i)$ We have $\cM_{s}(\alpha, \beta)\subseteq \cM_{s'}(\alpha', \beta')$ with estimates
	$\abs{\bA}_{s',\alpha',\beta'}^{\lip(\tw)} \leq \abs{\bA}_{s,\alpha,\beta}^{\lip(\tw)}$;
	\\[1mm]
	\noindent $(ii)$ Let
	\begin{equation}\label{proj.def.mat}
		\Pi_{\tN} \bA(\omega;\vf) := \sum_{|\ell| \leq \tN} \bA(\omega;\ell) e^{\im \ell \cdot \vf} \,, \quad \Pi_{\tN}^\perp := {\rm Id} - \Pi_{\tN}\,,
	\end{equation}
  be the projectors on the frequencies $\ell\in\Z^{\nu}$ smaller and larger than $\tN\in\N$, respectively. Then, for any $\tb\geq 0$,
  \begin{equation}
  	| \Pi_{\tN} \bA |_{s+\tb,\alpha,\beta}^{\lip(\gamma)} \leq \tN^{\tb} | \bA |_{s,\alpha,\beta}^{\lip(\gamma)}\,, \quad | \Pi_{\tN} ^\perp\bA |_{s,\alpha,\beta}^{\lip(\gamma)} \leq \tN^{-\tb} | \bA |_{s+\tb,\alpha,\beta}^{\lip(\gamma)}\,.
  \end{equation}
\end{lem}

\noindent{\bf Commutators and flows.}
These classes of matrices enjoy also closure properties under commutators and flow generation.
We define the adjoint operator
\begin{equation}
\label{adj}
\ad_\bX(\bV):=\im [\bX,  \bV]  \ ;
\end{equation}
note the multiplication by the imaginary unit in the definition of the adjoint map.
\begin{lem}[Commutator]
	\label{lem:com}
	Let $\alpha>0$ and $s\geq s_0 > \frac{\nu+1}{2}$. 	Assume $\bV \in \cM_{s}(\alpha,0)$ and $\bX \in \cM_{s}(\alpha, \alpha)$.  Then $\ad_\bX(\bV)$  
	belongs to $\cM_{s}(\alpha, \alpha)$ with tame estimates
	\begin{equation}\label{eq:ad_lip_alg}
	\begin{aligned}
	\abs{\ad_\bX(\bV)}_{s,\alpha,\alpha}^\wlip{\tw}\leq &\,
	C_{s_0}|\bX|_{s_0,\alpha,\alpha}^\wlip{\tw} |\bV|_{s,\alpha,0}^\wlip{\tw} +C_{s}|\bX|_{s,\alpha,\alpha}^\wlip{\tw} |\bV|_{s_0,\alpha,0}^\wlip{\tw}
	\,.
	\end{aligned}
	\end{equation}
\end{lem} 
\begin{lem}[Flow]\label{lem:flow}
	Let $\alpha>0$, $ s\geq s_0>\frac{\nu+1}{2}$. Assume $\bV \in \cM_{s}(\alpha,0)$, $\bX \in \cM_{s}(\alpha, \alpha)$.  Then the followings hold true:
	\begin{itemize}
		\item[(i)]For any $ r\in[0,s]$, the operator 
		$ e^{\im \bX} $
		is bounded in $H^r(\T^{\nu+1})\times H^r(\T^{\nu+1})$;
		\item[(ii)] The operator   $e^{\im  \bX} \, \bV \, e^{-\im  \bX}$ belongs to $\cM_{s}(\alpha, 0)$, whereas 
		$e^{\im  \bX} \, \bV \, e^{-\im  \bX} - \bV$ belongs to $\cM_{s}(\alpha, \alpha)$ with the quantitative tame estimates 
	\begin{equation}\label{eq:flow_tame}
		\begin{aligned}
			& \abs{e^{ \im  \bX} \, \bV \, e^{-\im  \bX}}_{s,\alpha, 0}^{\lip(\tw)} \leq 
			e^{ C_{s_0} |\bX|_{s_0,\alpha, \alpha}^{\lip(\tw)} } \big( C_{s} |\bX |_{s,\alpha,\alpha}^{\lip(\tw)} |\bV|_{s_0,\alpha, 0}^{\lip(\tw)} + |\bV|_{s,\alpha, 0}^{\lip(\tw)} \big)  \,; \\
			& \abs{e^{ \im  \bX} \, \bV \, e^{-\im  \bX} -\bV}_{s,\alpha, \alpha}^{\lip(\tw)} \leq 
			e^{ C_{s_0} |\bX|_{s_0,\alpha, \alpha}^{\lip(\tw)} } \big( C_{s} |\bX |_{s,\alpha,\alpha}^{\lip(\tw)} |\bV|_{s_0,\alpha, 0}^{\lip(\tw)} \\
			& \quad \quad \quad \quad \quad \quad \quad \quad \quad \quad \quad \quad \quad \quad  + C_{s_0} |\bX |_{s_0,\alpha,\alpha}^{\lip(\tw)} |\bV|_{s,\alpha, 0}^{\lip(\tw)} \big)\,.
		\end{aligned}
	\end{equation}
	\end{itemize}
\end{lem}
The proofs of these two results are postponed in Appendix \ref{app:tech}.

\section{Embeddings}\label{sec:embe}

In this section we want to embed the class of pseudodifferential operators $\Ops^{m} $, which are defined on the exponential basis $(e_j(x):=e^{\im j x})_{j\in\Z}$,
 into the class of matrix operators $\cM_{s}(\alpha,\beta)$ in Definition \ref{M.sdecay.ab}, which are constructed  on the $L^2$-basis $(\psi_{j}(x))_{j\in\Z}$ of eigenfunctions for $L_q$, instead.
We adopt the following notations in this section.
 The coefficients of a function $u(x):\T \rightarrow \C$ are denoted with respect to the eigenfunction basis and the exponential basis respectively by
 \begin{equation}
 	u^j := \left(u,\psi_{j}\right)_{L^2(\T)}=\int_\T u(x) \bar{\psi_{j}(x)} \wrt x \,, \quad \whu(j):=\left( u,e_j \right)_{L^2(\T)} = \int_\T u(x) e^{-\im j x} \wrt x \,.
 \end{equation}
%
%
\subsection{Craig-Wayne Lemma for smooth potential.}
The idea is that, for suitably smooth potentials $q(x)$, each eigenfunction $\psi_{j}(x)$ is mostly concentrated around $e_j(x):=e^{\im jx}$ and $e_{-j}(x):=e^{-\im jx}$, decaying outside the subspace spanned by the two exponentials. This holds true in the analytic setting and we need to extend this principle to the Sobolev regularity case.
Recall the following result in \cite{CrWa}:
\begin{thm}
	{\bf (Craig-Wayne Lemma - Lemma 6.6, \cite{CrWa}).}
	Let $q(x):\T\rightarrow \R$ be analytic on the strip $\T_{\bar\sigma} := \Set{z \in \C | \re (z)\in\T, \ \abs{{\rm Im} z}\leq \bar\sigma }$. Then, for any $\sigma_*\in (0,\bar\sigma)$, there exists $C>0$ depending only on $q$ and $\sigma_*$ such that, for any $j,j'\in\Z$
	\begin{equation}\label{eq:cw_an}
	| \left({\psi_j,e_{j'}}\right)_{L^2(\T)} | \leq C e^{-\sigma_*|j-|j'||} \ .
	\end{equation}
\end{thm}
To extend this result when $q(x)\in H^\infty(\T)$, we follow a construction presented by Pöschel in \cite{Posch11}. For $s\in\R$, $u(x)=\sum_{n\in\Z}\whu(n)e_n(x)\in H^s(\T)$ and $j\in \Z$, we define the shifted norm
\begin{equation}
\norm{u}_{s;j}^2:= \norm{u e_j}_s^2 = \sum_{n\in\Z} \braket{n}^{2s} \abs{\whu(n-j)}^2 =  \sum_{n\in\Z} \braket{n+j}^{2s} \abs{\whu(n)}^2 \ .
\end{equation}
Consider the eigenvalue equation
\begin{equation}\label{eq:LS_0}
-\partial_{xx} f(x) +q(x) f(x) = \lambda f(x) \ \rightsquigarrow \ A_\lambda f(x) = V f(x)
\end{equation}
where $A_\lambda := +\partial_{xx} + \lambda \,{\rm Id} $ and $V f(x):= q(x) f(x)$. Fix $s>0$, $n \in \N_0$ and consider the orthogonal decomposition $H^s(\T) := \cP_n \oplus \cQ_n$, where
\begin{equation}
\begin{aligned}
	 & \cP_n := \linspan_\C\{ e_n,e_{-n} \} \,,  \\
	 & \cQ_n:=\Big\{ v= \sum_{m\in\Z} \whv(m)e_m \in H^s(\T) \ :\ \whv(m) = 0 \text{ for }\abs m = n \Big\}\,.
\end{aligned}
\end{equation}
Let $P_n$ and $Q_n$ be the corresponding orthogonal projections on $\cP_n$ and $\cQ_n$, respectively. We write $f = u+v \in H^s(\T)$, with $u:= P_n f$ and $v := Q_n f$,  and we consider the Lyapunov-Schmidt reduction scheme for equation \eqref{eq:LS_0}:
\begin{equation}\label{eq:LS_1}
\left\{ \begin{array}{l}
A_\lambda u = P_n V (u+v) \\
A_\lambda v = Q_n V (u+v)
\end{array} \right. .
\end{equation}
We call the first equation in \eqref{eq:LS_1} the $P$-equation and the second one the $Q$-equation. From the forthcoming discussion, it will be clear that the case $n=0$ corresponds to treat the case $q=0$, which is trivial. Therefore, from now on, let $n\geq 1$.
We solve first the $Q$-equation. For any $\lambda\in\cU_n$, where
 \begin{equation}\label{Un.set}
 	\cU_n := \{ \mu \in \C \,:\, |\mu-n^2|\leq n/2 \}
 \end{equation}
 the operator $A_\lambda$ is invertible  on the range of $Q_n$ with bound on the inverse given by $\| A_\lambda^{-1}\|_{\cL(H^s(\T))} \leq 2\, n^{-1} $
and $A_\lambda^{-1}e_m = (-m^2+\lambda)^{-1}e_m$. Therefore, the $Q$-equation can be rewritten as
\begin{equation}\label{eq:Qeq_v}
({\rm Id}-T_n)v = T_n u\,, \quad T_n:= A_\lambda^{-1}Q_n V \,.
\end{equation}
\begin{lem}\label{lem_Tn}
	Let $q\in H^s(\T)$. There exists $\tC_s>0$ independent of $n\in \N$ such that, for any $w\in H^s(\T)$ and $j\in\Z$, one has
	\begin{equation}
	\norm{T_n w}_{s;j}\leq \tC_s\, n^{-1} \norm{q}_{s} \norm{w}_{s;j} \,, \quad \forall\, w \in \cQ_n \,.
	\end{equation}
\end{lem}
\begin{proof}
	Using the self-adjointness of $A_\lambda^{-1}$ together with Cauchy-Schwartz inequality, we compute, recalling that $\lambda\in \cU_n$ as in \eqref{Un.set} and that $\whw(\pm n)=0$,
	\begin{equation}
	\begin{split}
	\| & T_n w \|_{s;j}^2  = \sum_{j'\in \Z, \  | j'| \neq n} \braket{j'+j}^{2s} \big| \big( A_\lambda^{-1} Q_n V w, e_{j'}\big)_{L^2} \big|^2 \\
	& =\sum_{ |j'|\neq  n} \frac{1}{|\lambda-(j')^2|^2} \Big|\braket{j+j}^{s} \sum_{k\in\Z}\whq(k)\whw(j'-k) \Big|^2\\
	& \leq \sum_{| j'|\neq  n} \frac{ g(j',j)}{|\lambda-(j')^2|^2}\sum_{k\in\Z}\braket{k}^{2s}|\whq(k)|^2\braket{j'+j-k}^{2s}| \whw(j'-k)|^2\\
	& \leq C_s \sum_{ |j'| \neq  n}\frac{1}{|\lambda-(j')^2|^2} \sum_{k\in\Z}\braket{k}^{2s}|\whq(k)|^2\braket{j'+j-k}^{2s}|\whw(j'-k)|^2\\
	& \leq C_s \sum_{k\in\Z}\braket{k}^{2s}|\whq(k)|^2 \sum_{ |a+k |\neq n} \frac{1}{|\lambda-(a+k)^2|^2}\braket{a+j}^{2s}|\whw(a)|^2 \\
	& \leq C_s \norm{q}_s^2 \frac{4}{n^2} \norm{w}_{s;j}^2 =:  \frac{\tC^2}{ n^2} \norm{q}_s^2 \norm{w}_{s;j}^2\,,
	\end{split}
	\end{equation}
	where we have defined $\tC_{s}^2:=4 C_s $ and $ g(j',j):=  \sum_{k\in\Z} \frac{\braket{j'+j}^{2s}}{\braket{k}^{2s}\braket{j'+j-k}^{2s} } \leq  C_s< \infty$.
\end{proof}
By Lemma \ref{lem_Tn}, we can prove the invertibility of the operator ${\rm Id}-T_n$.
\begin{cor}\label{cor:inv_Tn}
	Fix $s>0$ and $n\in\N$. Provided $\|q\|_{s}\leq \frac{n}{2\tC_{s}}$, with $\tC_{s}$ as in Lemma \ref{lem_Tn}, the operator ${\rm Id}-T_n$ is invertible and bounded with respect to the shifted norm $\|\,\cdot\,\|_{s;j}$ with bounds for any  $j\in\Z$ given by $$\sup_{\|w\|_{s;j}=1}\| ({\rm Id}-T_n)^{-1}w \|_{s;j}\leq 2 \,, \quad \forall \, w\in \cQ_n  \,. $$
\end{cor}
\begin{proof}
	Under the assumption $\|q\|_{s}\leq \frac{n}{2\tC_{s}}$, the claim follows directly by Lemma \eqref{lem_Tn} and a Neumann series argument.
\end{proof}
We come back now to equation \eqref{eq:Qeq_v}, which now has the well-defined solution
\begin{equation}\label{eq:sol_Qeq}
v:=({\rm Id}-T_n)^{-1}T_n u \in \cQ_n , \ \text{ with } \ \norm{v}_{s;j} \leq 2 \tC_{s} \,n^{-1} \norm{q}_s \norm{u}_{s;j} \ .
\end{equation}
Therefore, the $P$-equation reduces to
\begin{equation*}
A_\lambda u = P_n V (u+ v) = P_n V u + P_n V ({\rm Id}-T_n)^{-1} T_n u = P_n V ({\rm Id}-T_n)^{-1} u\,,
\end{equation*}
which is actually a $2\times 2$ system in the variables $(\whu(-n),\whu(n))$:
\begin{equation}\label{eq:Peq_matrix}
S_n(\lambda) \begin{pmatrix}
\whu(-n) \\ \whu(n)
\end{pmatrix} := \begin{pmatrix}
\lambda - n^2 - a_{-n} & -c_{-n} \\ -c_n & \lambda -n^2 -a_n
\end{pmatrix} \begin{pmatrix}
\whu(-n) \\ \whu(n)
\end{pmatrix} = 0\,,
\end{equation}
where $a_n:= \left( V({\rm Id}-T_n)^{-1} e_n, e_n \right)_{L^2}$,  $c_n:=\left( V({\rm Id}-T_n)^{-1}e_{-n},e_n \right)_{L^2}$.
\begin{lem}
	The following hold:
	\\[1mm]
	\noindent $(i)$ 	$a_n=a_{-n}$ for any $n\in\N$;
			\\[1mm]
		\noindent $(ii)$ For a real valued potential $q$, one has $c_{-n}=\bar{c_n}$ for any $n\in\N$;
			\\[1mm]
		\noindent $(iii)$ If $\norm{q}_s\leq \frac{n}{2\tC_{s}}$, then the determinant of $S_n(\lambda)$ has exactly two roots $\lambda_{n,+},\lambda_{n,-}\in D_n$, where  $ D_n:=\{ \mu\in\C \ | \ \abs{\mu-n^2} \leq \frac{2\tC_{s}}{3}\norm{q}_s \}\subset \cU_n$.
\end{lem}
\begin{proof}
	(i) Note that
	\begin{equation*}
	(VT_n)^* = (A_\lambda^{-1}Q_nV)^* V^* = V^*  (A_\lambda^{-1}Q_n)^* V^* = \bar V A_\lambda^{-1} Q_n \bar V  = \bar V \bar{T_n}\,.
	\end{equation*}
	By expanding in Neumann series, we have $( V({\rm Id}-T_n)^{-1} )^* = \bar{V({\rm Id}-T_n)^{-1}}$. Now
	\begin{equation*}
	\begin{split}
	a_n & = \big( e_n,(V({\rm Id}-T_n)^{-1})^*e_n \big)_{L^2} = \big( e_n,\bar{V({\rm Id}-T_n)^{-1}}e_n \big)_{L^2} \\
	& = \big( e_n,\bar{V({\rm Id}-T_n)^{-1}e_{-n}}\big)_{L^2} = \big( V({\rm Id}-T_n)^{-1}e_{-n},e_{-n}\big)_{L^2} = a_{-n} \ ;
	\end{split}
	\end{equation*}
	(ii) Since $q$ is real valued, we obtain $\left(V({\rm Id}-T_n)^{-1}\right)^* =V({\rm Id}-T_n)^{-1} $. Hence,
	\begin{equation*}
	\bar{c_n} = \big( e_n,V({\rm Id}-T_n)^{-1}e_{-n} \big)_{L^2} =\big( V({\rm Id}-T_n)^{-1}e_{n},e_{-n} \big)_{L^2} = c_{-n} \ ;
	\end{equation*}
	(iii) The claim follows from a topological degree argument: we refer to Lemma 2 and Lemma 3 in \cite{Posch11} for more details.
\end{proof}
Therefore, the solutions of the system \eqref{eq:Peq_matrix}, namely the vectors 
$$u_{n,\pm}:=(\wh{u_{n,\pm}}(-n),\wh{u_{n,\pm}}(n)) \simeq \wh{u_{n,\pm}}(-n)e_{-n}+\wh{u_{n,\pm}}e_n$$
 are eigenvectors for the matrix $\Big(\begin{smallmatrix}
n^2+a_{n} & c_{-n} \\ c_n & n^2 + a_n
\end{smallmatrix}\Big)$.
Using the solution \eqref{eq:sol_Qeq} of the $Q$-equation, we define $v_{n,\pm}:=({\rm Id}-T_n)^{-1}T_n u_{n,\pm}\in\cQ_n$. Finally, we define
\begin{equation}
f_{n,\pm}= u_{n,\pm}+v_{n,\pm} =  u_{n,\pm} + ({\rm Id}-T_n)^{-1}T_n u_{n,\pm} \in \cP_n\oplus \cQ_n \ .
\end{equation}
which are exactly the eigenfunctions related  to the eigenvalues $\lambda_{n,\pm}$.

We are now ready to state the main theorem concerning the localization of the eigenfunctions for the operator $L_q$ on the exponential basis.

\begin{thm}\label{thm:decay_eigen_smooth}
	{\bf (Craig-Wayne Lemma in Sobolev regularity).}
	Let $s>0$, $n\in\N$ be fixed and let $q\in H^s(\T)$ be real-valued. Assume $\norm{q}_s \leq \frac{ n}{2 \tC_s}$, with $\tC_{s}$ as in Lemma \ref{lem_Tn}. Then, 
	for any $m\in\Z$, we have the following polynomial decay:
	\begin{equation}\label{eq:decay_eigen_smooth}
	| ( f_{n,\pm},e_m )_{L^2} | \leq 2 \braket{  |m|- n}^{-s} \ .
	\end{equation}
\end{thm}
\begin{proof}
	For $\abs m=  n$, the estimate \eqref{eq:decay_eigen_smooth} is trivial, as $P_n f_{n,\pm}=u_{n,\pm}$ is clearly arbitrarily bounded, for instance by 1. Hence, let $\abs m\neq n$. Note the following duality with respect to the shifted norm:
	\begin{equation}
	\abs{ \left(f,g\right)_{L^2} } = \abs{\left( f e_j, g e_j \right)_{L^2}} \leq \norm{f e_j}_s \norm{g e_j}_{-s} = \norm f_{s;j} \norm g_{-s;j} \ .
	\end{equation}
	Therefore, we have, for any $j_1,j_2\in\Z$,
	\begin{equation}
	\begin{aligned}
	&|(f_{n,\pm},e_m)_{L^2}|  = \abs{\left(v_{n,\pm},e_m\right)_{L^2}} \\
	& \leq |\whu_{n,\pm}(n)| \big|(({\rm Id}-T_n)^{-1}T_n e_n,e_m)_{L^2}\big| + |\whu_{n,\pm}(-n)| \big|({\rm Id}-T_n)^{-1}T_n e_{-n},e_m)_{L^2}\big| \\
	& \leq \|({\rm Id}-T_n)^{-1}T_n e_n\|_{s;j_1}\norm{e_m}_{-s;j_1} + \|({\rm Id}-T_n)^{-1}T_n e_{-n}\|_{s;j_2} \norm{e_m}_{-s;j_2} \\
	& \leq \frac{\tC_s}{ n}\norm q_s \big( \norm{e_n}_{s;j_1} \norm{e_m}_{-s;j_1} + \norm{e_{-n}}_{s;j_2} \norm{e_m}_{-s;j_2} \big)\\
	& = \frac{\tC_s}{n}\norm{q}_s \big( \braket{n+j_1}\braket{m+j_1}^{-s} + \braket{-n+j_2}^s \braket{m+j_2}^{-s} \big) \ .
	\end{aligned}
	\end{equation}
	Now, choosing $j_1=-n$ and $j_2=n$, we conclude that
	\begin{equation}
	\abs{\left(f_{n,\pm},e_m\right)_{L^2}} \leq \frac{\tC_s\norm q_s}{n} \left( \frac{1}{\braket{m-n}^s}  + \frac{1}{\braket{m+n}^s}\right) \leq \frac{2}{\braket{\abs m- n}^s}
	\end{equation}
	and the claim is proved.
\end{proof}
The decay in \eqref{eq:decay_eigen_smooth} in Sobolev regularity is the key for properly defining the change from the exponential to the $L_q$-eigenfunctions basis and vice versa.
\begin{defn}\label{M_basis_ch}
	We define the linear operator $\fM=\big( \fM_{[n]}^{[m]} \big)_{n,m\in\N_0}$ by the blocks
	\begin{equation}
	\fM_{[n]}^{[m]} :=
	\begin{pmatrix}
	\left( \psi_{-n},e_{-m} \right)_{L^2} & \left( \psi_{-n},e_{m} \right)_{L^2} \\ \left( \psi_{n},e_{-m} \right)_{L^2} & \left( \psi_{n},e_{m} \right)_{L^2}
	\end{pmatrix} 
\quad \forall n,m \geq 1\,,
	\end{equation}
	$\fM_{[0]}^{[m]}:= \big( (\psi_{0},e_{-m})_{L^2} , (\psi_{0},e_{m})_{L^2}\big)$ for $m\neq 0$, $\fM_{[n]}^{[0]}:= \big( (\psi_{-n},1)_{L^2} , (\psi_{n},1)_{L^2}\big)^{T}$ for $n\neq 0$, and $\fM_{[0]}^{[0]}:=(\psi_{0},1)_{L^2}$.
\end{defn}
\begin{cor}\label{cor:decay_class}
	For any $s\geq 0$ and any real-valued $q\in H^\infty(\T)$, there exists a constant $C(s,q)>0$ such that
	\begin{equation}\label{sdecay_M}
		| \fM |_{s;M}^2 := \sum_{h\in\N_0} \braket{h}^{2s} \sup_{\abs{n-m}=h} \| \fM_{[n]}^{[m]} \|_{\rm HS}^2 \leq C(s,q)<\infty \,.
	\end{equation}
	Moreover, $| \fM^T|_{s;M} = |\fM|_{s;M}$.
\end{cor}
\begin{proof}
	For a given $s\geq 0$, fix any $s_1>s+s_0$ and take $N_0=N_0(q,s_1)\in\N$ such that $\| q\|_{s_1} \leq \frac{N_0}{2\tC_{s_1}}$, with $\tC_{s_1}>0$ as in Lemma \ref{lem_Tn}.
	For $n \geq N_0$, we can apply Theorem \ref{thm:decay_eigen_smooth}: for any $m\in\N_0$ we have
	\begin{equation}\label{eq:boundA}
	\|\fM_{[n]}^{[m]}\|_{{\rm HS}}^2 = |\fM_{-n}^{-m}|^2+|\fM_{-n}^{m}|^2+|\fM_{n}^{-m}|^2+|\fM_{n}^{m}|^2 \leq  \frac{16}{\braket{ m- n}^{2s_1}}
	\end{equation}
	For $0\leq  n < N_0$, we use the direct decay effect for the eigenfunctions of $L_q$, together with the Peetre inequality $\braket{|m|- n}^{s_1}\leq \braket{m}^{s_1}\braket{n}^{s_1}$:
	\begin{equation*}
	q\in H^\infty \ \Rightarrow \ \psi_{\pm n} \in H^\infty \ \Rightarrow \ \abs{\left( \psi_{\pm n},e_m \right)_{L^2}} \leq \frac{C_{s_1}}{\braket{m}^{s_1}}\leq C_{s_1}\frac{\braket{n}^{s_1}}{\braket{ |m|-n}^{s_1}}  \quad\forall\,m\in\Z \,.
	\end{equation*}
	The bound that we obtain in this case is, for any $m\in\N_0$,
	\begin{equation}\label{eq:boundB}
	\|\fM_{[n]}^{[m]}\|_{{\rm HS}}^2 \leq 16 \,C_{s_1}\frac{\braket{n}^{2s_1}}{\braket{ m- n}^{2s_1}} \leq \frac{16\,C_{s_1} \braket{N_0}^{2s_1}}{\braket{ m-n}^{2s_1}} \ .
	\end{equation}
	Summing up \eqref{eq:boundA} and \eqref{eq:boundB}, we can conclude that
	\begin{equation*}
	\begin{split}
	|\fM|_s^2 & = \sum_{h\in\N_0} \braket{h}^{2s} \sup_{m,n\in\N_0 \atop \abs{m-n}=h} \|\fM_{[n]}^{[m]}\|_{{\rm HS}}^2\\
	& \leq \sum_{h\in\N_0} \braket{h}^{2s} \Big(\sup_{ |m-n|=h \atop 0\leq\abs n < N_0} \|\fM_{[n]}^{[m]}\|_{\rm HS}^2+\sup_{ |m-n|=h\atop \abs n \geq N_0} \|\fM_{[n]}^{[m]}\|_{\rm HS}^2\Big)\\
	& \leq \sum_{h\in\N_0} \braket{h}^{2s} \frac{16}{\braket{h}^{2s_1}}( 1+C_{s_1} \braket{N_0}^{2s_1} )\leq C(s,s_1,N_0)<\infty\,,
	\end{split}
	\end{equation*}
	which implies \eqref{sdecay_M}. The equation $| \fM^T|_{s;M} = |\fM|_{s;M}$ follows straightforward.
\end{proof}
%
\subsection{Pseudodifferential operators embed into off-diagonals decaying operators.}
We recall the definition of $s$-decay norms in Definitions \eqref{def:sdecay},
where the matrix representation of a linear operator $A(\vf)=(A_{[n]}^{[n']}(\vf))_{n,n'\in\N_0}$ is constructed along the $L_q$-eigenfunction basis. At the same time, we write 
$$A (\vf)  = (  \tA_{[n]}^{[n']}(\vf))_{n,n'\in\N_0}\,, \quad \text{where} \quad \tA _{j}^{j'}(\vf):= \big( A(\vf) e_{j},e_{j'} \big)_{L^2} $$
is the
 matrix representation of $A$ along the exponential basis.
	Let $\fM_{[n]}^{[m]} $ be as in Definition \ref{M_basis_ch}. The block representations of a linear operator $A$ with respect to the $L_q$-eigenfunction basis and to the exponential basis are related by
\begin{equation}\label{eq:rule1}
	A_{[n]}^{[n']} (\ell)= \sum_{p,p'\in\N_0} \fM_{[p]}^{[n']} \tA_{[p']}^{[p]} (\ell) (\fM^T)_{[n]}^{[p']} \,, \quad \forall\,n,n'\in\N_0\,, \ \ell\in\Z^{\nu}\,.
\end{equation}
The next result tells that matrices of pseudodifferential operators in the class $\Ops^{m}$ with smoothing orders $m\leq 0$ embed into the class $\cM_{s}(\alpha,\beta)$ of linear matrices with smoothing $s$-decay for suitable $\alpha,\beta\geq0$.
\begin{thm}\label{thm:embed_PSDO_decay}
	Let $A^d \in \Ops^{-\alpha}(\omega)$ and $A^o (\omega)\in \Ops^{-\beta}$, Lipschitz continuous with respect to $\omega\in\Omega\subseteq R_{\tM}$, such that $A^d=(A^d)^*$ and $\bar A^o = (A^o)^*$. Define the operator matrix $\bA$ as in \eqref{eq:Amatrix}.
	Then there exists $\sigma_{\fM}=\sigma_{\fM}(s_0,\alpha,\beta)$ such that, for any $s \geq s_0$, we have  $\bA\in \cM_{s}(\alpha,\beta)$, with estimates
\begin{equation}
	\| \bA \|_{s,\alpha,\beta}^{\lip(\tw)} \lesssim_{s,\alpha,\beta} \| A^d\|_{-\alpha,s+\sigma_{\fM},0}^{\lip_{\tw}} +  \| A^o\|_{-\beta,s+\sigma_{\fM},0}^{\lip_{\tw}}\,.
\end{equation}
\end{thm}
The claim of the theorem above is deduced by the result of the following lemma.
	\begin{lem}\label{tecn1}
		Let $\alpha_1,\alpha_2\in\R$ and $\mu\leq 0$ such that $\alpha_1+\alpha_2+\mu\leq 0$ and let  $A(\omega)\in\Ops^{\mu}$ be Lipschitz continuous with respect to $\omega\in\Omega\subseteq R_{\tM}$. Then, there exists $\sigma_{\fM}=\sigma_{\fM}(s_0,\alpha_1,\alpha_2)>0$ such that, for  any $s\geq s_0$, the operator
		$ \braket{D}^{\alpha_1} A \braket{D}^{\alpha_2}$, with $\braket{D}$ defined as in \eqref{diago.Dx},
		belongs to $\cM_{s}$, with estimates
		\begin{equation}\label{stima.maledetta}
			| \braket{D}^{\alpha_1} A \braket{D}^{\alpha_2} |_{s}^{\lip(\tw)} \lesssim_{s,\alpha_1,\alpha_2} \| A \|_{\mu,s+\sigma_{\fM},0}^{\lip(\tw)}\,.
		\end{equation}
	\end{lem}
	\begin{proof}
		Since $\braket{D}$, defined as in \eqref{diago.Dx}, is clearly independent of parameters, we assume without loss of generality that $A=\Op(a(\vf,x,\xi))$ is independent of $\omega\in\Omega$.
		 For any $j,j'\in\Z$ and $\ell\in\Z^\nu$, we have
		\begin{equation*}
		\begin{aligned}
		{\tA}_j^{j'}(\ell) :=& \frac{1}{(2\pi)^\nu}\int_{\T^{\nu}\times\T}a(\vf,x,D)[e^{\im\,jx}] e^{-\im\,j'x}e^{-\im\,\ell\cdot\vf}\wrt\vf\wrt x \\
		= & \frac{1}{(2\pi)^\nu}\int_{\T^{\nu}\times\T}a(\vf,x,j)e^{\im(j-j')x} e^{-\im\,\ell\cdot\vf}\wrt\vf\wrt x\,.
		\end{aligned}
		\end{equation*}
		Let $n,n'\in\N_0$ By integrating by parts both  in $\vf\in\T^\nu$ and in $x\in\T$ (if $j\neq j'$,  with $j=\pm n$ and $j'=\pm n'$)  and recalling Definition \ref{defn.pseudo.norm}, we obtain that
		\begin{equation}\label{cases}
		\begin{aligned}
		\| & \tA_{[n]}^{[n']}(\ell)\|_{\rm HS}^2  = | \tA_{-n}^{-n'}(\ell) |^2 + | \tA_{-n}^{n'}(\ell) |^2+| \tA_{n}^{-n'}(\ell) |^2+| \tA_{n}^{n'}(\ell) |^2 \\
		& \leq \begin{cases}
		2\braket{\ell}^{-2N}(\braket{n'+n}^{-N}+ \braket{n'-n}^{-N})^2 \braket{n}^{2\mu} \| A \|_{\mu,N,0}^2  & \text{ if } n\neq n' \\
		2 \braket{\ell}^{-2N} (1+\braket{n'+n}^{-N})^2 \braket{n}^{2\mu}   \| A \|_{\mu,N,0}^2& \text{ if } n=n'
		\end{cases}\\
		& \leq 4 \braket{\ell,n'-n}^{-2N}\braket{n}^{2\mu} \| A \|_{\mu,N,0}^2 \,,
		\end{aligned}
		\end{equation}
		for some $N=N(s,\alpha_1,\alpha_2,\nu)\in \N$ to determine. By \eqref{maledetta}, \eqref{eq:rule1}, Corollary \ref{cor:decay_class} and Cauchy-Schwartz inequality, we compute
		\begin{equation}
			\begin{aligned}
				& | \braket{D}^{\alpha_1} A \braket{D}^{\alpha_2} |_s^2 = \sum_{h\in \N_0 \atop \ell\in\Z^\nu } \braket{\ell,h}^{2s} \sup_{|n-n'|=h} \| \la n \ra^{\alpha_1} \la n' \ra^{\alpha_2} A_{[n]}^{[n']}(\ell) \|_{\rm HS}^2\\
				& = \sum_{h\in \N_0 \atop \ell\in\Z^\nu }  \sup_{|n-n'|=h} \Big\| \la n \ra^{\alpha_1} \la n' \ra^{\alpha_2}  \la \ell,h \ra^{s} \sum_{p,p'\in\N_0} \fM_{[p]}^{[n']} \tA_{[p']}^{[p]} (\ell) (\fM^T)_{[n]}^{[p']}  \Big\|_{\rm HS}^2\\
				& = \sum_{h\in \N_0 \atop \ell\in\Z^\nu }  \sup_{|n-n'|=h} \Big\| \sum_{p,p'\in\N_0} \tfrac{\la n \ra^{\alpha_1} \la n' \ra^{\alpha_2}  \la \ell,h \ra^{s} \la n'-p \ra^N \fM_{[p]}^{[n']} \,\la \ell,p-p' \ra^N \la p \ra^{\mu} \tA_{[p']}^{[p]} (\ell)\, \la p'-n \ra^N(\fM^T)_{[n]}^{[p']} }{\la p \ra^{\mu} \la n'-p \ra^N \la \ell,p-p' \ra^N \la p'-n \ra^N  } \Big\|_{\rm HS}^2\\
				& \lesssim_{N}  | \fM|_{N;M}^4  \| A \|_{\mu,N,0}^2 \sum_{h\in \N_0 \atop \ell\in\Z^\nu }  \sup_{|n-n'|=h}  \tG(\ell,n,n')\,,
			\end{aligned}
		\end{equation}
	where, by Peetre inequality and the condition $\alpha+\beta+\mu\leq 0$,
	\begin{equation}
		\begin{aligned}
			& \tG(\ell,n,n') := \sum_{p,p'\in\N_0} \frac{\la n \ra^{2\alpha_1} \la n' \ra^{2\alpha_2}  \la \ell,n'-n \ra^{2s}  }{\la p \ra^{2\mu} \la n'-p \ra^{2N} \la \ell,p-p' \ra^{2N} \la p'-n \ra^{2N}  } \\
			& \lesssim_{\alpha_1,\alpha_2} \sum_{p,p'\in\N_0} \frac{  \la \ell,n'-n \ra^{2s}  }{ \la n'-p \ra^{2(N+|\alpha_1|)} \la p- n \ra^{2|\alpha_2|} \la \ell,p-p' \ra^{2N} \la p'-n \ra^{2N}  }\\ 
			& \lesssim_{s,N,\alpha_1,\alpha_2} \sum_{p\in\N_0} \frac{\la\ell, n'-n\ra^{2s}}{ \la n'-p \ra^{2(N+|\alpha_1|)}  \la \ell,p-n \ra^{2(N+|\alpha_2|-1) }  }\\
			&  \lesssim_{s,N,\alpha_1,\alpha_2} \frac{1}{\la \ell,n'-n \ra^{2(N + \max\{| \alpha_1|,|\alpha_2|\}-s-2)}}\,.
		\end{aligned}
	\end{equation}
	Therefore, we deduce estimate \eqref{stima.maledetta} by choosing any $N=N(s,\alpha_1,\alpha_2,\nu)\in\N$ such that $N + \max\{| \alpha_1|,|\alpha_2|\}-s-2\geq s_0 > \frac{\nu+1}{2}$. For instance, fix $N:=\max\big\{ \lfloor s+s_0+2 - \max\{| \alpha_1|,|\alpha_2|\}\rfloor , \lfloor s\rfloor +1 \big\}\in\N$. In particular, the loss of regularity $\sigma_{\fM}>0$ is given by $\sigma_{\fM}:=N-s$, with $N$ fixed as before and $\sigma_{\fM}$ depending on $s$ only with respect to its fractional part $\lfloor s\rfloor +1  -s \in (0,1]$.
	\end{proof}
	\begin{proof}[Proof of Theorem \ref{thm:embed_PSDO_decay}]
	The thesis now follows by applying  Lemma \ref{tecn1} with the operators $A^d \in \Ops^{-\alpha}$ and $A^o\in \Ops^{-\beta}$ instead of a generic $A$ and inserting everything into the definition in \eqref{eq:sdecay_matrix_norm}.
\end{proof}

\section{The Magnus normal form}\label{sec:magnus}

The difficulty in treating equation \eqref{eq:KG_s} is that it is not perturbative in the size of the potential, so standard KAM techniques do not apply directly.
To deal with this problem, we perform a change of coordinates, adapted to fast oscillating systems,   which  puts \eqref{eq:KG_s} in a perturbative setting.
As done in \cite{FM19}, we refer to this procedure as \emph{Magnus normal form}.

To begin with, we recall  the Pauli matrices notation. Let us introduce
\begin{equation}\label{eq:pauli}
\bsigma_1=\left(\begin{matrix}
0 & {\rm Id} \\ {\rm Id} & 0
\end{matrix}\right), \; \ \ \  \bsigma_2=\left(\begin{matrix}
0 & -\im \\ \im & 0
\end{matrix}\right), \; \ \  \  \bsigma_3=\left(\begin{matrix}
{\rm Id} & 0 \\ 0 & -{\rm Id}
\end{matrix}\right) , 
\end{equation}
and, moreover, define 
$$
\bsigma_4:=
\left(\begin{matrix}
{\rm Id} & {\rm Id} \\ -{\rm Id} & -{\rm Id}
\end{matrix}\right) \ , \quad
\b1:= \left(\begin{matrix}
{\rm Id} & 0 \\ 0 & {\rm Id}
\end{matrix}\right) , 
\quad
\b0 := \left(\begin{matrix}
0 & 0 \\ 0 & 0
\end{matrix}\right). 
$$
Using Pauli matrix notation, equation  \eqref{eq:KG_matrix} reads as
\begin{equation}\label{eq:KG_s}
\begin{aligned}
\im\dot\phi(t)= & \bH(t)\phi(t):=(\bH_0+\bW(\omega t))\phi(t) \ , \\
&	\bH_0:=B\bsigma_3,\ \ \  \bW(\omega t):=\frac{1}{2}\, B^{-1/2}V(\omega t) B^{-1/2} \bsigma_4 \ .
\end{aligned}
\end{equation}
Note that, by assumption  ({\bf V}), one has $ V \in \Ops^{0} $;
therefore, Theorem \ref{cor:pseudo_sqrt} and Lemma \ref{pseudo_compo} imply that 
\begin{equation}
\label{ass_magnus}
B\in \Ops^1  \quad \mbox{and} \quad B^{-1/2}VB^{-1/2} \in \Ops^{-1} \,.
\end{equation} 

The main result of the section is the following:
\begin{thm}\label{lem:magnus}
	{\bf (Magnus normal form).}
	Let $\tw>0$ be fixed. For any $\gamma_0 \in (0,1)$, there exist a set $\Omega_0 \subset R_\tM \subset \R^\nu$ and a constant $c_0 >0$ (independent of $\tM$),  with
	\begin{equation}
	\label{meas_omega0}
	\frac{\meas (R_{\tM}\backslash\Omega_0)}{\meas(R_\tM)} \leq c_0 \gamma_0 , 
	\end{equation}
	such that the following holds true. There exists a  time dependent change of coordinates $\vf(t) = e^{- \im \bY(\omega;\omega t)} \psi(t)$, where  $ \bY(\omega;\omega t)=Y(\omega;\omega t)\bsigma_4 $ and $Y \in \Ops^{-1}$,
	such that, for any $\omega\in\Omega_0$, equation \eqref{eq:KG_s} is conjugated to 
	\begin{equation}\label{eq:system_kg_pert}
	\im \dot\psi(t)=\widetilde{\bH}(t)\psi(t), \quad \widetilde{\bH}(t):=\bH_0+\bV(\omega; \omega t)  \ ,
	\end{equation}
	defined for any $\omega \in R_{\tM}$, where 
	\begin{equation}
	\label{eq:V}
	\bV(\omega;\vf)= \begin{pmatrix}
	V^{d}(\omega;\vf) & V^{o}(\omega;\vf)\\ -\oV^{o}(\omega;\vf) & -\oV^{d}(\omega;\vf)
	\end{pmatrix},  \ \ \quad
	\begin{aligned}
		&V^d \in \Ops^{-1}\,, \quad 	[V^d]^* = V^d \,, \\
		&  V^o\in \Ops^{0}\,, \quad [V^o]^* = \bar{V^o}\,.
	\end{aligned}
	\end{equation}
	Moreover, for any fixed $\delta\in\N_0$, there exists $\sigma_0:= \sigma_0(\delta) :=\sigma_0(\delta,\tau,\nu)>0$ such that, for any $s_0\leq s \leq S-\sigma_0$,
	\begin{equation}\label{est:VdVo}
		\begin{aligned}
			& \| Y \|_{-1,s,\delta}^{\lip(\tw)} \lesssim_{s,\delta} (\gamma_0\,\tM)^{-1}\,,
			& \| V^d \|_{-1,s,\delta}^{\lip(\tw)}+ \| V^o \|_{0,s,\delta}^{\lip(\tw)} \lesssim_{s,\delta} (\gamma_0\,\tM)^{-1}\,.
		\end{aligned}
	\end{equation}
%
\end{thm}
\begin{proof}
	The proof is split into two parts: one for the formal algebraic construction, which is essentially identical to the one in Lemma 3.1 in \cite{FM19}; the other for checking the estimates for the pseudodifferential operators that we have found. The proof of the measure estimate \eqref{meas_omega0} is postponed to Proposition \ref{prop:diop_magnus}.
	\\[1mm]
	\noindent \textbf{Step I).} 
	The change of coordinates 
	$\phi(t) = e^{- \im \bY(\omega;\omega t)} \psi(t)$ conjugates \eqref{eq:KG_s} to  $\im \d_t \psi(t) = \widetilde{\bH}(t) \psi(t)$, where the Hamiltonian $\widetilde{\bH}(t)$ is given by
	(see Lemma 3.2 in \cite{Bam16I}) 
	\begin{align}
		\label{eq:magnus_1}
		\widetilde{\bH}(t) & = \LieTr{\bY(\omega ;\omega t)}{\bH(t)}-\int_{0}^1\LieTr{s\bY(\omega;\omega t)}{\dot\bY(\omega ;\omega t)}\wrt s\,.
	\end{align}
	Expanding \eqref{eq:magnus_1} in commutators we have
	\begin{equation}\label{eq:magnus_2}
	\widetilde{\bH}(t) = \bH_0 +\im[\bY,\bH_0]-\tfrac{1}{2}[\bY,[\bY,\bH_0]]+\bW -\dot\bY +\bR \ ,
	\end{equation}
	where the remainder $\bR$ of the expansion is given in integral form by
	\begin{equation}\label{eq:magnus_rem}
	\begin{split}
	\bR   := & \int_{0}^1\frac{(1-s)^2}{2}\LieTr{s\bY}{\ad_\bY^3(\bH_0)}\wrt s\\ & + \im \int_0^1\LieTr{s\bY}{[\bY,\bW]}\wrt s -\im\int_0^1(1-s)\LieTr{s\bY}{[\bY,\dot\bY]}\wrt s . 
	\end{split}
	\end{equation}
	From the properties of the Pauli matrices, we note that $\bsigma_4^2=\b0$. This means that the terms in \eqref{eq:magnus_rem} involving $\bW$ and $\dot\bY$ are null, and the remainder is given only by
	\begin{equation}\label{eq:magnus_rem2}
	\bR = \int_{0}^1\frac{(1-s)^2}{2}\LieTr{s\bY}{\ad_\bY^3(\bH_0)}\wrt s . 
	\end{equation}
	We ask $\bY$ to solve the homological equation 
	\begin{equation}\label{hom.eq}
	\b0 = \bW-\dot\bY= \big(\tfrac{1}{2}\,B^{-1/2} V(\omega t) B^{-1/2} - \dot Y(\omega; \omega t) \big) \bsigma_4 .
	\end{equation}	
	By \eqref{ass_magnus}, let $\tfrac12 B^{-1/2} \wh V(\ell) B^{-1/2} = \Op (w(\vf,x,\xi)) \in \Ops^{-1}$, where $w(\vf,x,\xi)\in S^{-1}$ is  independent of $\omega$, as are both $B$ and $V$.
	Expanding in Fourier coefficients with respect to the angles and recalling \eqref{assumption V}, the solution $Y(\omega;\vf)= \Op(p(\omega;\vf,x,\xi))$  of the homological equation \eqref{hom.eq} has symbol satisfying
	\begin{equation}\label{eq:magnus_sol_hom}
		\whp(\omega;\ell,x,\xi)=
		\begin{cases}
			\frac{1}{\im\,\omega\cdot \ell}\,\whw(\ell,x,\xi) & \text{ for } \ell\in\Z^{\nu}\backslash\{0\}\,, \\
			0 & \text{ for } \ell=0\,,
		\end{cases}
	\end{equation}
which is defined for any $\omega$ in the set of Diophantine frequency vectors
\begin{equation}\label{eq:magnus_omega}
	\Omega_0 = \Omega_0(\gamma_0, \tau_0) :=\Big\{\omega\in R_{\tM} \,:\, |\omega\cdot \ell|\geq \frac{\gamma_0\,\tM}{\braket{\ell}^{\tau_0}} \quad \forall\,\ell\in\Z^{\nu}\backslash\{0\} \Big\} \ .
\end{equation}
for some $\gamma_0>0$ and $\tau_0>\nu-1$.
In Proposition \ref{prop:diop_magnus} below we will prove that \eqref{meas_omega0} holds
for  some constant $c_0>0$  independent of $\tM$ and $\gamma_0$.\\
	It remains to compute the terms involving $\bH_0$ in \eqref{eq:magnus_2} and \eqref{eq:magnus_rem2}. Using again the structure of the Pauli matrices, we get
	\begin{equation}\label{eq:ad_X1}
	\ad_\bY(\bH_0):= \im [Y\bsigma_4,B\bsigma_3]  = \im[Y,B]\b1-\im [Y,B]_{\rm a}\bsigma_1 \ ,
	\end{equation}
	where we have denoted by $[Y,B]_{\rm a}:= YB + BY$ the anticommutator. A similar computation shows that
	\begin{equation}\label{eq:ad_X2}
	\begin{aligned}
	\ad_\bY^2(\bH_0) & := - [Y\bsigma_4,[Y\bsigma_4,B\bsigma_3]]
	= 4YBY\bsigma_4  \,,
	\end{aligned}
	\end{equation}
	which also implies that $\ad_\bY^3(\bH_0) = \b0$.
	This shows that $\bR\equiv \b0$ and, imposing \eqref{eq:magnus_sol_hom} in \eqref{eq:magnus_2}, together with \eqref{eq:ad_X1}-\eqref{eq:ad_X2}, we obtain $\widetilde{\bH}(t) = \bH_0 +\bV(\omega;\omega t) $,
	where $\bV$ is as in \eqref{eq:V}, with
	\begin{equation}\label{eq:magnus_4}
	\begin{split}
	V^d(\omega;\vf) &:= \im[Y(\omega;\vf),B]+2Y(\omega;\vf)BY(\omega;\vf) \ ,\\
	V^o(\omega ;\vf) &:= -\im[Y(\omega;\vf ),B]_{\rm a}+2Y(\omega;\vf)BY(\omega;\vf) \ .
	\end{split}
	\end{equation}
	\noindent \textbf{Step II).} We show now that $Y, V^d$ and $V^o$, defined in \eqref{eq:magnus_sol_hom} and \eqref{eq:magnus_4} respectively, are pseudodifferential operators in the proper classes, provided $\omega$ is sufficiently nonresonant, and that they satisfy the estimates \eqref{est:VdVo}.
	We start with the generator of the transformation $\bY$.  First, we extend the definition of the symbol $p$ in \eqref{eq:magnus_sol_hom} to all the parameters $\omega\in R_{\tM}$. Denoting such extension with the same name, we set
	\begin{equation}
		p(\omega;\vf,x,\xi) := \sum_{\ell\in\Z^\nu}\frac{\chi \big(\omega\cdot\ell\, \rho_{\ell}^{-1}\big)}{ \im\, \omega\cdot\ell}\, \whw(\ell,x,\xi) e^{\im\,\ell\cdot \vf} \,, \quad \rho_{\ell}:= \gamma_0\tM \braket{\ell}^{-\tau_0}\,,
	\end{equation}
where $\chi$ is an  even, positive $\cC^\infty$ cut-off such that
\begin{equation}\label{cutoff}
	\chi(\xi) = \begin{cases}
		0 & \text{ if } \ \abs\xi\leq \frac13 \\
		1 & \text{ if } \ \abs\xi \geq \frac23 
	\end{cases}\,, \qquad \pa_\xi \chi(\xi) >0   \quad \forall\,\xi \in (\tfrac13,\tfrac23) \,.
\end{equation}
	Let $\delta\in\N_0$ be arbitrary and let $s\in [s_0,S-\sigma_0]$, with $\sigma_0>0$ to determine. Then, by \eqref{eq:magnus_omega} and using that $w\in S^{-1}$, we obtain, for any $ 0\leq \beta \leq \delta$,
	\begin{equation}
		\begin{aligned}
			\|\pa_\xi^\beta p(\omega;\,\cdot\,,\cdot\,,\xi) \|_s \leq \frac{1}{\gamma_0\,\tM} \|\pa_\xi^\beta w(\omega;\,\cdot\,,\cdot\,,\xi) \|_{s+\tau_0} \lesssim_{s,\beta}  \frac{1}{\gamma_0\,\tM} \braket{\xi}^{-1-\beta}\,.
		\end{aligned}
	\end{equation}
	It implies that $\| Y(\omega)\|_{-1,s,\delta}^{\infty} \lesssim_{s,\delta} \frac{1}{\gamma_0\,\tM}$. To compute the Lipschitz seminorm, using the  notation $\Delta_{12}f(\omega)	= f(\omega_1)-f(\omega_2) $,
	with $\omega_1, \omega_2\in R_{\tM}$, $\omega_1\neq \omega_2$
	note that
	\begin{equation}\label{frac.lip}
		\begin{aligned}
			\Delta_{12}\Big( \frac{\chi \big(\omega\cdot\ell\, \rho_{\ell}^{-1}\big)}{ \im\, \omega\cdot\ell} \Big) & =  \frac{\Delta_{12}\big( \chi \big(\omega\cdot\ell\, \rho_{\ell}^{-1}\big) \big)}{\im\,\omega_1\cdot \ell} + \chi \big(\omega_2\cdot\ell\, \rho_{\ell}^{-1}\big) \,\Delta_{12}\Big(\frac{1}{\im\,\omega\cdot\ell} \Big)\\
			 & =  \frac{\Delta_{12}\big( \chi \big(\omega\cdot\ell\, \rho_{\ell}^{-1}\big) \big)}{\im\,\omega_1\cdot \ell} - \chi \big(\omega_2\cdot\ell\, \rho_{\ell}^{-1}\big) \,\frac{(\omega_1-\omega_2)\cdot\ell}{\im\,(\omega_1\cdot\ell)(\omega_2\cdot\ell)}\,.
		\end{aligned}
	\end{equation}
	Since $w\in S^{-1}$ is independent of $\omega$, by \eqref{frac.lip} and arguing as before, we get
	\begin{equation}
		\begin{footnotesize}
					\begin{aligned}
				\frac{\|\Delta_{12}\pa_\xi^\beta p(\omega;\,\cdot\,,\cdot\,,\xi) \|_{s-1}}{| \omega_1-\omega_2 |} &\lesssim \Big( \frac{1}{\gamma_0\,\tM}  \|\pa_\xi^\beta w(\omega;\,\cdot\,,\cdot\,,\xi) \|_{s-1+\tau_0} +  \frac{1}{\gamma_0^2\,\tM^2}  \|\pa_\xi^\beta w(\omega;\,\cdot\,,\cdot\,,\xi) \|_{s+2\tau_0} \Big) \\
				& \lesssim_{s,\beta} \frac{1}{\gamma_0^2\,\tM} \braket{\xi}^{-1-\beta} \,.
			\end{aligned}
		\end{footnotesize}
	\end{equation}
	It implies that $\| Y(\omega) \|_{-1,s,\delta}^{\rm lip} \lesssim_{s,\delta} \frac{1}{\gamma_0^2\,\tM} $ and we conclude that $Y=\Op(p)\in \Ops^{-1}$, satisfying the estimate 
	\begin{equation}\label{est.X.PD}
		\| Y \|_{-1,s,\delta}^{\lip(\tw)}\lesssim_{s,\delta} \frac{\max\{1,\tw/\gamma_0\}}{\gamma_0\,\tM}
	\end{equation}
	for any $\delta\in\N_0$ and $s_0\leq s \leq S-2\tau_0 $.
	We finally move to analyse $V^d$ and $V^o$ in \eqref{eq:magnus_4}. By Lemma \ref{pseudo_compo}, Lemma \ref{pseudo_commu}, Theorem \ref{cor:pseudo_sqrt} and estimate \eqref{est.X.PD}, it follows that $V^d \in \Ops^{-1}$ and $V^o \in \Ops^{0}$ with the claimed estimates \eqref{est:VdVo} hold with $\sigma_0:= 2\tau_0 + 2\delta+1$.\\
	Finally, $V$ is a real selfadjoint operator, simply because it is a real bounded potential, and therefore  $V^* = V = \bar{V}$. 
	It follows by Remark \ref{rem:coeff.mat} and 
	the explicit expression \eqref{eq:magnus_sol_hom} that $Y^* = Y = \bar{Y}$.
	Using these properties one verifies by a direct computation that $[V^d]^* = V^d$ and $[V^o]^* = V^o$.
\end{proof}

%
%
%
%

We conclude with the proof the measure estimate of the set $\Omega_0$ in \eqref{eq:magnus_omega}.
\begin{prop}\label{prop:diop_magnus}
	For $\gamma_0>0$ and $ \tau_0 >\nu -1$, 	 the set $\Omega_0$ defined in \eqref{eq:magnus_omega} fulfills  \eqref{meas_omega0}.
\end{prop}
\begin{proof}
	For any $\ell \in \Z^\nu \setminus\{0\}$, define  the sets
	$\cR_\ell^0=\cR_\ell^0(\gamma_0,\tau_0):=\big\{ \omega\in R_{\tM} \,:\, \abs{\omega\cdot \ell}< \frac{\gamma_0\,\tM}{\braket{\ell}^{\tau_0}}  \big\}$.
	By Lemma \ref{tecnico}
	$		|\cR_\ell^0|\lesssim
	\frac{\gamma_0}{|\ell|^{\tau_0+1}}\tM^{\nu}$.
	Therefore the set  $\cG:=\bigcup_{\ell\neq 0}\cR_\ell^0$ has measure bounded by $\abs \cG \leq C\gamma_0\,\tM^{\nu}$, which proves the claim.
\end{proof}

\section{The KAM reducibility transformation}\label{sec:kam}

In this section we perform the KAM reduction of the operator
\begin{equation}\label{initial.H.KAM}
	\begin{aligned}
		& \bH^{(0)}(\omega;t):= \wt\bH(\omega;t) := \bH_0^{(0)} + \bV^{(0)}(\omega;\omega t) \,,\\
		& \bH_0^{(0)}:= \bH_0\,, \quad \bV^{(0)}(\omega;\omega t):=\bV(\omega t ; \omega)
	\end{aligned}
\end{equation}
as found in Theorem \ref{lem:magnus}, with the potential $\bV^{(0)}(\omega;\omega t)$ being perturbative, in the sense that the smallness of its norm is controlled by the size $\tM$ of the frequency vector $\omega$. The result of this reduction is a Hamiltonian time-independent and block-diagonal, as stated in Theorem \ref{cor:iter_flow}. This reduction is based on the KAM iteration in Theorem \ref{thm:iter_lemma}. At each step of such iteration, we ask the parameter $\omega\in R_\tM\subset\R^\nu$ to  satisfy second order non-resonance Melnikov conditions on the normal form obtained at the previous step, namely bounds \eqref{eq:Omega_p} on the inverse of the finite dimensional operators \eqref{G.block.p-1}. Such conditions are \emph{balanced} with respect to $\alpha\in(0,1)$ between the gain in regularity $\braket{n\pm n'}^\alpha$, needed for preserving the scheme, and the loss in size $\tM^\alpha$, which would prevent the imposition of the smallness condition \eqref{eq:small_cond_kam} when $\alpha=1$. The construction of these non-resonance conditions and the proof that they hold for most values of the parameter $\omega\in R_\tM$ is finally proved in Section \ref{sec:melnikov}.

Given $\tau >0$ and $\tN_0\in\N$ we define the parameters
\begin{equation}\label{eq:par_kam}
	\begin{aligned}
		&\tN_{-1}:=1 \,, \quad \tN_\tp:=\tN_0^{\chi^\tp} \,, \quad \chi:=3/2 \,, \quad \tp\in\N_0\,, \\
		&  \varrho=\varrho(\tau):= 6\tau + 4 \,, \quad \beta=\beta(\tau):= \varrho(\tau)+1 \,, \quad \Sigma(\beta):= \sigma_0+\sigma_{\fM} + \beta\,,
	\end{aligned}
\end{equation}
where $\sigma_0,\sigma_{\fM}>0$ are as in Theorem \ref{lem:magnus} and Theorem \ref{thm:embed_PSDO_decay}, respectively.
For the purposes of the KAM scheme, it is more convenient to work with operators of type $\cM_{s}(\alpha,\beta)$. Of course, as we have seen in Section \ref{sec:embe}, pseudodifferential operators belong to such a class.
\begin{lem}\label{lem.magnus_size}
	{\bf (Initialization of the KAM reducibility).}
	For any $s\in [s_0,s_0 + \Sigma(\beta)]$, the operator $\bV^{(0)}(\omega):=\bV(\omega) $ defined in \eqref{eq:V} belongs to $\cM_{s}(1,0)$ with estimate  
	\begin{equation}
	\abs{\bV^{(0)}}_{s, 1, 0}^{\wlip{\tw}} \leq C_s (\gamma_0\,\tM)^{-1}\,,
	\end{equation}
	where  $C_s >0$ is independent of $\tM$.
\end{lem}
\begin{proof}
	The claimed estimate follows directly from Theorem \ref{thm:embed_PSDO_decay} and the estimate \eqref{est:VdVo} in Theorem \ref{lem:magnus}.
\end{proof}
From now on, we choose as Lipschitz weight $\tw:=\gamma/\tM^\alpha$ and, abusing notation, we introduce the following quantities:
\begin{equation}\label{eq:delta_sigma_p}
	\delta_{s}^{(\tp)}:= |\bV^{(\tp)}|_{s,\alpha,0}^\wlip{\gamma} :=| \bV^{(\tp)} |_{s,\alpha,0}^{\lip(\gamma/\tM^{\alpha})}
	, \quad \tp\in\N_0 \ , \quad s \in[s_0,s_0+\Sigma(\beta)]\,.
\end{equation}
Furthermore, {\bf we fix once for all $\alpha \in (0,1)$}. We  introduce  the indexes sets:
\begin{equation}\label{eq:indices_meln}
	\begin{aligned}
				&\cI^-  :=\big\{(\ell,|j|,|j'|)\in\cI^+ \,:\, (\ell,|j|,|j'|)\neq(0,|j|,|j|)\big\} \,, \\
		& \cI^+ :=\Z^{\nu}\times\N_0\times\N_0  \,, \quad \cI_\tN^\pm := \cI^\pm \cap \{ | \ell |\leq \tN \}\,, \quad \tN\geq 1\,.
	\end{aligned}
\end{equation}
%

\begin{thm}\label{thm:iter_lemma}
	{\bf (Iterative Lemma).}
	There exists $\tN_0=\tN_0(\tau,\nu,s_0)\in\N$ such that, if
	\begin{equation}\label{eq:small_cond_kam}
		C_{s_0}\tN_0^{\Lambda}\frac{\tM^\alpha}{\gamma}|\bV^{(0)}|_{s_0+\beta,\alpha,0}^\wlip{\gamma}\leq 1 \,, \quad \Lambda:=2\tau +2+\varrho\,,
	\end{equation} 
	the following holds inductively for any $\tp\in\N_0$:
	\\[1mm]
	\noindent ${\bf(S1)_{\tp}}$ 
	There exists a Hamiltonian operator
	\begin{equation}\label{eq:Hp.stat}
		\bH^{(\tp)}(\omega;t):=\bH_0^{(\tp)}(\omega)+\bV^{(\tp)}(\omega;\omega t)
	\end{equation}
	defined for all $\omega \in R_{\tM}$, where	$\bH_0^{(\tp)}(\omega)$ is time-independent and block diagonal
	\begin{equation}\label{eq:S2_p_def}
		\begin{aligned}
			&\bH_0^{(\tp)}   =\diag\Big\{ \left.H_0^{(\tp)}\right._{[n]}^{[n]}(\omega) \,:\, n\in\N_0 \Big\} \bsigma_3\,,
		\end{aligned}
	\end{equation}
	such that,  for each $n\in\N_0$, the block $\left.H_0^{(\tp)}\right._{[n]}^{[n]}(\omega)$ is self-adjoint, with estimate for any $\tp\geq 1$
\begin{align}
	& \sup_{n\in\N_0}\braket{n}^\alpha\,\big\|\left.H_0^{(\tp)}\right._{[n]}^{[n]}-\left.H_0^{(0)}\right._{[n]}^{[n]}\big\|_{{\rm HS}}^\wlip{\gamma} \lesssim_{s_0,\beta}  (\gamma_0\,\tM)^{-1}\,, \label{eq:S2_p_est1}  \\
	& \sup_{n\in\N_0}	\braket{n}^\alpha\,\big\|\left.H_0^{(\tp)}\right._{[n]}^{[n]}-\left.H_0^{(\tp-1)}\right._{[n]}^{[n]}\big\|_{{\rm HS}}^\wlip{\gamma} \lesssim_{s_0,\beta}  (\gamma_0\,\tM)^{-1}  \tN_{\tp-2}^{-\varrho}\,.\label{eq:S2_p_est}
\end{align}
	 For any $s\in[s_0,s_0+\Sigma(\beta)]$, the remainder 
	 \begin{equation}
	 \bV^{(\tp)}(\omega;\omega t)= \begin{pmatrix}
	 V^{d,(\tp)}(\omega;\omega t) & V^{o,(\tp)}(\omega;\omega t) \\ -\overline{ V^{o,(\tp)}}(\omega;\omega t) & - \overline{ V^{d,(\tp)}}(\omega;\omega t)
	 \end{pmatrix}
	 \end{equation}
	 belongs to $\cM_{s}(\alpha,0) $, with estimates
	\begin{equation}\label{eq:S3_p}
		\delta_{s}^{(\tp)} \leq \delta_{s+\beta}^{(0)}\, \tN_{\tp-1}^{-\varrho}\,, \quad 
		\delta_{s+\beta}^{(\tp)}\leq \delta_{s+\beta}^{(0)} \,\tN_{\tp-1}\,.
	\end{equation}
\noindent ${\bf(S2)_{\tp}}$ 
	Define the sets $\Omega_{\tp}$ by $\Omega_0:= \Omega_0(\gamma_0,\tau_0)\subset R_{\tM}$ as in \eqref{eq:magnus_omega} and, for all $\tp \geq 1$,
	\begin{equation}\label{eq:Omega_p}
				\begin{aligned}
				\Omega_{\tp} := \Omega_{\tp}(\gamma,\tau) := \Big\{ \omega \in \Omega_{\tp-1} \,:\, & \big\|\big( \tG_{\ell,,n,n'}^{\pm,(\tp-1)} (\omega)\big)^{-1}\big\|_{\Op(n,n')} \leq \frac{2\tN_{\tp-1}^\tau}{\gamma}\frac{\tM^\alpha}{\braket{n\pm n'}^\alpha} \\
				&      \quad \quad \quad \quad  \quad \quad  \quad  \quad \quad  \ \  \forall \,(\ell,n,n')\in \cI_{\tN_{\tp-1}}^\pm \Big\}\,,
			\end{aligned}
\end{equation}
where, recalling the notation in \eqref{MRML}, the operator $\tG_{\ell,,n,n'}^{\pm,(\tp-1)}(\omega)\in \cL(\cL(\fE_{n},\fE_{n'}))$ is defined as
\begin{equation}\label{G.block.p-1}
	\tG_{\ell,,n,n'}^{\pm,(\tp-1)}(\omega):= \omega \cdot \ell \ \uno_{n,n'} + M_L\big(\!\left.H_0^{(\tp-1)} \right._{[n]}^{[n]}(\omega)\big) \pm M_R\big(\!\left.H_0^{(\tp-1)} \right._{[n']}^{[n']}(\omega) \big) \,,
\end{equation}
 and  the indexes sets $\cI_{\tN_{\tp-1}}^{\pm}$ are defined in \eqref{eq:indices_meln}.
	For any $\tp\geq 1$, there exist a time-dependent Hamiltonian transformation, defined for all $\omega\in R_{\tM}$, of the form $\b\Phi_{\tp - 1}(\omega; t)=e^{\im \bX^{(\tp-1)}(\omega ;\omega t)} $ with
	\begin{equation}
		\bX^{(\tp-1)}(\omega;\omega t) = \begin{pmatrix}
			X^{d,(\tp-1)}(\omega;\omega t)  & X^{o,(\tp-1)}(\omega;\omega t) \\
			-\overline{ X^{o,(\tp-1)}}(\omega;\omega t)  & -\overline{ X^{d,(\tp-1)}}(\omega;\omega t) 
		\end{pmatrix}\,,
	\end{equation}
such that, for any $\omega\in\Omega_{\tp}$, the following conjugation formula holds:
\begin{equation}\label{conju.rule.kam}
	\bH^{(\tp)} (\omega; t)= \big(\b\Phi_{\tp - 1}(\omega; t)\big)^{-1}\bH^{(\tp-1)} (\omega; t) \,\b\Phi_{\tp - 1}(\omega; t)\,.
\end{equation}
For any $s\in[s_0,s_0+\Sigma(\beta)]$, we have $\bX^{(\tp-1)} \in \cM_{s}(\alpha,\alpha)$, with estimate
		\begin{equation}\label{eq:S1_p}
			\| \bX^{(\tp-1)}\|_{s,\alpha,\alpha}^\wlip{\gamma} \leq \tN_{p-1}^{2\tau+1}\tN_{\tp-2}^{-\varrho} \delta_{s+\beta}^{(0)} \,.
		\end{equation}
\end{thm}


\subsection{Proof of Theorem \ref{thm:iter_lemma}}

We prove Theorem \ref{thm:iter_lemma} by induction. We start prove ${\bf (S1)_{\tp}}$-${\bf (S2)_{\tp}}$ when $\tp=0$.

\paragraph{Proof of ${\bf (S1)_{0}}$-${\bf (S2)_{0}}$}
Property \eqref{eq:S2_p_def} follows by \eqref{initial.H.KAM}, \eqref{eq:KG_s}. 
 Property \eqref{eq:S3_p} holds trivially by \eqref{eq:delta_sigma_p}, \eqref{eq:par_kam} and Lemma \ref{lem.magnus_size}.

\paragraph{The reducibility step}
In this section we describe the generic inductive step, showing how to transform $\bH^{(\tp)}(\omega; t)$ into $\bH^{(\tp+1)}(\omega;t)$ by the conjugation with $\b\Phi_{\tp}(\omega;t)=e^{ -\im \bX^{(\tp)}(\omega; \omega t)} $ of the form
\begin{equation}\label{eq:matrix_transf}
\bX^{(\tp)}(\omega; \omega t)=\begin{pmatrix}
	X^{d,(\tp)}(\omega; \omega t) & X^{o,(\tp)}(\omega; \omega t) \\ -\overline{X^{o,(\tp)}}(\omega; \omega t) & -\overline{X^{d,(\tp)}}(\omega; \omega t)
	\end{pmatrix} \,, \quad   \begin{aligned}
	&X^{d,(\tp)} = [X^{d,(\tp)}]^* \,, \\
	&  \overline{X^{o,(\tp)}} = [X^{o,(\tp)}]^* \,,
\end{aligned} 
	\end{equation}
The Hamiltonian $\bH^{(\tp)}(\omega;t)$ is  transformed, as in \eqref{eq:magnus_1}, in
\begin{equation}
\begin{aligned}
\bH^{(\tp+1)}(t)&:= \LieTr{\im\,\bX^{(\tp)}(\omega; \omega t)}{\bH^{(\tp)}(t)}
\\
& - \int_{0}^1\LieTr{\im\,s\bX^{(\tp)}(\omega; \omega t)}{\dot \bX^{(\tp)}(\omega; \omega t)}\wrt s \,.
\end{aligned}
\end{equation}
By expanding in commutators, we get
\begin{equation}\label{eq:transf_ham_1}
\begin{aligned}
\bH^{(\tp+1)}
& = \bH_0^{(\tp)} +\Pi_{\tN_\tp}\bV^{(\tp)}  +\im [\bX^+, \bH_0^{(\tp)}]-\dot{\bX}^{(\tp)} + \bV^{(\tp+1)} \,, 
\end{aligned}
\end{equation}	
where
\begin{equation}\label{Vp+1}
	\begin{aligned}
		\bV^{(\tp+1)} & :=   \LieTr{\im\,\bX^{(\tp)}}{\bH_0^{(\tp)}} - (\bH_0^{(\tp)} + \im[\bX^{(\tp)}, \bH_0^{(\tp)}])   +\Pi_{\tN_\tp}^{\perp}\bV^{(\tp)}\\
		& \ \ + \LieTr{\im\,\bX^{(\tp)}}{\bV^{(\tp)}}-\bV^{(\tp)}- 	\Big( \int_{0}^{1}\LieTr{\im\,s \bX^{(\tp)}}{\dot{\bX}^{(\tp)}}\wrt{s} - \dot{\bX}^{(\tp)} \Big) \,,
	\end{aligned}
\end{equation} 
and $\Pi_{\tN}$, $\Pi_{\tN}^\perp$ be defined as in \eqref{proj.def.mat}.
We ask now $\bX^{(\tp)}$ to solve the homological equation:
\begin{equation}\label{eq:hom_eq_kam_gen}
\im [\bX^{(\tp)}(\vf), \bH_0^{(\tp)}]-\omega \cdot \partial_{\vf}\bX^{(\tp)}(\vf)+\Pi_{\tN_\tp}\bV^{(\tp)}(\vf) = \bZ^{(\tp)}
\end{equation}
where $\bZ^{(\tp)} $ is the diagonal, time independent part of $V^{d,(\tp)}$:
\begin{equation}\label{eq:Z_kam_gen}
\begin{aligned}
	 \bZ^{(\tp)} =& \bZ^{(\tp)}(\omega) := \begin{pmatrix}
		Z^{(\tp)}(\omega) & 0\\ 0 & - Z^{(\tp)}(\omega)
	\end{pmatrix}\,, \\
&  Z^{(\tp)}(\omega) = \diag\big\{\left.V^{d,(\tp)}\right._{[n]}^{[n]}(\omega;0)\,:\, n\in\N_0 \big\} \,.
\end{aligned}
\end{equation}
By \eqref{eq:S2_p_def} and \eqref{eq:matrix_transf},  equation \eqref{eq:hom_eq_kam_gen},  reads  block-wise for any $n,n'\in\N_0$ and $\abs \ell \leq \tN_\tp$  as
\begin{equation}\label{hom_eq_block}
			\begin{cases}
					\im  \tG_{\ell,,n,n'}^{-,(\tp)}(\omega)
					\left.X^{d,(\tp)} \right._{[n]}^{[n']}(\omega;\ell) 
					= \left.V^{d,(\tp)} \right._{[n]}^{[n']}(\omega;\ell) - \left. Z^{(\tp)} \right._{[n]}^{[n']}(\omega)\,,
			\\
					\im \tG_{\ell,,n,n'}^{+,(\tp)}(\omega)
					 \left.X^{o,(\tp)} \right._{[n]}^{[n']}(\omega;\ell) 
					 = \left.V^{o,(\tp)} \right._{[n]}^{[n']}(\omega;\ell)\,,
			\end{cases}
\end{equation}
where $\tG_{\ell,,n,n'}^{\pm,(\tp)}(\omega)\in\cL(\cL(\fE_{n},\fE_{n'}))$ are defined as in \eqref{G.block.p-1} at the step $\tp$.

%
For any $\omega\in \Omega_{\tp+1}$, the operator $\tG_{\ell,,n,n'}^{\pm,(\tp)}$ are invertible by Lemma \ref{lemma.astratto.fin}-(ii) and, therefore, the  homological equations \eqref{hom_eq_block} are solved by
\begin{align}
	&
		\!	\left. X^{d,(\tp)} \right._{[n]}^{[n']}(\omega;\ell) := \begin{cases}
				-\im \big( \tG_{\ell,,n,n'}^{-,(\tp)}(\omega) \big)^{-1}  \left. V^{d,(\tp)} \right._{[n]}^{[n']}(\omega;\ell)\,, & (\ell,n,n')\in \cI_{\tN_{\tp}}^{-}\,,\\
				0 & \text{ otherwise }\,,
			\end{cases}
	 \label{eq:sol_hom_d} \\
	&
		\!	 \left. X^{o,(\tp)} \right._{[n]}^{[n']}(\omega;\ell) := -\im \big( \tG_{\ell,,n,n'}^{+,(\tp)}(\omega)  \big)^{-1}\left. V^{o,(\tp)} \right._{[n]}^{[n']}(\omega;\ell) \ , \quad (\ell,n,n')\in \cI_{\tN_{\tp}}^{+} \label{eq:sol_hom_o} \,.
\end{align}
The fact that $\Omega_{\tp+1}$ is actually a set of large measure, that is $\me_r(R_{\tM}\backslash\Omega_{\tp+1})=O(\gamma^{1/2})$, recalling \eqref{def:mr}, will be clear as a direct consequence of Lemma \ref{lemma.set.infty} and Theorem  \ref{lem:meas_infty}.

\begin{lem}\label{lem:est_gen_kam}
	{\bf (Homological equations).}
	The operator $\bX^{(\tp)}(\omega)$ defined in \eqref{eq:matrix_transf}, \eqref{eq:sol_hom_d}, \eqref{eq:sol_hom_o} (which, for all $\omega\in\Omega_{\tp+1}$, solves the homological equation \eqref{eq:hom_eq_kam_gen}) admits an extension to $R_{\tM}$. For all $s\in [s_0,s_0+\Sigma(\beta)]$, such extended operator (still denoted by $\bX^{(\tp)}$) belongs to $\cM_s(\alpha,\alpha)$, with estimate
	\begin{equation}\label{eq:gen_kam_est}
		| \bX^{(\tp)} |_{s,\alpha,\alpha}^{\lip(\gamma)} \lesssim \tN_{\tp}^{2\tau+1} \frac{\tM^\alpha}{\gamma} \delta_{s}^{(\tp)}\,.
	\end{equation}
\end{lem}
\begin{proof}
	We prove only the existence of the extension of $X^{o,(\tp)}$, defined by \eqref{eq:sol_hom_o}, and the estimate $\braket{D}^{\alpha}X^{o,(\tp)}\in \cM_{s}$. The proofs that the operators $X^{o,(\tp)}$, $X^{o,(\tp)}\braket{D}^\alpha$, $ \la D\ra^{\pm \alpha} X^{o,(\tp)} \la D \ra^{\mp \alpha}$ belong to $ \cM_{s}$ and the equivalent claim for $X^{d,(\tp)}$, leading to $\bX^{(\tp)}\in \cM_{s}(\alpha,\alpha)$, follow similarly and we omit the details.
	
	First, we extend the solution in \eqref{eq:sol_hom_o} to all $\omega\in R^{\tM}$ by setting, for any $(\ell,n,n')\in \cI_{\tN_{\tp}}^+$,
	\begin{equation}\label{Xop.extended}
		\left. X^{o,(\tp)} \right._{[n]}^{[n']}(\omega;\ell) := -\im \, \chi\big( \tg_{\ell,n,n'}^{+,(\tp)}(\omega) ^{-1} \rho \big) \big( \tG_{\ell,,n,n'}^{+,(\tp)}(\omega)  \big)^{-1}\left. V^{o,(\tp)} \right._{[n]}^{[n']}(\omega;\ell)\,,
	\end{equation}
	where
	\begin{equation}
		 \tg_{\ell,n,n'}^{+,(\tp)}(\omega) := \big\|\big( \tG_{\ell,,n,n'}^{\pm,(\tp)} (\omega)\big)^{-1}\big\|_{\Op(n,n')}
		  \,, \quad \rho:= \tfrac12\gamma\, \tM^{-\alpha} \braket{n+n'}^\alpha \braket{\ell}^{-\tau}\,,
	\end{equation}
	 and $\chi\in \cC^\infty(\R,\R)$ is an even positive $\cC^\infty$ cut-off function as in \eqref{cutoff}.
	 We deduce, by \eqref{eq:Omega_p} at the step $\tp+1$,
	\begin{equation}
		\begin{aligned}
			\big\| \big(\braket{D}^\alpha X^{o,(\tp)}\big)_{[n]}^{[n']}(\omega;\ell) \big\|_{\rm HS} \lesssim \tN_{\tp}^\tau \frac{\tM^\alpha}{\gamma} \frac{\braket{n}^\alpha}{\braket{n+n'}^{\alpha}}\,.
		\end{aligned}
	\end{equation}
Since $\frac{\braket{n}}{\braket{n+n'}}\leq 1$ for any $n,n'\in\N_0$, by \eqref{eq:sdecay_norm} we obtain
\begin{equation}\label{semin.infty.hom}
	\sup_{\omega\in R_\tM}| \braket{D}^\alpha X^{o,(\tp)}(\omega) |_{s}  \lesssim \tN_{\tp}^\tau \frac{\tM^\alpha}{\gamma} \sup_{\omega\in R_\tM} | V^{o,(\tp)}(\omega)|_{s} \,.
\end{equation}
We move now to the estimate for the Lipschitz seminorm. First, note that, for any $\omega_1,\omega_2 \in R_{\tM}$, $\omega_1\neq \omega_2$, recalling the notation used in \eqref{frac.lip},
\begin{equation}
	\begin{aligned}
		 \Delta_{12}\Big( &\chi\big( \tg_{\ell,n,n'}^{+,(\tp)}(\omega)  \rho^{-1} \big) \big( \tG_{\ell,,n,n'}^{+,(\tp)}(\omega)  \big)^{-1}\Big) =  \Delta_{12}\Big(\chi\big( \tg_{\ell,n,n'}^{+,(\tp)}(\omega)  \rho^{-1} \big) \Big) \big( \tG_{\ell,,n,n'}^{+,(\tp)}(\omega_1)  \big)^{-1} \\
		& -  \chi\big( \tg_{\ell,n,n'}^{+,(\tp)}(\omega_2)  \rho^{-1} \big)\,(\tG_{\ell,,n,n'}^{+,(\tp)}(\omega_2))^{-1} \Delta_{12} \big(\tG_{\ell,,n,n'}^{+,(\tp)}(\omega)\big)  (\tG_{\ell,,n,n'}^{+,(\tp)}(\omega_1))^{-1}\,.
	\end{aligned}
\end{equation}
In particular, 
\begin{equation}
	\begin{footnotesize}
		\begin{aligned}
			\Delta_{12} \big(\tG_{\ell,,n,n'}^{+,(\tp)}(\omega)\big) = (\omega_1-\omega_2) \cdot \ell \ \uno_2 + M_L\big( \Delta_{12}\big(\!\left.H_0^{(\tp)} \right._{[n]}^{[n]}(\omega)\big)\big) \pm M_R\big( \Delta_{12}\big(\!\left.H_0^{(\tp)} \right._{[n']}^{[n']}(\omega)\big) \big) \,.
		\end{aligned}
	\end{footnotesize}
\end{equation}
By \eqref{eq:Omega_p}, \eqref{eq:S2_p_est1},
we get
\begin{equation}
	\Big\| 	
	 \Delta_{12}\Big( \chi\big( \tg_{\ell,n,n'}^{+,(\tp)}(\omega)  \rho^{-1} \big) \big( \tG_{\ell,,n,n'}^{+,(\tp)}(\omega)  \big)^{-1}\Big) 
	 \Big\|_{\Op(n,n')} \lesssim \tN_{\tp}^{2\tau +1}  \frac{\tM^{2\alpha}}{\gamma^2} \frac{|\omega_1-\omega_2 |}{\la n+n' \ra^{\alpha}} \,.
\end{equation}
Therefore, from \eqref{Xop.extended} and arguing as above, we deduce, 
\begin{equation}
		\begin{aligned}
			\sup_{\omega_1, \omega_2 \in R_{\tM} \atop \omega_1 \neq \omega_2 } \frac{\big\| \Delta_{12}\big( \big(\la D\ra^\alpha X^{o,(\tp)} \big)_{[n]}^{[n']}(\omega;\ell) \big)\big\| }{| \omega_1-\omega_2 |} & \lesssim\tN_{\tp}^{2\tau +1} \frac{\tM^{2\alpha}}{\gamma^2}  
			\sup_{\omega \in \R_\tM}  \big\|\! \left. V^{o,(\tp)} \right._{[n]}^{[n']}(\omega;\ell) \big\|_{\rm HS} \\
			& + \tN_{\tp}^{\tau} \,\frac{\tM^\alpha}{\gamma}
			\sup_{\omega_1, \omega_2 \in R_{\tM} \atop \omega_1 \neq \omega_2 } \frac{\big\| \Delta_{12}\big( \left.V^{o,(\tp)} \right._{[n]}^{[n']}(\omega;\ell) \big)\big\| }{| \omega_1-\omega_2 |}\,,
		\end{aligned}
\end{equation}
and, consequently, for all $s\in [s_0,s_0+\Sigma(\beta)]$,
\begin{equation}\label{semin.lip.hom}
	\sup_{\omega_1, \omega_2 \in R_{\tM} \atop \omega_1 \neq \omega_2 } \frac{ \big| \Delta_{12} \big( \braket{D}^\alpha X^{o,(\tp)} (\omega) \big) \big|_{s-1} }{| \omega_1-\omega_2|} \lesssim  \tN_{\tp}^{2\tau +1} \frac{\tM^{2\alpha}}{\gamma^2} \delta_{s}^{(\tp)}\,. 
 \end{equation}
By Definition \eqref{def:sdecay} and \eqref{semin.infty.hom}, \eqref{semin.lip.hom}, we conclude $	| \braket{D}^\alpha X^{o,(\tp)} |_s^{\lip(\gamma)} \lesssim \tN_{\tp}^{2\tau +1} \frac{\tM^{\alpha}}{\gamma} \delta_{s}^{(\tp)}$,
as claimed.
\end{proof}

By \eqref{eq:transf_ham_1}, \eqref{eq:hom_eq_kam_gen}, \eqref{eq:Z_kam_gen}, for all $\omega \in \Omega_{\tp}$, we obtain
\begin{equation}\label{Hp+1.after}
	\begin{aligned}
		& \bH^{(\tp+1)}(\omega;t) = \bH_0^{(\tp+1)}(\omega) + \bV^{(\tp+1)}(\omega;\omega t)\,, \\
		& \bH_0^{(\tp+1)}(\omega) := \bH_0^{(\tp)}(\omega) + \bZ^{(\tp)}(\omega)\,,
	\end{aligned}
\end{equation}
where $\bV^{(\tp+1)}(\omega;\omega t)$ is given in \eqref{Vp+1}. Since $\bH_0^{(\tp)}$, $\bV^{(\tp)}$ (by induction) and $\bX^{(\tp)}$ (by construction) are defined for all $\omega\in R_{\tM}$, we get that $\bH^{(\tp+1)}(\omega;t)$ is defined as well for all parameters $\omega\in R_\tM$. The new operator $\bH^{(\tp+1)}$ in \eqref{Hp+1.after} has the same form as $\bH^{(\tp)}$ in \eqref{eq:Hp.stat}. The new normal form $\bH_0^{(\tp+1)}$ in \eqref{Hp+1.after} is block-diagonal.

\begin{lem}\label{new.block.lemma}
{\bf (New block-diagonal part).}
	For all $\omega\in R_{\tM}$, we have
	\begin{equation}
		\bH_0^{(\tp+1)}(\omega) :=\bH_0^{(\tp)}(\omega) + \bZ^{(\tp)}(\omega) = \diag \big\{ \left. H_0^{(\tp+1)} \right._{[n]}^{[n]}(\omega) \, : \, n\in \N_0  \big\} \bsigma_3\,,
	\end{equation}
	where, for each  $n\in\N_0$, the block $ \left. H_0^{(\tp+1)} \right._{[n]}^{[n]}(\omega)$ is self-adjoint and satisfies the estimate
	\begin{equation}
		\sup_{n\in\N_0}\braket{n}^\alpha\,\big\|\left.H_0^{(\tp)}\right._{[n]}^{[n]}(\omega)-\left.H_0^{(\tp-1)}\right._{[n]}^{[n]}(\omega)\big\|_{{\rm HS}}^\wlip{\gamma}  \leq \delta_{s_0}^{(\tp)}\,. 
	\end{equation}
\end{lem}
\begin{proof}
	Recalling \eqref{eq:Z_kam_gen}, we have that $\bZ^{(\tp)}(\omega)= \diag\big\{  \left. V^{d,(\tp)} \right._{[n]}^{[n]}(\omega;0)  \, : \, n\in \N_0 \big\}\bsigma_3$ is self-adjoint, by \eqref{struttura} in Definition \ref{M.sdecay.ab} and the induction assumption that $\bV^{(\tp)}\in \cM_{s}(\alpha,0)$, in particular for $s=s_0$. This implies directly that $\bH_0^{(\tp+1)}$ is self-adjoint as well and the estimate $\braket{n}^\alpha\| \left. V^{d,(\tp)} \right._{[n]}^{[n]} (\,\cdot\,;0) \|_{\rm HS}^{\lip(\gamma)} \leq | \braket{D}^\alpha V^{d,(\tp)} |_{s_0}^{\lip(\gamma)} \leq \delta_{s_0}^{(\tp)}$.
\end{proof}

\paragraph{The iterative step.}
We now assume, by the induction assumption, that ${\bf (S1)_{\tp}}$, ${\bf (S2)_{\tp}}$ hold and we prove ${\bf (S1)_{\tp+1}}$, ${\bf (S2)_{\tp+1}}$. By Lemma \ref{lem:est_gen_kam}, \eqref{eq:Omega_p}, \eqref{eq:matrix_transf} and \eqref{Hp+1.after}, there exists $\bX^{(\tp)}\in \cM_{s}(\alpha,\alpha)$ such that, for any $\omega\in \Omega_{\tp+1}$, the conjugation \eqref{conju.rule.kam}, \eqref{eq:Hp.stat} holds at $\tp+1$, with \eqref{eq:Hp.stat} defined for all $\omega\in R_{\tM}$. By Lemma \ref{new.block.lemma}, the induction assumption on \eqref{eq:S2_p_est1}, \eqref{eq:S3_p}, we have that \eqref{eq:S2_p_def}, \eqref{eq:S2_p_est1}, \eqref{eq:S2_p_est} hold at the step $\tp+1$, with each block $ \left. H_0^{(\tp+1)} \right._{[n]}^{[n]}(\omega)$ being self-adjoint.
The proofs of \eqref{eq:S1_p} and of \eqref{eq:S3_p} at the step $\tp+1$ follow by Lemma \ref{lem:est_gen_kam} and the following lemma.
\begin{lem}
	{\bf (Nash-Moser estimates).}
	For any $s\in[s_0,s_0+\Sigma(\beta)]$, the operator $\bV^{(\tp+1)}$ in \eqref{Vp+1} belongs to $\cM_{s}(\alpha,0)$ with the iterative estimates
\begin{align}
	& \delta_{s}^{(\tp+1)} \lesssim_{s,\alpha} \tN_{\tp}^{-\beta} \delta_{s+\beta}^{(\tp)} + \tN_{\tp}^{2\tau+1} \frac{\tM^\alpha}{\gamma} \delta_{s_0}^{(\tp)} \delta_{s}^{(\tp)} \,, \label{nash.moser.1}\\
	& \delta_{s+\beta}^{(\tp+1)} \lesssim_{s,\alpha}  \delta_{s+\beta}^{(\tp)} + \tN_{\tp}^{2\tau+1} \frac{\tM^\alpha}{\gamma} \big( \delta_{s_0+\beta}^{(\tp)} \delta_{s}^{(\tp)}  + \delta_{s_0}^{(\tp)} \delta_{s}^{(\tp)}   \big)\,.\label{nash.moser.2}
\end{align}
Moreover, the estimates \eqref{eq:S3_p} hold at the step $\tp+1$.
\end{lem}
\begin{proof}
	The estimates \eqref{nash.moser.1}, \eqref{nash.moser.2} follow by \eqref{Vp+1}, \eqref{eq:hom_eq_kam_gen}, Lemmata \ref{lem:com}, \ref{lem:flow}, \eqref{propieta.class.sdecay}-(ii) and \eqref{eq:delta_sigma_p}, \eqref{eq:gen_kam_est}, \eqref{eq:small_cond_kam}, \eqref{eq:par_kam}. The estimates \eqref{eq:S3_p} at the step $\tp+1$ follow by \eqref{nash.moser.1}, \eqref{nash.moser.2}, \eqref{eq:S3_p} and the smallness condition \eqref{eq:small_cond_kam}, for $\tN_0=\tN_0(s_0,\beta)\in \N$ large enough.
\end{proof}

\subsection{Diagonalization of the operator $\bH^{(0)}$}

In Theorem \ref{thm:iter_lemma}, we proved the generic step of the KAM iteration. In this section, we conclude the KAM reducibility, showing the existence of the limit flow that fully diagonalize, under the smallness condition \eqref{eq:small_cond_kam}, the operator $\bH_0(\omega;t)$ in \eqref{initial.H.KAM} obtained after the Magnus transform in Theorem \ref{lem:magnus}.
\begin{cor}\label{cor:final_blocks}
	{\bf (Final blocks).}
	Assume \eqref{eq:small_cond_kam}. The sequence $ \big\{ \bH_0^{(\tp)} (\omega;\tM,\alpha)  \big\}_{\tp\in\N_0}$ converges,  for any $\omega\in R_\tM$, to
	\begin{equation}\label{final.eigen}
		\bH_0^{(\infty)}(\omega):= \bH_0^{(0)} + \bZ^{(\infty)}(\omega)=\diag\Big\{ \left.H_0^{(\infty)}\right._{[n]}^{[n]}(\omega) \,:\, n\in\N_0 \Big\} \bsigma_3\,,
	\end{equation}
	with estimates
	\begin{equation}\label{stime.blocks.infty}
		\begin{aligned}
			& 	\sup_{n\in\N_0}\braket{n}^\alpha\,\big\|\!\left.H_0^{(\infty)}\right._{[n]}^{[n]}-\!\left.H_0^{(\tp)}\right._{[n]}^{[n]}\big\|_{{\rm HS}}^\wlip{\gamma} \lesssim_{s_0,\beta}  (\gamma_0\,\tM)^{-1}  \tN_{\tp-1}^{-\varrho}\,, \quad 
			\forall\, \tp\in \N_0\,.
		\end{aligned}
	\end{equation}
	For each $n\in\N_0$, the block $\left.H_0^{(\infty)}\right._{[n]}^{[n]}(\omega;\tM,\alpha) $ is self-adjoint and the finitely many eigenvalues $\lambda_{n}^{(\infty)}(\omega;\tM,\alpha) \in \spec\big( \!\left.H_0^{(\infty)}\right._{[n]}^{[n]}(\omega;\tM,\alpha)  \big)$
	  are real and positive, admitting an asymptotic of the form
	\begin{equation}\label{final.eigen.expans}
	\lambda_{n}^{(\infty)}(\omega) = \lambda_{n} + \varepsilon_{\lambda_{n}^{(\infty)}}(\omega)\,, \quad
	\sup_{n\in\N_0} \braket{n}^\alpha |\varepsilon_{\lambda_{n}^{(\infty)}}|^{\lip(\gamma)} \leq C_{s_0,\beta} (\gamma_0\,\tM)^{-1}\,.
\end{equation}
\end{cor}
\begin{proof}
	By estimate \eqref{eq:S2_p_est} in Theorem \ref{thm:iter_lemma}, we have that $ \big\{ \bH_0^{(\tp)} (\omega;\tM,\alpha)  \big\}_{\tp\in\N_0}$ is a Cauchy sequence with respect to the norm $\sup_{n\in\N_0}\braket{n}^\alpha\,\| ( \,\cdot\,)_{[n]}^{[n]}\|_{{\rm HS}}^\wlip{\gamma} $. The estimates \eqref{stime.blocks.infty} for the limit, block diagonal operator $\bH_0^{(\infty)}$ follow from \eqref{eq:S2_p_est1}, \eqref{eq:S2_p_est} with a standard telescopic series argument, assuming the smallness condition \eqref{eq:small_cond_kam}, Lemma \ref{lem.magnus_size} and choosing $\tN_0\in \N$ large enough. Each block $\left. H_0^{(\infty)} \right._{[n]}^{[n]}(\omega)$ is self-adjoint because it is the limit of self-adjoint blocks. The expansion and the estimate in\eqref{final.eigen} follow by Lemma \ref{lemma.astratto.fin}-(i). Indeed, we have
	\begin{equation}
	\begin{aligned}
		\sup_{n\in\N_0} \braket{n}^\alpha | \varepsilon_{\lambda_{n}^{(\infty)}} |^{\lip(\gamma)} & = \sup_{n\in\N_0} \braket{n}^\alpha | \lambda_{n}^{(\infty)} - \lambda_n |^{\lip(\gamma)}\\
		& \lesssim \sup_{n\in\N_0}\braket{n}^\alpha\,\big\|\!\left.H_0^{(\infty)}\right._{[n]}^{[n]}-\!\left.H_0^{(0)}\right._{[n]}^{[n]}\big\|_{{\rm HS}}^\wlip{\gamma} \lesssim_{s_0,\beta}  (\gamma_0\,\tM)^{-1} \,.
	\end{aligned}
\end{equation}
This concludes the proof.
 \end{proof}

We define the set $\Omega_{\infty}\subset R_{\tM}$ of the second order balanced Melnikov non-resonance conditions for the final blocks as
	\begin{equation}\label{Omega.infty}
	\begin{aligned}
		\Omega_{\infty} := \Omega_{\infty}(\gamma,\tau) := \Big\{ \omega \in \Omega_0 \,:\, & \big\|\big( \tG_{\ell,,n,n'}^{\pm,(\infty)} (\omega)\big)^{-1}\big\|_{\Op(n,n')} \leq \frac{\braket{\ell}^\tau}{\gamma}\frac{\tM^\alpha}{\braket{n\pm n'}^\alpha} \\
		&      \quad \quad \quad \quad \quad \quad  \quad  \quad \quad \ \  \forall \,(\ell,n,n')\in \cI^\pm  \Big\}\,,
	\end{aligned}
\end{equation}
where $\Omega_0$ is defines as in \eqref{eq:magnus_omega} and, recalling the notation in \eqref{MRML},
\begin{equation}\label{G.block.infty}
	\tG_{\ell,,n,n'}^{\pm,(\infty)}(\omega):= \omega \cdot \ell \ \uno_2 + M_L\big(\!\left.H_0^{(\infty)} \right._{[n]}^{[n]}(\omega)\big) \pm M_R\big(\!\left.H_0^{(\infty)} \right._{[n']}^{[n']}(\omega) \big)\in\cL(\fE_{n},\fE_{n'}) \,,
\end{equation}
with the indexes sets defines in \eqref{eq:indices_meln}.
\begin{lem}\label{lemma.set.infty}
	We have $\Omega_{\infty}\subseteq \cap_{\tp\in\N_0} \Omega_{\tp} $.
\end{lem}
\begin{proof}
 	We prove by induction that $\Omega_{\infty}\subseteq \Omega_{\tp}$ for any $\tp\in \N_0$. For $\tp=0$ the claim is trivial because $\Omega_{\infty}\subset \Omega_0$ by \eqref{Omega.infty}. We now assume by induction that $\Omega_{\infty}\subseteq \Omega_{\tp}$ for some $\tp\in\N_0$ and we show that $\Omega_{\infty}\subseteq \Omega_{\tp+1}$. Let $\omega\in\Omega_{\infty}$ and
 $(\ell,j,j')\in \cI_{\tN_{\tp}}^{+}$.
 First, by \eqref{G.block.infty} and \eqref{G.block.p-1}, we have $\tG_{\ell,,n,n'}^{\pm,(\tp)}(\omega) = \tG_{\ell,,n,n'}^{\pm,(\infty)}(\omega) + \tT_{\ell,n,n'}^{\pm,(\tp)}(\omega)\,,$
where
\begin{equation}
	\tT_{\ell,n,n'}^{\pm,(\tp)}(\omega):= M_L\big(\!\left.H_0^{(\tp)} \right._{[n]}^{[n]}(\omega) - \!\left.H_0^{(\infty)} \right._{[n]}^{[n]}(\omega)\big) \pm M_R\big(\!\left.H_0^{(\tp)} \right._{[n']}^{[n']}(\omega) - \!\left.H_0^{(\infty)} \right._{[n']}^{[n']}(\omega) \big) \,.
\end{equation}
  Since $\omega\in\Omega_{\infty}$, $\tG_{\ell,,n,n'}^{\pm,(\infty)}$ is invertible, By Corollary \ref{cor:final_blocks}, \eqref{Omega.infty} \eqref{eq:par_kam} and the smallness condition \eqref{eq:small_cond_kam}, we get
  \begin{equation}
  	\begin{aligned}
  		\| \big( \tG_{\ell,,n,n'}^{\pm,(\infty)}(\omega) \big)^{-1} \tT_{\ell,n,n'}^{\pm,(\tp)} \|_{\Op(n,n')} &\leq \| \big( \tG_{\ell,,n,n'}^{\pm,(\infty)}(\omega) \big)^{-1}  \|_{\Op(n,n')} \|  \tT_{\ell,n,n'}^{\pm,(\tp)} \|_{\Op(n,n')}\\
  		& \lesssim_{s_0,\beta} \frac{\tM^\alpha}{\gamma\,\gamma_0\,\tM} \frac{\tN_{\tp}^\tau \tN_{\tp-1}^{-\varrho}}{\braket{n-n'}^\alpha} \leq \frac12\,,
  	\end{aligned}
  \end{equation}
choosing $\tN_0=\tN_0(\tau,\nu,s_0)\in\N$ large enough. It implies that, for $\omega\in\Omega_{\infty}$, the operator $\tG_{\ell,,n,n'}^{\pm,(\tp)}$ is invertible by a Neumann series argument, with estimate
\begin{equation}
		\begin{footnotesize}
			\begin{aligned}
				\| \big( \tG_{\ell,,n,n'}^{\pm,(\tp)}(\omega) \big)^{-1} \|_{\Op(n,n')}\leq \frac{	\| \big( \tG_{\ell,,n,n'}^{\pm,(\infty)}(\omega) \big)^{-1} \|_{\Op(n,n')}}{1-	\| \big( \tG_{\ell,,n,n'}^{\pm,(\infty)}(\omega) \big)^{-1} \tT_{\ell,n,n'}^{\pm,(\tp)} \|_{\Op(n,n')}} \leq 2	\| \big( \tG_{\ell,,n,n'}^{\pm,(\infty)}(\omega) \big)^{-1}\|_{\Op(n,n')}\,.
			\end{aligned}
		\end{footnotesize}
\end{equation}
By \eqref{Omega.infty}, \eqref{eq:Omega_p} and the induction assumption that $\omega\in\Omega_{\tp}$, we conclude that $\omega\in\Omega_{\tp+1}$, which is the claim and concludes the proof.
\end{proof}

We now define the sequence of invertible maps
\begin{equation}
	\label{seq.Xn}
	\cW_{0}:={\rm Id}\,, \quad \cW_{\tp}(\omega;t) := e^{\im \bX^{(0)}(\omega;\omega t)} \circ \cdots\circ e^{\im \bX^{(\tp-1)}(\omega;\omega t)}\,, \quad \tp\in\N\,.
\end{equation}

\begin{thm}\label{cor:iter_flow}
	{\bf (KAM reducibility).}
	Fix $\alpha\in(0,1)$. There exists $\tN_0=\tN_0(\tau,\nu,s_0)\in\N$ such that, if \eqref{eq:small_cond_kam} is verified, for any $\omega \in R_{\tM}$, the sequence of transformations $(\cW_{\tp}(\omega))_{\tp\in\N}$ 
	converges in  $\cL(H^r(\T^{\nu+1})\times H^r(\T^{\nu+1}))$, $r\in[0,s_0]$, to an invertible operator $\cW_{\infty}(\omega)$  with estimate
\begin{equation}
	\begin{aligned}
		&\| (\cW_{\infty}(\omega))^{\pm} - (\cW_{\tp}(\omega))^{\pm} \|_{\cL(H^r(\T^{\nu+1})\times H^r(\T^{\nu+1}))} \lesssim_{s_0}(\gamma_0\,\tM)^{-1} \tN_{\tp+1}^{2\tau+1} \tN_{\tp}^{-\varrho}\,.
	\end{aligned}
\end{equation}
Moreover, for any $\omega\in\Omega_{\infty}$, we have
\begin{equation}
	\bH_0^{(\infty)}(\omega) =( \cW_{\infty}(\omega;t))^{-1} \bH^{(0)}(\omega,t)\cW_{\infty}(\omega;t) = \diag\Big\{ \left.H_0^{(\infty)}\right._{[n]}^{[n]}(\omega) \,:\, n\in\N_0 \Big\} \bsigma_3\,,
\end{equation}
where $	\bH_0^{(\infty)}(\omega) $ is as in Corollary \ref{cor:final_blocks}, with each block $\left.H_0^{(\infty)}\right._{[n]}^{[n]}(\omega;\tM,\alpha) $, $n\in\N_0$, being self-adjoint and with eigenvalues $\{ \lambda_{n,-}^{(\infty)}(\omega;\tM,\alpha),\lambda_{n,+}^{(\infty)} (\omega;\tM,\alpha)\}$, are real and positive, admitting the asymptotics \eqref{final.eigen.expans}.
\end{thm}
\begin{proof}
	The claim follows by Lemma \ref{lem:flow}-(i), Theorem \ref{thm:iter_lemma}, Lemma \ref{lemma.set.infty}  and a standard argument in KAM reducibility schemes, see for instance Lemma 7.5 in \cite{BKM16}. Therefore, we omit the details.
\end{proof}

\subsection{Balance Melnikov conditions and measure estimates}\label{sec:melnikov}

The goal of this section is to prove that the set $\Omega_{\infty}\subset R_{\tM}$ of non-resonance conditions \eqref{Omega.infty} is of large measure with respect to the annulus $R_{\tM}\subset \R^\nu$. This will be achieved in Theorem \ref{lem:meas_infty}. This result shows second order Melnikov conditions for perturbations of the  eigenvalues
$(\lambda_j)_{j\in\Z}$
 of the operator $B$ defined in \eqref{def:B}. Explicitly, for $j\in \Z$,
\begin{equation}\label{lambda.j.unpert}
	\lambda_j:=\sqrt{j^2+\tq+ d(j)}=|j|+\frac{c_j(q)}{\braket{j}} \,, \quad c_j(q) := \braket{j} (\sqrt{j^2+\tq+d(j)}-|j|)\,.
\end{equation}
One directly checks that, for any$ j \in \Z$,
\begin{equation}\label{bound}
	\begin{aligned}
		0 \leq |c_j(q)| & \leq\max\{ c_0(q),\, |\tq+d(j)|  \} \\
		& \leq \max\{ c_0(q),\, |\tq| + \|d(j)\|_{\ell^2(\Z)} \}=:\tm^2\,,
	\end{aligned}
\end{equation} 
recalling, by \eqref{assumption Q}, that $(d(j))_{j\in\Z}\in \ell^2(\Z)$. 

We recall that the relative measure of a measurable set $\Omega$ is defined as
\begin{equation}
	\label{def:mr}
	\me_r (\Omega) := \frac{|\Omega|}{|R_{\tM}|} \equiv \frac{|\Omega|}{\tM^\nu \, (2^{\nu }-1) c_\nu} \,,
\end{equation}
where $| \cC|$ is the Lebesgue measure of the set $\cC$ and $c_\nu$ is the volume of the unitary ball in $\R^\nu$.
We have the following standard estimate.
%

\begin{lem}\label{tecnico}
	Fix  $\ell \in \Z^\nu \setminus \{0\}$ and 
	let $R_\tM \ni \omega \mapsto \varsigma(\omega) \in \R$ be a Lipschitz function fulfilling
	$\abs{\varsigma}^{\lip}_{R_\tM} \leq \tc_0 < |\ell| . $
	Define 
	$f(\omega) = \omega \cdot \ell  + \varsigma(\omega)$.
	Then, for any $\delta \geq 0$, the measure of the set $ A:=\Set{\omega\in R_\tM | \abs{f(\omega)} \leq \delta } $ satisfies the upper bound
	\begin{equation}\label{measd}
		\abs{A}\leq \frac{ 2 \delta }{|\ell|- \tc_0} (4\tM)^{\nu-1} \ .
	\end{equation}
	
\end{lem}
\begin{proof}
	Take $\omega_1 = \omega + \epsilon \,\ell$, with $\epsilon$ sufficiently small so that $\omega_1 \in R_\tM$. 
	Then
	$
\tfrac{|f(\omega_1) - f(\omega)|}{|\omega_1 - \omega|} \geq |\ell| -\abs{\varsigma}^{\lip}_{R_\tM} > |\ell|- \tc_0
	$
	and the estimate follows by Fubini theorem.
\end{proof}

The main result is the following theorem.

\begin{thm}\label{lem:meas_infty}
	{\bf (Measure estimates).}
	Let $\Omega_0$, $\Omega_{\infty}$ be defined in \eqref{eq:magnus_omega}, \eqref{Omega.infty}, respectively.
	Let $\gamma\in(0,1)$ and
	\begin{equation}\label{scelta.tau}
		\gamma_0 =\gamma^{\alpha/4}\,, \quad 	\tau > \nu -1 + \alpha + \frac{\tau_0}{\alpha}\,.
	\end{equation}
	Then,  for $\tM \geq \tM_0(s_0,\beta)$ large enough,  there exists  a constant $C_\infty >0$, independent of $\tM$ and $\gamma$, such that
	\begin{equation}
	\label{final_est}
	\me_r(\Omega_0 \backslash\Omega_{\infty})\leq C_\infty \gamma^{1/2} \ .
	\end{equation}
\end{thm}

Before starting with the proof, we reformulate the set $\Omega_{\infty}$ in \eqref{Omega.infty} in terms of lower bounds for the eigenvalues of $\tG_{\ell,,n,n'}^{\pm,(\infty)}(\omega)$ in \eqref{G.block.infty}.

\begin{lem}\label{reform.meln}
	We have
		\begin{equation}	\label{eq:Omegainfty}
\begin{aligned}
				\Omega_{\infty} (\gamma,\tau)\equiv & \Big\{ \omega\in \Omega_0 \,:\, 
	|\omega \cdot \ell +\mu_{n}(\omega) \pm \mu_{n'}(\omega) | \geq   \frac{\gamma}{\braket{\ell}^{\tau}}
	\frac{\braket{n\pm n'}^{\alpha}}{\tM^{\alpha}}\,,\\
&	  \forall \, (\ell,n,n')\in\cI^{\pm}, \ \mu_{m}(\omega)\in \spec\big(\! \left.H_0^{(\infty)}\right._{[m]}^{[m]}(\omega;\tM,\alpha)\big),\, m=n,n' 
	\Big\}\,,
\end{aligned}
			\end{equation}
		where the self-adjoint blocks $\left.H_0^{(\infty)}\right._{[n]}^{[n]}(\omega;\tM,\alpha) $, $n\in\N_0$, are given in Corollary \ref{cor:final_blocks}.
\end{lem}
\begin{proof}
	By Lemma \ref{lemma:MRMRL} (see also Lemma 7.2 in \cite{BKM16}), we have, for any $(\ell,n,n')\in\cI^{\pm}$,
	\begin{equation}
		\spec\big( \tG_{\ell,,n,n'}^{\pm,(\infty)} \big)\! =\! \Big\{  \omega \cdot \ell +\mu_{n}(\omega) \pm \mu_{n'}(\omega) \, : \,   \mu_{m}(\omega)\!\in \spec\big(\! \left.H_0^{(\infty)}\right._{[m]}^{[m]}(\omega)\big),\, m=n,n'  \Big\}\,.
	\end{equation}
	The claim, follows by  
	Lemma \ref{lemma.astratto.fin}-(ii)
	and the definition of the set $\Omega_{\infty}$ in \eqref{Omega.infty}.
\end{proof}

The rest of the section is devoted to the proof of Theorem \ref{lem:meas_infty}.
	We write the complementary set $\Omega_0\setminus \Omega_{\infty}$ as
	\begin{equation}\label{compl.Omega}
		\Omega_0\setminus \Omega_{\infty} =\Bigg( \bigcup_{\ell\in\Z^\nu, n,n'\in\N_0 \atop (\ell,n,n')\neq (0,n,n)} \cQ_{\ell,n,n'}^{(-)}\Bigg) \cup\Bigg(  \bigcup_{\ell\in\Z^\nu, n,n'\in\N_0} \cQ_{\ell,n,n'}^{(+)} \Bigg)
	\end{equation}
where, by Lemma \ref{reform.meln}, we define the "nearly-resonant sets" as
\begin{align}
	&\begin{footnotesize}
			\begin{aligned}
			\cQ_{\ell,n,n'}^{(\pm)}:=\cQ_{\ell,n,n'}^{(\pm)}(\gamma,\tau)    := \bigcup \Big\{  \wt\cQ_{\ell,\mu_{n},\mu_{n'}}^{(\pm)}(\gamma,\tau)  \,: \, \mu_{m}(\omega)\in \spec\big(\! \left.H_0^{(\infty)}\right._{[m]}^{[m]}(\omega)\big),\, m=n,n'\Big\}\,, 
		\end{aligned}
	\end{footnotesize} \nonumber \\
& \begin{footnotesize}
	\begin{aligned}
		\wt\cQ_{\ell,\mu_{n},\mu_{n'}}^{(\pm)} := \wt\cQ_{\ell,\mu_{n},\mu_{n'}}^{(\pm)} (\gamma,\tau) & :=  \Big\{  \omega\in \Omega_0 \,:\,   
		|\omega \cdot \ell +\mu_{n}(\omega) \pm \mu_{n'}(\omega) | <  \frac{\gamma}{\braket{\ell}^{\tau}}
		\frac{\braket{n\pm n'}^{\alpha}}{\tM^{\alpha}} \Big\} \,.
	\end{aligned}
\end{footnotesize} \label{Q.nearly.res}
\end{align}
Some of these sets are actually empty.
\begin{lem}\label{lemma.empty.1}
	For $\tM\geq \tM_0$ large enough, if $	\cQ_{\ell,n,n'}^{(\pm)}  \neq \emptyset $, then $|n\pm n'|\leq C_1 \tM \braket{\ell}$.
\end{lem}
\begin{proof}
	If  $\cQ_{\ell,n,n'}^{(\pm)}  \neq \emptyset $, then there exists $\omega \in \Omega_0$ such that
	\begin{equation}\label{parte1}
		| \mu_{n}(\omega) \pm \mu_{n'}(\omega) | < \frac{\gamma}{\braket{\ell}^{\tau}}
		\frac{\braket{n\pm n'}^{\alpha}}{\tM^{\alpha}} + \tM |\ell |\,,
	\end{equation}
for some eigenvalues $\mu_{m}(\omega)\in \spec\big(\! \left.H_0^{(\infty)}\right._{[m]}^{[m]}(\omega)\big)$, $ m=n,n' $. By \eqref{final.eigen.expans} in Corollary \ref{cor:final_blocks}, \eqref{lambda.j.unpert} and \eqref{bound}, we have
\begin{equation}\label{parte2}
	\begin{aligned}
		| \mu_{n}(\omega) \pm \mu_{n'}(\omega)| & \geq |n\pm n'| -\frac{|c_n(q)|}{\braket{n}} - \frac{| c_{n'}(q) |}{\braket{n'}}
		- | \varepsilon_{\mu_{n}}(\omega) | - | \varepsilon_{\mu_{n'}}(\omega)| \\
		& \geq |n\pm n' | -2\tm^2 - 2C_{s_0,\beta} (\gamma_0 \tM)^{-1}\,.
	\end{aligned}
\end{equation}
Choosing $\tM\gg 1$ large enough, the claim follows by combining \eqref{parte1} with \eqref{parte2}. 
\end{proof}

\begin{lem}\label{lemma.empty.0}
	Let $\gamma_0\geq 2\gamma$ and $\tau\geq \tau_0$. For any $\ell\in\Z^\nu\setminus\{0\}$ and $n\in\N$ such that 
	\begin{equation}\label{condizione.extra.0}
		\braket{n}^\alpha \geq \tR_0(\ell):= \frac{4\,C_{s_0,\beta}}{(\gamma_0\,\tM)^2}\braket{\ell}^{\tau_0}\,,
	\end{equation}
	we have $\cQ_{\ell,n,n}^{(-)}(\gamma,\tau)=\emptyset$. Moreover, $\cQ_{\ell,0,0}^{(-)}(\gamma,\tau)=\emptyset$
\end{lem}
\begin{proof}
	Let $\mu_{n},\mu_{n}'\in \spec\big(\! \left.H_0^{(\infty)}\right._{[n]}^{[n]}(\omega)\big) $. Note that, when $n=0$, the block $ \left.H_0^{(\infty)}\right._{[0]}^{[0]}(\omega)$ is one dimensional and the spectrum contains one simple eigenvalue. When $\mu_{n}=\mu_{n}'$, then, recalling that $\cQ_{\ell,n,n}^{(-)}\subset \Omega_0$, we have $|\omega\cdot\ell| \geq\frac{\gamma_0\,\tM}{\braket{\ell}^{\tau_0}}\geq\frac{\gamma}{\braket{\ell}^{\tau}}\tM^{-\alpha}$. Therefore, let $n\geq 1$ and $\mu_{n}\neq\mu_{n}'$. By Corollary \ref{cor:final_blocks} and \eqref{condizione.extra.0}, we have
	\begin{equation}
		\begin{aligned}
			| \omega\cdot\ell +\mu_{n} & - \mu_{n}'  | \geq | \omega\cdot \ell | - |\varepsilon_{\mu_{n}}(\omega) | - |\varepsilon_{\mu_{n}'}(\omega)|\\
			& \geq \frac{\gamma_0\,\tM}{\braket{\ell}^{\tau_0}} - \frac{2\, C_{s_0,\beta}}{\gamma_0\,\tM} \braket{n}^{-\alpha}  \geq \frac{\gamma_0\,\tM}{2\braket{\ell}^{\tau_0}} \geq \frac{\gamma}{\braket{\ell}^\tau} \frac{1}{\tM^\alpha}\,.
		\end{aligned}
	\end{equation}
	This proves the claim.
 \end{proof}

Given $\gamma_1\in(0,1)$ and $\tau_1\geq 1$ to choose,   we define the sets, for $(\ell,j)\in \Z^{\nu+1}\setminus\{0\}$,
\begin{equation}
	\cR_{\ell,j}^1:=\cR_{\ell,j}^1(\gamma_1,\tau_1);= \Big\{ \omega\in\Omega_0 \,:\, | \omega\cdot\ell + j | < \frac{\gamma_1}{\braket{\ell}^{\tau_1}} \frac{\braket{j}^\alpha}{\tM^\alpha} \Big\}\,.
\end{equation}

\begin{lem}\label{lemma.empty.2}
	Let $\gamma_1\geq 2\gamma$ and $\tau\geq \tau_1>1$. Then, for any $(\ell,n,n')\in\cI^{\pm}$, if
	\begin{equation}\label{condizione.extra}
		\braket{\min\{n,n'\}}^{\alpha} \braket{n\pm n'}^\alpha \geq \tR_1(\ell):= 8 \max\Big\{ \tm^2 ,\frac{C_{s_0,\beta}}{\gamma_0\,\tM}  \Big\} \frac{\tM^\alpha}{\gamma_1} \braket{\ell}^{\tau_1}\,,
	\end{equation}
 	then $\cQ_{\ell,n,n'}^{(\pm)}(\gamma,\tau)\subset \bigcup_{(\ell,j)\neq 0} \cR_{\ell,j}^1(\gamma_1,\tau_1)$.
\end{lem}
\begin{proof}
	If $\omega\in \Omega_0\setminus  \bigcup_{(\ell,j)\neq 0} \cR_{\ell,j}^1(\gamma_1,\tau_1)$, then $	| \omega\cdot\ell + j | \geq \frac{\gamma_1}{\braket{\ell}^{\tau_1}} \frac{\braket{j}^\alpha}{\tM^\alpha}$ for any $(\ell,j)\in \Z^{\nu+1}\setminus\{0\}$.
Let $(\ell,n,n')\in \cI^{\pm}$. By Corollary \ref{cor:final_blocks}, \eqref{lambda.j.unpert}, \eqref{bound}, \eqref{condizione.extra} and the assumptions $\gamma_1\geq  2\gamma$, $\tau\geq \tau_1$, we get, for any eigenvalues  $\mu_{m}(\omega)\in \spec\big(\! \left.H_0^{(\infty)}\right._{[m]}^{[m]}(\omega)\big)$, $ m=n,n'$, recalling that $\alpha\in (0,1)$,
\begin{equation}
	\begin{aligned}
	 	| \omega\cdot \ell +\mu_{n}(\omega) & \pm \mu_{n'}(\omega) | \geq  |\omega\cdot \ell + n\pm n'| -\tfrac{|c_n(q)|}{\braket{n}} - \tfrac{| c_{n'}(q) |}{\braket{n'}}
		- | \varepsilon_{\mu_{n}}(\omega) | - | \varepsilon_{\mu_{n'}}(\omega)| \\
		& \geq \frac{\gamma_1}{\braket{\ell}^{\tau_1}} \frac{\braket{n\pm n'}^\alpha}{\tM^\alpha}  - \tm^2\big( \braket{n}^{-1} + \braket{n'}^{-1} \big) - \frac{C_{s_0,\beta}}{\gamma_0\,\tM} \big( \braket{n}^{-\alpha} + \braket{n'}^{-\alpha} \big) \\
		& \geq \frac{\gamma_1}{\braket{\ell}^{\tau_1}} \frac{\braket{n \pm n'}^\alpha}{\tM^\alpha} - 4\max\Big\{ \tm^2 ,\frac{C_{s_0,\beta}}{\tM}  \Big\} \braket{\min\{n,n'\}}^{-\alpha}\\
		& \geq \frac{\gamma_1}{2\braket{\ell}^{\tau_1}} \frac{\braket{n \pm n'}^\alpha}{\tM^\alpha}   \geq \frac{\gamma}{\braket{\ell}^{\tau}} \frac{\braket{n \pm n'}^\alpha}{\tM^\alpha} \,.
	\end{aligned}
\end{equation}
This shows that $\omega \notin \wt\cQ_{\ell,\mu_{n},\mu_{n'}}^{(\pm)}(\gamma,\tau) \subset \cQ_{\ell,n,n'}^{(\pm)}(\gamma,\tau)$ and the claim is proved.
\end{proof}
We finally move to the estimate of $\Omega_0\setminus\Omega_{\infty}$. By \eqref{condizione.extra}, we have
\begin{equation}
	|\Omega_0 \setminus \Omega_{\infty} | \leq \bigg| \bigcup_{\ell\in\Z^{\nu},\, n,n'\in\N_0 \atop (\ell,n,n')\neq (0,n,n)}  \cQ_{\ell,n,n'}^{(-)} \bigg| + \bigg| \bigcup_{\ell\in\Z^{\nu},\, n,n'\in\N_0}  \cQ_{\ell,n,n'}^{(+)} \bigg| =:  \tI_{-} + \tI_{+}\,.
\end{equation}
We show the estimate for $\tI_{-}$ which is the most delicate one. The estimate for $\tI_{+}$ follows similarly and therefore we omit. By Lemmata \ref{lemma.empty.1}, \ref{lemma.empty.0}, \ref{lemma.empty.2},
we have
\begin{equation}
	 \begin{aligned}
	 	\bigcup_{\ell\in\Z^{\nu},\, n,n'\in\N_0 \atop (\ell,n,n')\neq (0,n,n)}  \cQ_{\ell,n,n'}^{(-)} (\gamma,\tau) = & \bigcup_{\ell\in\Z^\nu, \, n\in\N \atop \braket{n}^\alpha < \tR_0(\ell) } \cQ_{\ell,n,n}^{(-)}(\gamma,\tau) \cup \bigcup_{\ell\in\Z^{\nu},\, n-n'\in\Z\setminus\{0\},  \atop | n-n'| \leq C_1 \tM\braket{\ell} } \cR_{\ell,n-n'}^1 (\gamma_1,\tau_1)\\
 		&  \cup \bigcup_{\ell\in\Z^{\nu},\, n,n'\in\N_0,\, (\ell,n,n')\neq(0,n,n) ,  \atop | n-n'| \leq C_1 \tM\braket{\ell},\,  
 			\braket{\min\{n,n'\}}^{\alpha} \braket{n - n'}^\alpha < \tR_1(\ell)}\cQ_{\ell,n,n'}^{(-)}(\gamma,\tau)\,,
	 \end{aligned}
\end{equation}
and therefore
\begin{equation}
	 \begin{aligned}
		\tI_{-} = &\bigg| \bigcup_{\ell\in\Z^\nu, \, n\in\N \atop \braket{n}^\alpha < \tR_0(\ell) } \cQ_{\ell,n,n}^{(-)}(\gamma,\tau) \bigg| + \bigg| \bigcup_{\ell\in\Z^{\nu},\, n-n'\in\Z\setminus\{0\},  \atop | n-n'| \leq C_1 \tM\braket{\ell} } \cR_{\ell,n-n'}^1 (\gamma_1,\tau_1)\bigg| \\
		&  +\bigg| \bigcup_{\ell\in\Z^{\nu},\, n,n'\in\N_0,\, (\ell,n,n')\neq(0,n,n) ,  \atop | n-n'| \leq C_1 \tM\braket{\ell},\,  
			\braket{\min\{n,n'\}}^{\alpha} \braket{n - n'}^\alpha < \tR_1(\ell)}\cQ_{\ell,n,n'}^{(-)}(\gamma,\tau)\bigg| =: \tI_{-,1} + \tI_{-,2} +\tI_{-,3}\,.
	\end{aligned}
\end{equation}
By Lemma \ref{tecnico}, we have, for some numerical constants $C_2,C_3>0$,
\begin{align}
	&| \cQ_{\ell,n,n'}^{(-)}(\gamma,\tau) | \leq C_2 \frac{\gamma\, \tM^{\nu-1-\alpha}}{\braket{\ell}^{\tau+1}}\braket{n-n'}^\alpha\,,  \label{stima.set2} \\
	& | \cR_{\ell,n-n'}^{1}(\gamma_1,\tau_1)  | \leq C_3  \frac{\gamma_1\, \tM^{\nu-1-\alpha}}{\braket{\ell}^{\tau_1+1}}\braket{n-n'}^\alpha\,. \label{stima.set3}
\end{align}
By \eqref{stima.set2} and \eqref{condizione.extra.0}, we estimate $\tI_{-,1}$ by
\begin{equation}
	\tI_{-,1} \leq \sum_{\ell\in\Z^\nu, \, n\in\N \atop \braket{n}^\alpha < \tR_0(\ell) } | \cQ_{\ell,n,n}^{(-)}(\gamma,\tau) | \lesssim \frac{\gamma}{\gamma_0^{2/\alpha}}\tM^{\nu-1-\alpha-\frac{2}{\alpha}} \sum_{\ell\in\Z^\nu} \frac{1}{\braket{\ell}^{\tau+1 -\frac{\tau_0}{\alpha}}} \lesssim \frac{\gamma}{\gamma_0^{2/\alpha}\tM^{1+\alpha+2/\alpha}}\, \tM^\nu\,,
\end{equation}
with $\tau+1-\frac{\tau_0}{\alpha}>\nu$.  By \eqref{stima.set3}, we estimate $\tI_{-,2}$ by
\begin{equation}
	\begin{aligned}
		\tI_{-,2} \leq \sum_{\ell\in\Z^{\nu},\, j\in\Z\setminus\{0\},  \atop |j| \leq C_1 \tM\braket{\ell} } |\cR_{\ell,j}^1 (\gamma_1,\tau_1)| \lesssim \gamma_1\tM^\nu \sum_{\ell\in\Z^\nu} \frac{1}{\braket{\ell}^{\tau_1-\alpha}} \lesssim \gamma_1 \,\tM^{\nu}\,,
	\end{aligned}
\end{equation}
with $\tau_1-\alpha>\nu$. 
By \eqref{stima.set2} and \eqref{condizione.extra}, we estimate $\tI_{-,3}$ by
\begin{equation}
	\begin{aligned}
		\tI_{-,3} & \leq \sum_{\ell\in\Z^{\nu},\, n,n'\in\N_0,\, (\ell,n,n')\neq(0,n,n) ,  \atop | n-n'| \leq C_1 \tM\braket{\ell},\,  
			\braket{\min\{n,n'\}}^{\alpha} \braket{n - n'}^\alpha < \tR_1(\ell)}  |\cQ_{\ell,n,n'}^{(-)}(\gamma,\tau)| \\
		& \lesssim \gamma\,\tM^{\nu-1-\alpha} \sum_{\ell\in\Z^{\nu},\, m\in\N_0,\, j\in\Z^\nu\setminus\{0\},  \atop |j| \leq C_1 \tM\braket{\ell},\,  
			m<  < \tR_1(\ell)^{1/\alpha} \braket{j}^{-1}  } \frac{\braket{j}^\alpha}{\braket{\ell}^{\tau+1}} 	\lesssim \frac{\gamma}{\gamma_1^{1/\alpha}}\tM^{\nu-\alpha} \sum_{\ell\in\Z^{\nu},\, j\in\Z^\nu\setminus\{0\},  \atop |j| \leq C_1 \tM\braket{\ell}  
		 } \frac{\braket{j}^{\alpha-1}}{\braket{\ell}^{\tau+1-\frac{\tau_1}{\alpha}}} \\
	 & \lesssim \frac{\gamma}{\gamma_1^{1/\alpha}}\tM^{\nu} \sum_{\ell\in\Z^{\nu} } \frac{1}{\braket{\ell}^{\tau+1-\frac{\tau_1}{\alpha}-\alpha}} \lesssim \frac{\gamma}{\gamma_1^{1/\alpha}}\tM^{\nu}\,,
	\end{aligned}
\end{equation}
with $\tau+1-\frac{\tau_1}{\alpha}-\alpha>\nu$. We conclude that, for $\gamma_0\in(0,1)$ and $\tau>0$ as in \eqref{scelta.tau},
\begin{equation}
	\tI_{-} = \tI_{-,1} + \tI_{-,2}+ \tI_{-,3} \lesssim \Big( \frac{\gamma}{\gamma_0^{2/\alpha}\tM^{1+\alpha+2/\alpha}} + \gamma_1+ \frac{\gamma}{\gamma_1^{1/\alpha}} \Big)\tM^\nu \lesssim \gamma^{1/2}\,\tM
\end{equation}
choosing $\tau_1=\tau_0$, $\gamma_1=\gamma_0^2\simeq \gamma^{\alpha/2}$. By \eqref{def:mr}, it implies \eqref{final_est} and concludes the proof.

\subsection{Proof of Theorem \ref{thm:main}}\label{subsec:proof.main}
 Let $\tM_*= \tM_0$ and $\gamma_*:= \min\{ \gamma^{\frac{\alpha}{4}}, \gamma^\frac12$, with $\tM_0$ and $\gamma$ as in Theorem \ref{lem:meas_infty}. Then the set $\Omega_{\infty}^\alpha:= \Omega_{\infty}$, where $\Omega_{\infty}$ is defined in \eqref{Omega.infty}, satisfies \eqref{measure.state} by Theorem \ref{lem:magnus} and Theorem \ref{lem:meas_infty}. We define now $\cT(\omega;\omega t):= \big(e^{\bY(\omega;\omega t)} \circ \cW_{\infty}(\omega;\omega t) \big)^{-1}$, where $\bY(\omega;\omega t)$ is given in Theorem \ref{lem:magnus} and $\cW_{\infty}(\omega;\omega t)$ in Theorem \ref{cor:iter_flow}. Then, by Theorem \ref{lem:magnus},Theorem \ref{cor:iter_flow}, Corollary \ref{cor:final_blocks}, \eqref{eq:par_kam} and \eqref{eq:small_cond_kam}, setting $\sigma_*:= \Sigma(\beta)$, with $\Sigma(\beta)$ as in \eqref{eq:par_kam}, the change of coordinates $\phi = \cT(\omega;\omega t) \psi$ conjugates \eqref{eq:KG_matrix} to \eqref{eq:eff_sys}, where the map $\cT(\omega;\omega t)$ is bounded in $\cL(\cH^r)$ for any $r\in[0,s_0]$ and it is close to the identity, namely satisfies \eqref{close.id.state}. The expansion \eqref{eq:eff_eig} follows by Corollary \ref{cor:final_blocks}.

\appendix

\section{Pseudodifferential functional calculus}\label{app:funct.calc}
The goal of this appendix is to briefly give a definition as a pseudodifferential operator of $B=\sqrt{-\pa_{xx}+q(x)}$, starting from its standard spectral definition in terms of functional calculus for the operator $L_q=-\pa_{xx}+q(x)$. The construction is based on the definition of complex powers for self-adjoint operators proposed by Seeley in \cite{seeley} and the extension to pseudodifferential operators made by Shubin in his monograph \cite{shubin}. Since we are only interested with the parameter-free and time independent operator $B$, this section will only deal with operators and symbols independent of $\vf\in\T^\nu$ and $\omega\in \R^\nu$.
We recall the definition of the class of symbols $S^m$ of order $m\in\R$ in Definition \eqref{def:symbol}. For convenience, we introduce the following subclasses of symbols:
\begin{itemize}
	\item $ \dot S^m:= \{ f(x,\xi)\in S^m  \,: \, f(x,\mu \xi) = \mu^m f(x,\xi), \ \mu>0  \} $ [{\it homogeneous symbols}];
	\item $ {\rm C} S^m := \{ f(x,\xi)\in S^m \,:\, f(x,\xi) \sim \sum_{n=0}^\infty  f_{m-n}(x,\xi) , \ f_{m-n}\in \dot S^m \}$;
	\item ${\rm H}S^m:= \{ f(x,\xi)\in{\rm C}S^m \,:\, f_m(x,\xi)\neq 0 \ \text{ for } \abs\xi \neq 0 \} $ [{\it elliptic symbols}].
\end{itemize}
The classes of operators ${\rm OPC}S^m$ and   ${\rm OPH}S^m$  have the clear definitions of quantization of symbols in the classes ${\rm C}S^m$ and ${\rm H}S^m$, respectively.
\begin{lem}\label{classic}
	{\bf (Lemma 2.2, \cite{Saint91}).}
	Let $m\in\R$ and let $f_{m-n}\in S^{m-n}$ for $n\in\N_0$. Then there exists a symbol $f\in S^m$ (unique modulo $ S^{-\infty}$) such that, for any $k\in\N_0$,
	$ f-\sum_{n<k}f_{m-n}\in  S^{m-k} $. In this case, we write $f\sim \sum_{n\in\N_0}f_{m-j}$.
\end{lem}

\bigskip
\medskip 

\noindent{\bf Resolvent and parametrix of an elliptic symbol.}
The following proposition gives a characterization for the existence of the inverse for a symbol of order $m\in\R$.
\begin{prop}
	{\bf (Theorem 2.10, \cite{Saint91}).}
	If $a\in S^m$, the following four statements are equivalent:
	\\[1mm]
	\noindent (i) There exists $b\in S^{-m}$ such that $a\# b -1 \in S^{-\infty}$;
	\\[1mm]
	\noindent (ii) There exists $b\in S^{-m}$ such that $b\#a -1 \in S^{-\infty}$;
	\\[1mm]
	\noindent (ii) There exists $b_0\in S^{-m}$ such that $ab_0-1\in S^{-1}$;
	\\[1mm]
	\noindent (iv) There exists $\varepsilon>0$ such that $\abs{a(x,\xi)}\geq \varepsilon\braket{\xi}^m $ for $\abs\xi \geq 1/\varepsilon$.
	\\[1mm]
	When one of these condition is satisfied, then there exists $a^\#\in S^{-m}$ such that
	\begin{equation*}
	b \text{ solves (i) } \ \Leftrightarrow \ b \text{ solves (ii) } \ \Leftrightarrow \ b-a^\#\in S^{-\infty}.
	\end{equation*}
	Moreover, if $a\in {\rm C}S^m$, then $a$ satisfies (iv) if and only if $a\in {\rm H}S^m$.\\
\end{prop}
We apply this result to directly construct the symbol for the resolvent operator $G(A;\lambda):=(A-\lambda\,{\rm Id})^{-1}$, namely the parametrix of the operator $A-\lambda\,{\rm Id}$, when $A=\Op(a(x,\xi))\in {\rm OPH}S^{m}$.
Let $a(x,\xi)\sim \sum_{n=0}^\infty a_{m-n}(x,\xi)$ and set
\begin{equation}
	\wta_m(\lambda;x,\xi):= a_m(x,\xi)-\lambda\,, \quad \wta_{m-n}(\lambda;x,\xi):=a_{m-n}(x,\xi)\,, \quad n\in\N\,.
\end{equation} 
By Lemma \ref{classic}, there exists a symbol $\wta(\lambda;x,\xi)\in {\rm H}S^m$ such that $\wta \sim \sum_{n=0}^\infty \wta_{m-n}$  and $A-\lambda \,{\rm Id} = \Op( \wta(\lambda;x,\xi))$. Note that, with this choice of the symbol, we have that $\wt a_m(\lambda;x,\xi)$ is homogeneous of degree $m$ in the couple $(\xi,\lambda^{1/m})$.

First, we look for a formal symbol $b^0(\lambda;x,\xi)\sim \sum_{n=0}^\infty b_{-m-n}^0(\lambda;x,\xi)$ such that $b^0\# \wta \sim  1$. Recalling \eqref{compo_symb}, we compute
\begin{equation*}
\begin{aligned}
b^0\# \wta & \sim (b_{-m}^0 + b_{-m-1}^0 + b_{-m-2} +... )(\wta_m+\wta_{m-1}+\wta_{m-2}+...)\\
& + \sum_{\beta=1}^\infty \frac{1}{\im^\beta \beta !} \partial_\xi^\beta (b_{-m}^0 + b_{-m-1}^0 + b_{-m-2} +...) \pa_x^\beta (\wta_m+\wta_{m-1}+\wta_{m-2}+...) 
\end{aligned}
\end{equation*}
The symbol $b^0(\lambda;x,\xi)$ is therefore defined recursively by the relations
\begin{equation}\label{eq:b0_rel_res}
\begin{aligned}
& b_{-m}^0(\lambda;x,\xi) \wta_m(\lambda;x,\xi) = 1 \,;\\
& b_{-m-n}^0(\lambda;x,\xi)\wta_m(\lambda;x,\xi) + \sum_{p=0}^{n-1}b_{-m-p}^0(\lambda;x,\xi) \wta_{m-n+p}(\lambda;x,\xi) \\
& \qquad + \sum_{\beta=1}^n\frac{1}{\im^\beta\beta !}\sum_{p=0}^{n-\beta} \pa_\xi^\beta b_{-m-p}^0(\lambda;x,\xi) \pa_x^\beta \wta_{m-n+\beta+p}(\lambda;x,\xi) = 0 \,, \quad n\in\N \,.
\end{aligned}
\end{equation}
In particular, by explicit computations, the symbols $b_{-m-n}^0(\lambda;x,\xi)$ are of the form
\begin{equation}
b_{-m}^0(\lambda;x,\xi) = \frac{1}{a_m(x,\xi)-\lambda} \ , \quad b_{-m-n}^0(\lambda;x,\xi) = \frac{1}{(a_m(x,\xi)-\lambda)^{\beta_n}}p_{-m-n}(x,\xi) \ , 
\end{equation} 
where $\beta_n \in \N $ and each $p_{-m-n}(x,\xi)$ is independent of $\lambda$, involving the symbols $a_m(x,\xi),\dots,a_{m-n}(x,\xi)$ so that the each function $b_{-m-n}^0$ is homogeneous in $(\xi,\lambda^{1/m})$ of degree $-m-n$.
In order to obtain a true parametrix from the symbols $b_{-m-n}^0(\lambda;x,\xi)$, it is necessary to remove singularities for $\abs\xi + \abs \lambda^{1/m}$ by  a cut-off function. Let $\chi \in\C^\infty(\R,\R)$ be an even positive $\cC^\infty$-function as in \eqref{cutoff}.
We set $\wt\chi(\lambda;\xi):= \chi ( \abs \xi ^2 + \abs\lambda^{2/m}) $ and  we define
\begin{equation}\label{param_b}
b_{-m-n}(\lambda;x,\xi) := \wt\chi(\lambda;\xi) b_{-m-n}^0(\lambda;x,\xi) \in S^m\,,
\end{equation}
together with
\begin{equation}\label{BN_lambda}
B_{-m-n}(\lambda):=\Op( b_{-m-n}(\lambda;x,\xi) )\ , \quad B_{(N)}(\lambda):= \sum_{n=0}^{N-1} B_{-m-n} (\lambda) \ .
\end{equation}
This construction is summed up in the following result.
\begin{prop}\label{reso_pseudo}
	{\bf (Proposition 11.2, \cite{shubin}).}
	Let $A\in {\rm OPH}S^m$. We have 
	\begin{equation}
		G(A;\lambda)- B_{(N)}(\lambda) \in {\rm OPC}S^{-m-N}\,, \quad \forall\,N\in\N\,,
	\end{equation}
	where $B_{(N)}(\lambda)$ as in \eqref{BN_lambda}.
	In particular, there exists $B(\lambda) = b(x,D;\lambda)\in {\rm OPC}S^{-m}$ such that $b \sim \sum_{n=0}^\infty b_{-m-n}$, with $b_{-m-n}$ defined in \eqref{param_b}, and $B(\lambda)-B_{(N)}(\lambda)\in{\rm OPC}S^{-m-N}$ and $G(A;\lambda)-B(\lambda)\in {\rm OPC}S^{-\infty}$.
\end{prop}

\medskip 

\noindent{\bf  Functional calculus and holomorphic semigroup properties.}
By Proposition \ref{reso_pseudo},  the resolvent of an elliptic pseudodifferential operator is in the class of pseudodifferential operators as well. On the other side, it is possible to define many operators starting from the spectral resolution of an elliptic operator and its resolvent. The goal is therefore to relate these two constructions.
Let $A\in {\rm OPH}S^m$ be an elliptic pseudodifferential operator of order $m$ with principal symbol $a_m(x,\xi)$. We assume the following:
\begin{itemize}
	\item $A-\lambda \,{\rm Id}\in {\rm OPH}S^m(\Lambda):=\{ F(\lambda)\in {\rm OPH}S^m \,:\, \lambda\in\Lambda \}$, where
	\begin{equation}
		\Lambda:=\Set{\lambda\in\C | \pi-\varepsilon \leq \arg\lambda\leq \pi+\varepsilon }\,, \quad \varepsilon>0\,,
	\end{equation}
	 is a closed angle with vertex in 0. In particular, we assume $	a_m(x,\xi)-\lambda \neq 0 $ for $\xi\neq 0 $ and $\lambda\in(-\infty,0]$;
	\item The resolvent $G(A;\lambda)=(A-\lambda \,{\rm Id} )^{-1}$ is defined for any $\lambda\in\Lambda$ and $A^{-1}$ exists as an operator:
	\begin{equation}
	\sigma(A)\cap \Lambda = \emptyset  \ \left( \ \Rightarrow \ 0\notin \sigma(A)\ \right ) . 
	\end{equation}
\end{itemize}
For a fixed $\rho>0$ small enough such that $B_\rho(0)\cap \sigma(A) = \emptyset$, we consider the clockwise oriented contour $\Gamma:=\Gamma_1\cup\Gamma_2\cup\Gamma_3$, where $\Gamma_1:=\Set{re^{\im\pi} | +\infty > r > \rho}$, $\Gamma_2:=\Set{ \rho e^{\im\vf} | \pi>\vf>-\pi }$, $\Gamma_3:=\Set{re^{-\im\pi} | 0<r<+\infty}$, and we define the following operator
\begin{equation}\label{eq:Az_funct}
A_{z}:=-\frac{1}{2\pi\im} \int_\Gamma \lambda^z (A-\lambda \,{\rm Id})^{-1}\wrt\lambda \ ,
\end{equation}
where $z\in\C$, with $\re(z)<0$, and $\lambda^z$ is defined as a holomorphic function in $\lambda$ for $\lambda\in\C\setminus (-\infty,0]$. The integral over the unbounded contour in \eqref{eq:Az_funct} is always meant as the limit
\begin{equation}
A_z := - \frac{1}{2\pi\im} \lim_{R\rightarrow +\infty} \int_{\Gamma \cap \partial B_R(0)} \lambda^z (A-\lambda \,{\rm Id})^{-1} \wrt \lambda 
\end{equation}
in the topology of the ambient space where the operator $A$ lies: here the condition $\re(z)<0$ enters in the well-posedness of the definition of $A_z$.
The family of operators $\{ A_z \, :\, \re(z)<0\}$ enjoys some algebraic properties.
\begin{prop}
	{\bf (Proposition 10.1, \cite{shubin}).}
 We have the semigroup property $A_z A_w = A_{z+w}$ for any  $z,w\in\C$ with $\re(z),\re(w)<0$.
		If $A$ is invertible, then $A_{-k} = (A^{-1})^k$ for any $k\in\N$.
		Moreover, $A_z$ is a holomorphic operator-function of $z$ (for $\re(z)<0$) with values in the algebra of bounded operators on the Hilbert space $H^r(\T)$, $r\geq 0$.
\end{prop}
In the following theorem, the definition of $A_z$ in \eqref{eq:Az_funct} is connected to the complex power $A^z$ for any $z\in\C$. The construction is mainly due to Seeley \cite{seeley}.
\begin{thm}\label{thm:funct_seeley}
	{\bf (Theorem 10.1, \cite{shubin}).}
	For $z\in\C$ and $k\in\Z$ such that $\re(z)-k<0$, we define the following operator
	\begin{equation}
		A^z:=A^k A_{z-k}
	\end{equation}
Then, the definition of $A^z$ is independent of the choice of $k\in\Z$, provided  $\re(z)< k$.
	Moreover, the following holds:
		\\[1mm]
	\noindent $(i)$ If $\re(z)<0$, then $A^z=A_z$;
			\\[1mm]
		\noindent $(ii)$ The group property holds: $A^z A^w = A^{z+w}$ for any $z,w\in\C$;
		\\[1mm]
	\noindent $(iii)$ For $z=k\in\Z$, the definition of $A^k$ gives the usual $k$-th power of the operator $A$;
			\\[1mm]
		\noindent $(iv)$ For arbitrary $k\in\Z$ and $r\in\R$, the function $A^z$ is a holomorphic operator-function of $z$ in the half-plane $\re(z)<k$ with values in the Banach space $\cL(H^r(\T), H^{r-mk}(\T))$.	
\end{thm}

Note that, if one assume that the operator $A\in {\rm OPH}S^m$ is self-adjoint with a complete system of eigenfunctions $(\vf_j)_{j\in\Z}$ in $L^2(\T)$ corresponding to the eigenvalues $(\mu_j)_{j\in\Z}$ (assuming $\inf_{j\in\Z} \mu_j > 0$), then the action of $A^z$, as in Theorem \ref{thm:funct_seeley}, on a function $f(x):= \sum_{j\in\Z} \left(f,\vf_j\right)_{L^2}\vf_j(x) \in L^2(\T)$ is equivalent to its spectral definition:
\begin{equation}
A^z f(x) = \sum_{j\in\Z} \mu_j^z \left(f,\vf_j\right)_{L^2}\vf_j(x) \ .
\end{equation}
%
\medskip 

\noindent{\bf $A_z$ and $A^z$ as pseudodifferential operators.}
Under the assumptions of the previous paragraph on the elliptic operator $A\in {\rm OPH}S^m$, we construct the parametrix for the resolvent operator $(A-\lambda \, {\rm Id})^{-1}$ given by $b^0(\lambda;x,\xi) \sim \sum_{n=0}^\infty b_{-m-n}^0(\lambda;x,\xi)$ as in \eqref{eq:b0_rel_res}, with $\Lambda:=\C\setminus (-\infty,0]$. We now define, for any $n\in\N_0$ and $z\in\C$
\begin{equation}
b_{mz-n}^{(z),0}(x,\xi):= -\frac{1}{2\pi\im} \int_\Gamma \lambda^z b_{-m-n}^0(\lambda;x,\xi) \wrt\lambda , \quad b_{mz-n}^{(z)}(x,\xi):=\chi(|\xi|) b_{mz-n}^{(z),0}(x,\xi) ,
\end{equation}
where $\chi(\eta)$ is the cut-off function in \eqref{cutoff}, and we set
\begin{equation}
B_{mz-n}^{(z)} := \Op\big(b_{mz-n}^{(z)}(x,\xi)\big) \ , \quad B_{(N)}^{(z)} := \sum_{n=0}^{N} B_{mz-n}^{(z)} \,, \ \ N\in\N_0\,.
\end{equation}
\begin{thm}\label{thm:structure_seeley}
	{\bf (Structure Theorem - Theorem 11.2, \cite{shubin}).} Let $A \in {\rm OPH}S^m$.
	For any $z\in\C$, one has
	\begin{equation}
	A^z = A_z\in {\rm OPC}S^{m\re(z)}\,, \quad A^z - B_{(N)}^{(z)} \in {\rm OP}S^{m\re(z)-N} \,, \quad \forall\,N\in\N_0\,.
	\end{equation}
\end{thm}
\begin{rem}
	During the proof in \cite{shubin}, one formally considers $B^{(z)}\sim \sum_{n=0}^\infty B_{mz -n}^{(z)} \in {\rm OPC}S^{m\re(z)}$, so that $A^z - B^{(z)} \in {\rm OP}S^{-\infty}$.
\end{rem}
\begin{rem}
	The dependence on $z\in\C$ for the family of operators $(A^z)_{z\in\C}$ is holomorphic: in \cite{shubin}, proper subclasses of holomorphic symbols and pseudodifferential operators are discussed. Since we are going to consider fixed real powers of an operator $A\in {\rm OPH}S^m$, these properties are here omitted.
\end{rem}

\bigskip 


\noindent {\bf Powers of the Schr\"odinger operator $-\partial_{xx}+q(x)$.}
We specialize now the discussion so far to the case when the elliptic operator is given by $A=L_q:=-\pa_{xx}+q(x)$, acting on the scale $(H^r(\T))_{r\in\R}$ and with $q\in H^\infty(\T)$. Clearly, we have $L_q\in {\rm OPH}S^2$ with symbol given by $\xi^2+q(x)\in {\rm H}S^2$. 
\begin{proof}[Proof of Theorem \ref{cor:pseudo_sqrt}]
	By Theorem \ref{thm:funct_seeley}, define $L_q^{1/2}:= L_q \circ(L_q) _{-1/2}$, with $(L_q)_{-1/2}$ as in \eqref{eq:Az_funct}. Then, by Theorem \ref{thm:structure_seeley} we have $L_q^{1/2}\in {\rm OP}S^{2\frac{1}{2}}={\rm OP}S^1 $.
	The definition of $B^\mu$ as pseudodifferential operator follows from the same argument.
\end{proof}

\section{Technical results on off-diagonal decay operators}\label{app:tech}

In the following we consider the operators $\bV \in \cM_s(\alpha,0)$ and $\bX \in \cM_s(\alpha,\alpha)$ with matrix structures as in \eqref{eq:Amatrix}.
\\[1mm]
\noindent {\bf Proof of Lemma \ref{lem:com}.} Let $\bV \in \cM_s(\alpha,0)$ and $\bX \in \cM_s(\alpha,\alpha)$ with matrix structure as in \eqref{eq:Amatrix}. Then
\begin{equation}
    {\rm ad}_{\bX}(\bV) =  \im\, [\bX,\bV] = \begin{pmatrix}
        \im \, W^{d} & \im \, W^{o} \\ - \overline{\im\, W^{o}} & - \overline{\im\, W^{d}}
    \end{pmatrix}\,,
\end{equation}
where
\begin{equation}\label{ad.components}
    \begin{aligned}
    W^{d} & := X^{d}V^{d} - V^{d} X^{d} -(X^{o} \overline{V^{o}} - V^{o} \overline{X^{o}} )\,, \\
    W^{o} & := X^{d} V^{o} + V^{o} \overline{X^{d}} -(X^{o} \overline{V^{d}} + V^{d} X^{o}).
    \end{aligned}
\end{equation}
By Lemma \ref{prop:tame} and \eqref{diago.Dx}, we have the following estimates, for any $\varrho= 0,\pm\alpha$, omitting conjugations and superscripts,
\begin{equation}\label{stime.tutticasi}
    \begin{aligned}
    & | \braket{D}^{\varrho} X V \braket{D}^{-\varrho}|_{s}^{\lip(\tw)} \lesssim_{s}  | \braket{D}^{\varrho} X  \braket{D}^{-\varrho}|_{s}^{\lip(\tw)} | \braket{D}^{\varrho} V \braket{D}^{-\varrho}|_{s_0}^{\lip(\tw)} \\
    & \quad \quad \quad \quad \quad \quad \quad \quad \quad \quad + | \braket{D}^{\varrho} X  \braket{D}^{-\varrho}|_{s_0}^{\lip(\tw)} | \braket{D}^{\varrho}  V \braket{D}^{-\varrho}|_{s}^{\lip(\tw)}\,,\\
    & | \braket{D}^{\alpha} X V|_{s}^{\lip(\tw)} \lesssim_{s} | \braket{D}^{\alpha} X |_{s}^{\lip(\tw)} |  V|_{s_0}^{\lip(\tw)} + | \braket{D}^{\alpha} X |_{s_0}^{\lip(\tw)} |  V|_{s}^{\lip(\tw)}  \,, \\
    & | X V \braket{D}^{\alpha}|_{s}^{\lip(\tw)}  \lesssim_{s} | X \braket{D}^{\alpha}|_{s}^{\lip(\tw)} | \braket{D}^{-\alpha} V \braket{D}^{\alpha}|_{s_0}^{\lip(\tw)} \\
    & \quad \quad \quad \quad \quad \quad \quad + | X  \braket{D}^{\alpha}|_{s_0}^{\lip(\tw)} | \braket{D}^{-\alpha} V \braket{D}^{\alpha}|_{s}^{\lip(\tw)}\,,
    \end{aligned}
\end{equation}
    and similar estimates for $VX$ instead of $XV$. By \eqref{ad.components}, Definition \eqref{M.sdecay.ab} and the estimates \eqref{stime.tutticasi}, it follows easily that ${\rm ad}_{\bX}(\bV) \in \cM_{s}(\alpha,\alpha)$ for $s\geq s_0$, with the claimed estimate \eqref{eq:ad_lip_alg}.
\\[1mm]
\noindent {\bf Proof of Lemma \ref{lem:flow}.} The proof of item $(i)$ follows from Remark \eqref{rem:opnorm_vs_sdecay} and the fact that the flow $\Phi(\tau):=e^{\im\tau\bX}$ generated by the bounded operator $\bX$ in $H^r(\T^{\nu+1})\times H^r(\T^{\nu+1})$ stays bounded in the same bounded for $\tau\in[0,1]$.
We now move to the proof of item $(ii)$. We recall that the operator $\LieTr{\im\,\bX}{\bV}$ admits the expansion
\begin{equation}\label{ad.expan}
    \LieTr{\im\,\bX}{\bV} = \sum_{n=0}^{\infty}\frac{1}{n!} \ad_{\bX}^{n}(\bV) \,, \quad \ad_{\bX}^{0}:={\rm Id} \,, \quad \ad_{\bX}^{n}:=\ad_{\bX}\circ\ad_{\bX}^{n-1}\,.
\end{equation}
By Lemma \ref{lem:com}, it is not hard to show the following iterative estimates, for $n\geq 1$:
\begin{equation}\label{adnXV}
    \begin{aligned}
     | \ad_{\bX}^{n}(\bV) |_{s_0,\alpha,\alpha}^{\lip(\tw)} & \leq \big( C_{s_0} |\bX |_{s_0,\alpha,\alpha}^{\lip(\tw)} \big)^n |\bV |_{s_0,\alpha,0}^{\lip(\tw)}\,, \\
     | \ad_{\bX}^{n}(\bV) |_{s,\alpha,\alpha}^{\lip(\tw)} & \leq n C_{s}\big(C_{s_0}|\bX|_{s_0,\alpha,\alpha}^{\lip(\tw)}\big)^{n-1} | \bX |_{s,\alpha,\alpha}^{\lip(\tw)} |\bV |_{s_0,\alpha,0}^{\lip(\tw)} \\
    & \quad + \big(C_{s_0}|\bX|_{s_0,\alpha,\alpha}^{\lip(\tw)}\big)^{n}  |\bV |_{s,\alpha,0}^{\lip(\tw)}\,.
    \end{aligned}
\end{equation}
Then, the estimates \eqref{eq:flow_tame} follow by \eqref{ad.expan} and \eqref{adnXV}.

\begin{footnotesize}
	
\end{footnotesize}

\end{document}